 \documentclass[a4paper,reqno]{amsart}
\pdfoutput=1
\usepackage{amsmath}
\usepackage{amsthm}
\usepackage{amssymb}
\usepackage{amscd} %diagrams

\usepackage{bbm} %bold numbers
\usepackage{mathtools,mathrsfs} %nice calligraphy

\usepackage{graphics} %If you want to include postscript graphics

\usepackage{color}

\usepackage{textcomp}

\usepackage{verbatim} %DO NOT comment this line (comment environment)

\usepackage{enumitem}

\newtheoremstyle{newremark}
  {5pt}
  {5pt}
  {\rmfamily}
  {}
  {\rmfamily\bf}
  {.}
  {.5em}
  {}

\newtheorem{theorem}{Theorem}
\newtheorem{lemma}[theorem]{Lemma}
\newtheorem{corollary}[theorem]{Corollary}
\newtheorem{proposition}[theorem]{Proposition}

\theoremstyle{newremark}
\newtheorem{remark}[theorem]{Remark}
\newtheorem{definition}[theorem]{Definition}

\newtheorem*{definition*}{Definition} %no numbering for Theorem*
\newtheorem*{notations*}{Notations}

\numberwithin{theorem}{section}
\numberwithin{equation}{section}

\newcommand{\N}{\mathbf{N}} %natural numbers
\newcommand{\R}{\mathbf{R}} %real numbers

\newcommand{\Z}{\mathbf{Z}} %integers
\newcommand{\C}{\mathbf{C}} % complex numbers
\newcommand{\Ss}{\mathbf{S}}

\newcommand{\bgamma}{\boldsymbol{\gamma}}

\newcommand{\ba}{\mathbf{a}}

\newcommand{\hel} {
\hskip2.5pt{\vrule height7pt width.5pt depth0pt}
\hskip-.2pt\vbox{\hrule height.5pt width7pt depth0pt}
\, }

\newcommand{\restr}{\hel}

\def\XXint#1#2#3{{%
\setbox0=\hbox{$#1{#2#3}{\int}$}
\vcenter{\hbox{$#2#3$}}\kern-.5\wd0}}

\newcommand{\lt}{\left}
\newcommand{\rt}{\right}

\renewcommand{\leq}{\leqslant}
\renewcommand{\geq}{\geqslant}
\renewcommand{\subset}{\subseteq}
\renewcommand{\supset}{\supseteq}

\newcommand{\res}{\mathop{\hbox{\vrule height 7pt width .5pt depth 0pt
\vrule height .5pt width 6pt depth 0pt}}\nolimits}

\newcommand{\vhi}{\varphi}
\newcommand{\oo}{\infty}
\newcommand{\Om}{\Omega}
\newcommand{\eps}{\varepsilon}

\def\les{\lesssim}

\def\ba{\mathbf{a}}

\def\deg{\textup{deg}}
\def\W{\mathbb{W}}

\def\P{\mathrm{P}}
\def\p{\mathrm{p}}
\def\q{\mathrm{q}}

\def\curl{\mathrm{curl}\,}

\def\eh{{\eps_h}}
\def\bgamma{\boldsymbol{\gamma}}
\def\ovsigma{\overline{\sigma}}
\def\spt{\mathrm{spt}}
\def\loc{\mathrm{loc}}

\begin{document}

%=================
% TITLE AND AUTHOR
%=================

\title[\bf A G.L. model with topologically induced free discontinuities]{A Ginzburg-Landau model with\\ topologically induced free discontinuities}

\author{Michael Goldman}
\address{Laboratoire Jacques-Louis Lions (CNRS, UMR 7598), Universit\'e Paris Diderot, F-75005, Paris, France}
\email{goldman@math.univ-paris-diderot.fr}

\author{Beno\^it Merlet}
\address{Laboratoire P. Painlev\'e (CNRS UMR 8524), Universit\'e Lille 1, F-59655 Villeneuve d'Ascq Cedex, France}
\email{benoit.merlet@math.univ-lille1.fr}

\author{Vincent Millot}
\address{Laboratoire Jacques-Louis Lions (CNRS, UMR 7598), Universit\'e Paris Diderot, F-75005, Paris, France}
\email{millot@math.univ-paris-diderot.fr}

%=========
% ABSTRACT
%=========

\begin{abstract}
We study a variational model which combines features of the Ginzburg-Landau model in 2D and of the Mumford-Shah functional.  As in the classical Ginzburg-Landau theory, 
a prescribed number of point vortices appear in the small energy regime; the model allows for discontinuities, and the energy penalizes their length. The novel phenomenon here is that the vortices have a fractional degree $1/m$ with $m\geq2$ prescribed. Those vortices must be connected by line discontinuities to form clusters of total integer degrees. The vortices and line discontinuities are therefore coupled through a topological constraint. 
As in the Ginzburg-Landau model, the energy is parameterized by a small length scale $\eps>0$. We perform a complete $\Gamma$-convergence analysis of the model as $\eps\downarrow0$ in the small energy regime. We then study the structure of  minimizers of the limit problem. In particular, we show that the line discontinuities of a minimizer solve a variant of the Steiner problem. We finally prove that for small $\eps>0$,  the minimizers of the original problem have the same structure away from the limiting vortices.  
\end{abstract}

\maketitle

%==================
% TABLE OF CONTENTS
%==================

\tableofcontents

%%%%%%%%%%%%%%%%%%%%%%%%%%%%%%%%%%%%%%%%%%%%%%%%%%%%%%%%%%

\section{Introduction}\label{S1}
The purpose of this article is to study the asymptotic behavior of  a family of functionals combining aspects of both Ginzburg-Landau \cite{BBH,SS} and Mumford-Shah \cite{AFP,fuscoreview,lemenant} functionals in dimension two.  
Those extend the standard Ginzburg-Landau energy, and give rise to the formation of vortex points connected by line defects in the small energy regime. Interestingly,  vortices and  line defects are coupled through  topological constraints.
 
To be more specific, let us  introduce the mathematical context. We consider for $m\in \N$, $m\geq 2$, the group of $m$-th roots of unity 
$\mathbf{G}_m=\lt\{1,\mathbf{a},\mathbf{a}^2,\ldots,\mathbf{a}^{m-1}\rt\}$ with $\mathbf{a}:=e^{2i\pi/m}$. We are interested in maps taking values in  the quotient space $\C/\mathbf{G}_m$. We identify  $\C/\mathbf{G}_m$ with the round cone
\[
 \mathcal{N}:=\Big\{(z,t)\in\mathbb{C}\times\mathbb{R}: t=|z|\sqrt{m^2-1}\Big\}\subset \R^3\,
\]
by means of the map $\P  :\, \C\to \mathcal{N}$ defined as 
\[
 \P(z):=\frac{1}{m}\Big(\p(z),|z|\sqrt{m^2-1}\Big)\quad \text{with }  \p(z):=\frac{z^m}{|z|^{m-1}}\,.
\]
The map $\P$ induces an isometry between $\C/\mathbf{G}_m$ and  $\mathcal{N}$, and 
restricted to $\C\setminus\{0\}$ it defines a covering map of $\mathcal{N}\setminus\{0\}$ of degree $m$. For a given open set $\Omega$ and $p\ge 1$  we can thus say that $u\in W^{1,p}(\Omega, \C/\mathbf{G}_m)$ if $\P(u)\in W^{1,p}(\Omega,\mathcal{N})$ (where we say that a map $v\in W^{1,p}(\Omega,\mathcal{N})$ if $v$ takes values in $\mathcal{N}$ and $v\in W^{1,p}(\Omega,\R^3)$). 
\vskip3pt

\begin{figure}[h]
%\psfrag{m}{$m=3$}
%\psfrag{N}{$\mathcal{N}$} 
%\psfrag{S}{{\blue $\mathcal{S}$}}
%\psfrag{x}{$u_1$}\psfrag{y}{$u_2$}\psfrag{z}{$u_3$}
%\psfrag{P}{$\P(u_j)$} \psfrag{p}{$\p(u_j)/m$}
%\psfrag{a}{{\scriptsize $2\pi/3$}}
\includegraphics[width=11cm]{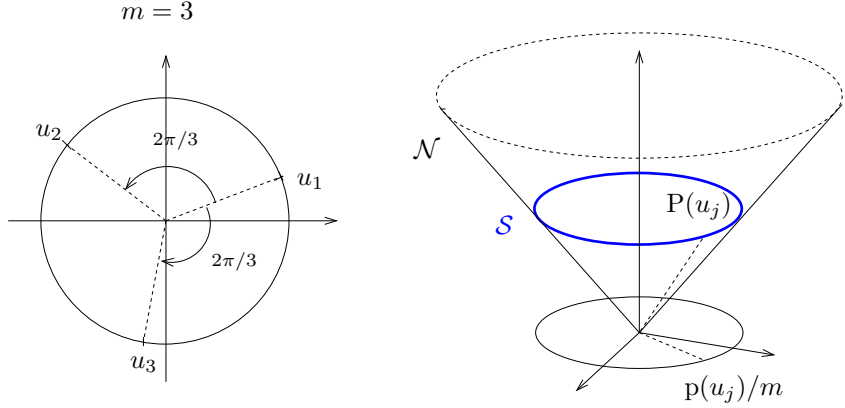} %\includegraphics[width=11cm]{figures/ConeN.eps} 
\caption{The cone $\mathcal{N}$ and the projection $\P$. $\P(u_1)=\P(u_2)=\P(u_3)$.\label{figN}}
\end{figure}

For a simply connected smooth bounded domain  $\Omega\subset\R^2$ and a ``small'' parameter $\eps>0$, the standard Ginzburg-Landau energy over $\Omega$ of a vector valued $W^{1,2}$-map reads 
$$E_\eps(u):= \frac{1}{2}\int_{\Omega} |\nabla u|^2 +\frac{1}{2\eps^2}(1-|u|^2)^2\,dx\,.$$
Here, the main functional under investigation is defined for $u\in SBV^2(\Omega)$ satisfying  the constraint $\P(u)\in W^{1,2}(\Omega;\mathcal{N})$ by 
\begin{equation}\label{intro:Feps}
 F_\eps^0(u):= E_\eps\big(\P(u)\big)\,
 +{\mathcal H}^1(J_u)\,,
\end{equation}
where $J_u$ denotes the jump set of $u$  (see \cite{AFP} and Section~\ref{S23} below). We stress that $F_\eps^0$ extends $E_\eps$, that is $F_\eps^0(u)=E_\eps(u)$ whenever $u\in W^{1,2}(\Omega)$, which comes from the isometric character of $\P$. In the same way $F_\eps^0$ appears as a Mumford-Shah type functional since 
$$F_\eps^0(u)=\frac{1}{2}\int_{\Omega} |\nabla u|^2 +\frac{1}{2\eps^2}(1-|u|^2)^2\,dx +{\mathcal H}^1(J_u)\,,$$
where $\nabla u$ denotes the absolutely continuous part of the measure $Du$. The constraint  $\P(u)\in W^{1,2}(\Omega;\mathcal{N})$ rephrases the fact that the functional is restricted to the class
$\big\{u\in SBV^2(\Omega): u^+/u^-\in \mathbf{G}_m\text{ on $J_u$}\big\}$. 
In particular, only specific discontinuities in the orientation are allowed.  The case $m=2$, which consists in identifying $u$ and $-u$, is of special interest as it appears in many physical models, see Section \ref{secmotiv} below. 
\vskip3pt

We also consider an Ambrosio-Tortorelli regularization of \eqref{intro:Feps} where the jump set $J_u$ is (formally) replaced by the zero set $\{\psi\sim 0\}$ of some scalar phase field function $\psi$, and the length ${\mathcal H}^1(J_u)$ by a suitable energy of $\psi$. We introduce a second small parameter $\eta$ and consider
for $u\in L^2(\Om)$ and $\psi\in W^{1,2}(\Om;[0,1])$ satisfying $\P(u)\in W^{1,2}(\Omega;\mathcal{N})$ { and} $u\psi\in W^{1,2}(\Omega)$, the functional 
\begin{equation}\label{intro:Feteps}
 F_\eps^\eta(u,\psi):=
 E_\eps\big(\P(u)\big) 
 +\frac{1}{2}\int_{\Omega} \eta |\nabla \psi|^2+\frac{1}{\eta} (1-\psi)^2\,dx\,.
\end{equation}
Compared to the original Ambrosio-Tortorelli  functional  \cite{AmbTor90,AmbTor}, $u$ and $\psi$ are only coupled  through the constraint $u\psi\in W^{1,2}(\Omega)$, and { not} in the functional itself. As for $F^0_\eps$, the functional $ F_\eps^\eta$ extends $E_\eps$ in the sense that $F_\eps^\eta(u,1)=E_\eps(u)$ whenever $u\in W^{1,2}(\Omega)$.
\vskip3pt

We aim to study low energy states (in particular minimizers) of  the functionals $ F_\eps^0$ and $F_\eps^\eta$ under  Dirichlet boundary conditions of the form $u=g$ on $\partial \Omega$ for a prescribed smooth $g\in C^\oo(\partial \Omega;\Ss^1)$. 
Concerning $ F_\eps^0$, we work in the class $\mathcal{G}_g(\Omega)$ of maps satisfying  $\P(u)=\P(g)$ on $\partial \Omega$. Then, we penalize possible deviations from $g$ on $\partial \Omega$ by considering the modified energy 
$$F^0_{\eps,g}(u):= F_\eps^0(u)+\mathcal{H}^1\big(\{u\not=g\}\cap\partial\Omega\big)\,.$$
Notice that such a penalization is necessary in order to have lower semi-continuity of the functional (see for instance \cite{GMS}). 
For  the functional $F_\eps^\eta$, we restrict ourselves to  admissible pairs $(u,\psi)$ satisfying $u\psi=g$ and $\psi=1$ on~$\partial\Omega$, and write  $\mathcal{H}_g(\Omega)$ the corresponding class. 
In this setting, the functionals $F^0_{\eps,g}$ and $F_\eps^\eta$ still extend $E_\eps$ restricted to $W^{1,2}_g(\Omega)$, so that 
\begin{equation}\label{compminval}
\min_{\mathcal{G}_g(\Omega)} F^0_{\eps,g}\leq \min_{W^{1,2}_g(\Omega)} E_\eps\quad\text{and} \quad\min_{\mathcal{H}_g(\Omega)} F^\eta_{\eps,g}\leq \min_{W^{1,2}_g(\Omega)} E_\eps \,.
\end{equation}
As in the classical Ginzburg-Landau theory \cite{BBH}, we assume that the winding number (or degree) is strictly positive, i.e., 
\[
d:=\deg(g,\partial \Omega)>0\,.
\]
In this way, $g$ does not admit a continuous $\Ss^1$-valued extension to $\Omega$. This topological obstruction is responsible for the formation of {\it vortices} (point singularities)
in any  configuration of small energy  $E_\eps$  as $\eps\to 0$, and the minimum value of  $E_\eps$ over $W^{1,2}_g$ is given by $\pi d |\log \eps|$ at first order. 
In view of \eqref{compminval}, creating   discontinuities in the orientation  may lead to configurations of  smaller energy. Indeed, direct constructions of competitors show that the minimum value of $F^0_{\eps,g}$ or $F_\eps^\eta$ is less than $\frac{\pi d}{m}|\log \eps|$ at first order,
and thus (almost) minimizers must have line singularities (or "diffuse" line singularities for $F_\eps^\eta$), at least for $\eps$ (and $\eta$) small enough.

\subsection{Heuristics} 
The starting point is the identity 
\[
E_\eps\big(\P(u)\big)=\frac{1}{m^2}E_\eps\big(\p(u)\big)+ \frac{m^2-1}{m^2}E_\eps\big(|\p(u)|\big)\,.
\] 
Following the standard theory of the Ginzburg-Landau functional \cite{BBH,SS}, one may expect that for  configurations $u$ of small energy, the leading term is $\frac{1}{m^2}E_\eps\big(\p(u)\big)$, and
that $\p(u)$ has (classical) Ginzburg-Landau energy $E_\eps$  close to the one of the minimizers under the boundary condition $\p(u)=\p(g)$ on $\partial\Omega$. Since $\p(g)=g^m$, its topological degree equals $md$, and $\p(u)$ should have $md$ distinct vortices of degree $+1$, i.e., $md$ distinct points $x_k$ in $\Omega$ such that  $\p(u)(x_k)=0$ and   
\[
\p(u)(x)\, \sim\, \alpha_k\dfrac{x-x_k}{|x-x_k|}\quad\mbox{ for $\eps \ll |x-x_k|\ll 1$ and some constant $\alpha_k\in \Ss^1$}\,.
\]
In terms of $E_\eps$, the energetic cost of each vortex is $\pi |\log\eps|$ at leading order, and therefore $E_\eps\big(\P(u)\big)$ should be less than  $\frac{\pi d}{m}|\log \eps|$, again at leading order.  This discussion led us to consider the energy regimes 
\begin{equation}\label{energregintr}
F^0_{\eps,g}(u)\leq \frac{\pi d}{m} |\log\eps|+ O(1) \quad\text{and}\quad F^\eta_\eps(u,\psi)\leq  \frac{\pi d}{m} |\log\eps| +O(1)
\end{equation}
for $u\in\mathcal{G}_g(\Omega)$ or $(u,\psi)\in\mathcal{H}_g(\Omega)$, respectively. Once again, it corresponds to the energy regime of $md$ vortices of degree $+1$ in the variable $\p(u)$. By an elementary topological argument, one can see that any pre-image by $\p$ of $\frac{x-x_k}{|x-x_k|}$ must have at least one discontinuity line departing from $x_k$, and has a (formal) winding number around $x_k$ equal to $1/m$ (in other words, the phase has a jump of $2\pi/m$ around $x_k$).  For this reason, any configuration $u$ satisfying \eqref{energregintr} must be discontinuous. In the sharp interface case~\eqref{intro:Feps}, we actually expect that each connected component of the jump set $J_u$ connects $mk$ vortices for some $k\in \{1,\cdots,d\}$, since the winding number around any such connected component must be an integer.  A similar picture should hold in the diffuse case~\eqref{intro:Feteps} with $J_u$ replaced by the zero set $\{\psi=0\}$. 
The energy associated with discontinuities is their length (or diffuse length), and there should be a competition between this term which favors clustered vortices and the so-called renormalized energy from Ginzburg-Landau theory which is a repulsive (logarithmic) point interaction.

\subsection{Motivation}\label{secmotiv}
Our original motivation for studying the functionals \eqref{intro:Feteps} and \eqref{intro:Feps} stemmed from the analysis of the defect patterns observed in the so-called ripple or $P_{\beta'}$
phase in biological membranes such as  lipid bilayers \cite{Sackmann,belfeiglev,LenzSchmidt,LubMkintosh,ruppelsackman}. In this phase, which is intermediate between the gel and the liquid phase, periodic corrugations are
observed at  the surface of the membranes (see \cite{ruppelsackman} for instance).
\begin{figure}
 \resizebox{8.2cm}{!}{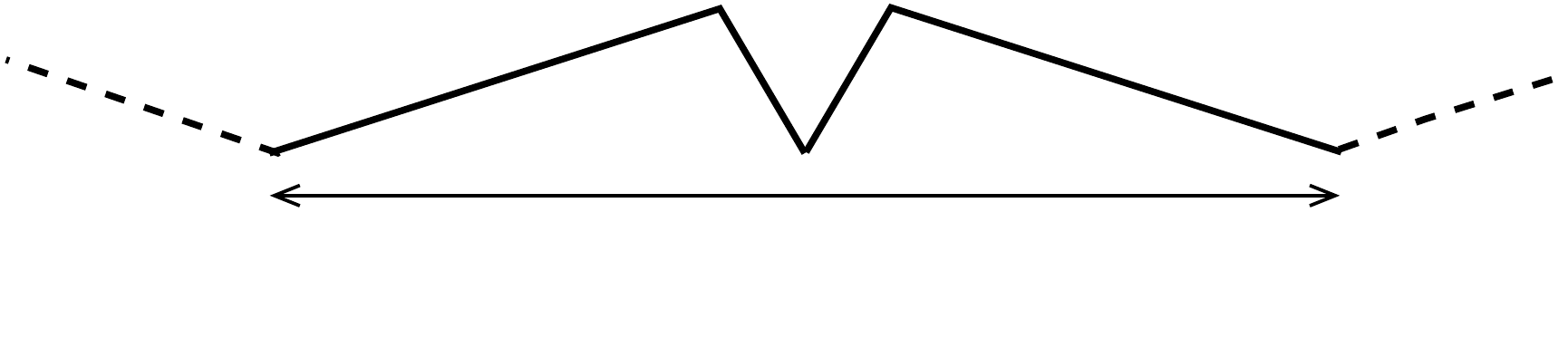} 
 \hspace{0.5cm}
  \resizebox{7.0cm}{!}{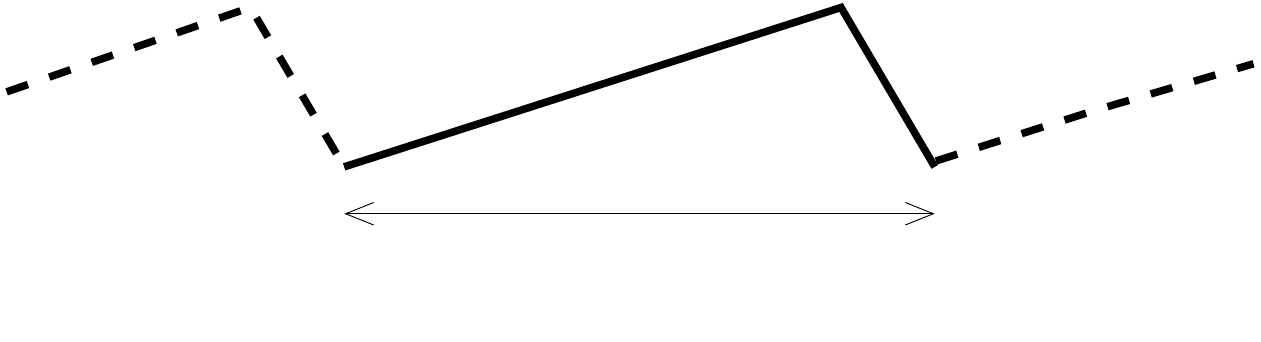}
   \caption{Top: profile of the $\Lambda-$phase. Bottom: profile of the $\Lambda/2-$phase.} \label{Lambdaphase}
 \end{figure}
Two different kinds of periodic sawtooth profiles are observed. A symmetric  one and an asymmetric one respectively
called  $\Lambda$ and  $\Lambda/2-$phases (see Figure~\ref{Lambdaphase} for a schematic representation of a cross-section). 
In the asymmetric phase, only defects of integer degree are allowed while
in the symmetric phase half integer degree vortices are also permitted. Since two vortices of degree $1/2$ have an energetic cost of order $\frac{\pi}{2} |\log \eps|$ (where $\eps$ is the lengthscale of the vortex) while a vortex of 
degree $1$ has a cost of order $\pi |\log \eps|$, it is expected that even in the regime where the $\Lambda/2-$phase is favored (which happens for nearly flat membranes), 
a phase transition occurs around the defects with the nucleation of  a small island of $\Lambda-$phase leading to the formation of  two vortices of degree $1/2$ (see Figure \ref{fig:demivortex}). 
In the model proposed by \cite{belfeiglev}, the order parameter is given by $f(\vhi)$, where $f$ is a fixed profile (corresponding to the one on the right part of Figure
\ref{Lambdaphase}) and $\vhi$ is the phase modulation. Their functional corresponds to $F_\eps^\eta$, for $\eps=\eta$, $m=2$ and $u=\nabla \vhi$ (so that $u$ represents the local 
speed at which the profile $f$ is modulated). In \cite{belfeiglev}, the authors further  argue that the constraint of $u$ being a gradient can be relaxed so that we recover completely our model. 
 \begin{figure}\begin{center}
 \resizebox{8.cm}{!}{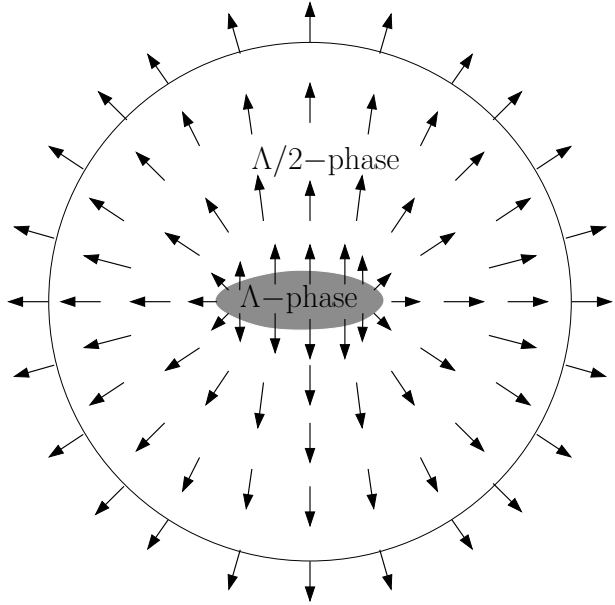} 
   \caption{Creation of two vortices of degree $1/2$.} \label{fig:demivortex}
 \end{center}
 \end{figure}
 
We also point out that~ \eqref{intro:Feps} and ~\eqref{intro:Feteps} have connections with many other models appearing in the literature. As an example, we can mention the issue of orientability of
Sobolev vector fields into~$\R\mathbf{P}^1$, see \cite{BallZar}. 
More generally, our functionals resemble the ones suggested recently to model liquid crystals where both points and lines singularities appear, see \cite{Bedford}. Similarly to \cite{BallZar}, a central issue here is to find square roots (and more generally $m-$th roots) of $W^{1,p}$-functions into $\Ss^1$  (see \cite{IL2017}),  and this is intimately related to the question of lifting of Sobolev functions into  $\Ss^1$, see \cite{BoBrMi,Demengel,BMP,Merletlift,IgnDav}. 

While completing this article, we have been aware of the work \cite{BadCicLuPo}, where  the authors perform an analogous $\Gamma$-convergence analysis
for a discrete model, obtaining in the continuous  limit almost the same functional as ours. These authors were  motivated by applications to liquid crystals, micromagnetics, and crystal plasticity, and we refer to their introduction for more references on the physical literature.

\subsection{Main results}
Our first main theorem is
a $\Gamma$-convergence result in the energy regime \eqref{energregintr} (we refer to \cite{DalMaso,Braides} for a complete exposition on $\Gamma$-convergence theory). To describe the limiting functional, we need to introduce the following objects. First, set $\mathcal{A}_d$ to be the family of all atomic measures of the form $\mu=2\pi\sum_{k=1}^{md}\delta_{x_k}$, for some $md$ distinct points $x_k\in \Omega$.  To $\mu\in \mathcal{A}_d$, we associate the so-called {\sl canonical harmonic map} $v_\mu$ defined by 
$$v_\mu(x):= e^{i\vhi_\mu(x)}\prod _{k=1}^{md}\frac{x-x_k}{|x-x_k|}\quad\text{with} \quad\begin{cases}
\Delta\vhi_\mu=0 &\text{in $\Omega$}\,,\\
v_\mu=g^m & \text{on $\partial\Omega$}\,.
\end{cases}$$
In turn, the {\sl renormalized energy} $\mathbb{W}(\mu)$ can be defined as the finite part of the energy of $v_\mu$, i.e.,  
$$\mathbb{W}(\mu):=\lim_{r\downarrow0} \left\{\frac{1}{2} \int_{\Omega \setminus B_r(\mu)}|\nabla v_\mu|^2\,dx-\pi md|\log r|\right\}\,,$$
and we refer to \eqref{formulW} for its explicit expression. 

We provide below a concise version of the $\Gamma$-convergence result, complete  statements can be found in Theorem \ref{Gammadiff} and  Theorem \ref{Gammasharp}.  

\begin{theorem}\label{intro:theoGamma}
The functionals  $\{F_{\eps,g}^0 -\frac{\pi d}{m} |\log \eps|\}$ and $\{F_\eps^\eta -\frac{\pi d}{m} |\log \eps|\}$ (respectively restricted to $\mathcal{G}_g(\Omega)$ and $\mathcal{H}_g(\Omega)$)  $\Gamma$-converge in the strong $L^1$-topology as $\eps\to 0$ and $\eta\to 0$ to the functional 
 \begin{equation*}
  F_{0,g}(u):= \frac{1}{2m^2}\int_{\Omega} |\nabla \vhi|^2\,dx+\frac{1}{m^2} \W(\mu)+md\bgamma_m +{\mathcal H}^1\big(J_u) +{\mathcal H}^1(\{u\neq g\}\cap \partial \Omega\big)
 \end{equation*}
 defined for $u\in SBV(\Omega;\Ss^1)$ such that $u^m=e^{i\vhi} v_\mu$  for some $\mu\in \mathcal{A}_d$ and $\vhi\in W^{1,2}(\Omega)$ satisfying $e^{i\varphi}=1$ on $\partial \Omega$. 
The constant $\bgamma_m$, referred to as {\sl core energy} (see \eqref{defbgammVnew}), only depends on $m$.  
\end{theorem}

We point out that there is of course a compactness result companion to Theorem \ref{intro:theoGamma}. Namely, if a sequence $\{u_\eps\}$ satisfies the energy bound  \eqref{energregintr}, and is uniformly bounded in $L^\infty(\Omega)$, then $\{u_\eps\}$ converges up to a subsequence in $L^1(\Omega)$, and $\{\p(u_\eps)\}$ converges (again up to a subsequence) in the weak $W^{1,p}$-topology  for every $p<2$. 
As can be expected, the proof of Theorem \ref{intro:theoGamma} 
combines ideas coming from the study of the Ginzburg-Landau functional \cite{BBH,SS,ColJer,AlicPon,XinLin,JerSon}, together with ideas  from free discontinuities problems \cite{AFP, braidesfree, BCS,AmbTor}. 
Concerning the compactness part, we have included complete proofs to provide  a rather self-contained exposition. 
Although some estimates (such as the $W^{1,p}$ bound, see Lemma~\ref{lemboundWq})  are certainly known to the Ginzburg-Landau community
(see for instance \cite{ColJer,XinLin}), they have never been used in the context of $\Gamma$-convergence.     
The $\Gamma$-$\liminf$ inequality is a relatively standard combination of techniques developed in \cite{ColJer,AlicPon,BCS}, while the construction of recovery sequences is a much more
delicate issue. The main difficulty comes  from  the the constraint $u^m= e^{i\vhi} v_\mu$, which prevents us to apply directly the existing approximation results by functions with a smooth jump set,  
see e.g. \cite{CorToa,BraiCP,PratFusDep,BelChGol}). Our approach uses a (new) regularization technique (see Lemma~\ref{Alternative2}) which is somehow reminiscent of \cite{AmbTor90} and could be of independent  interest. Another difficulty comes from the optimal profile problem defining the core energy $\bgamma_m$. The underlying minimization problem involves the Ginzburg-Landau energy of $\mathcal{N}$-valued maps, and one has to find almost minimizers which can be lifted into  $\C$-valued maps in $SBV^2$, see Section \ref{secoptprof}.  
\vskip3pt

The $\Gamma$-convergence result applies to minimizers of either $F^0_{\eps,g}$ or $F^\eta_\eps$ (whose existence is proven in Theorems \ref{exist} \& \ref{existsharp}). It shows that they 
converge in $L^1(\Omega)$ to a minimizer $u$ of $F_{0,g}$.  Our second main result deals with the characterization of such minimizer $u$. It is based on the following observations. 
First, from the explicit form of $F_{0,g}$, it follows that $\varphi=0$ in the representation $u^m=e^{i\varphi}v_\mu$. In particular, $u$ can be characterized as a solution of the minimization problem
$$\min\Bigg\{\frac{1}{m^2}\mathbb{W}(\mu)+{\mathcal H}^1\big(J_u) +{\mathcal H}^1(\{u\neq g\}\cap \partial \Omega\big): u\in SBV(\Omega;\Ss^1)\,,\;u^m=v_\mu\text{ for some }\mu\in\mathcal{A}_d\Bigg\}\,.  $$ 
In turn, this later can be equivalently rewritten as 
$$\min_{\mu\in\mathcal{A}_d} \min \Bigg\{\frac{1}{m^2}\mathbb{W}(\mu)+{\mathcal H}^1\big(J_u) +{\mathcal H}^1(\{u\neq g\}\cap \partial \Omega\big): u\in SBV(\Omega;\Ss^1)\,,\;u^m=v_\mu\Bigg\}\,.$$
As a consequence, fixing $\mu\in\mathcal{A}_d$ and solving 
$$ L(\mu):=\min \Bigg\{{\mathcal H}^1\big(J_u) +{\mathcal H}^1(\{u\neq g\}\cap \partial \Omega\big): u\in SBV(\Omega;\Ss^1)\,,\;u^m=v_\mu\Bigg\}\,,$$
we are left with a finite dimensional problem to recover the minimizers of $F_{0,g}$.  

Given $\mu\in\mathcal{A}_d$, we compare in Theorem \ref{intro:regmin}  below the minimization problem $L(\mu)$ with the following variant of the Steiner problem (see e.g. \cite{GilPol}): 
\begin{multline*}
\Lambda(\mu):=\min\Big\{\mathcal{H}^1(\Gamma): \Gamma\subset\overline\Omega \text{ compact with } {\rm spt}\, \mu \subset \Gamma \\
\textrm{ and every connected component $\Sigma$ satisfies }
{\rm Card}(\Sigma\cap{\rm spt}\,\mu)\in m\mathbb{N}\Big\}\,. 
\end{multline*}
We shall see that any minimizer $\Gamma$ of $\Lambda(\mu)$ is made of at most $d$ disjoint Steiner trees, i.e., 
connected trees made of a finite union of segments meeting either at points of ${\rm spt}\,\mu$, or at triple junction making a $120^\circ$ angle.  From now on, when talking about triple junctions we always implicitly include this condition on the angles. 
\vskip3pt

Our second main result is the following theorem, in which we assume $\Omega$ to be convex (to avoid issues at the boundary). 

\begin{theorem}\label{intro:regmin} 
Assume  that $\Omega$ is convex.  For every $\mu\in \mathcal{A}_d$,  $L(\mu) = \Lambda(\mu)$. 
Moreover, if $u$ is a minimizer for $L(\mu)$, then its jump set $J_u$ is a minimizer for $\Lambda(\mu)$, $u\in C^\infty(\overline\Omega\setminus J_u)$, and $u=g$ on $\partial\Omega$.  Vice-versa, 
if $\Gamma$ is a minimizer for $\Lambda(\mu)$, then there exists a minimizer $u$ for $L(\mu)$ such that~$J_u=\Gamma$.  
\end{theorem}

To complete the picture, we shall give several examples illustrating the fact that the geometry of minimizers for $\Lambda(\mu)$ strongly depends on $m$, $d$, and the location of ${\rm spt}\,\mu$.
In the case $m=2$, a minimizer for $\Lambda(\mu)$ is always given by a disjoint union of $d$ segments connecting the points of ${\rm spt}\,\mu$ (see Proposition \ref{propminconnect}). 
However, for $m\geq 3$ and $d\geq2$, minimizers are  not always the disjoint union of $d$ Steiner trees containing exactly $m$ vortices (see Proposition \ref{propexmin} and Proposition \ref{propexmin2}).
\vskip3pt

In our third and last main result, we use the characterization of the minimizers of $F_{0,g}$ provided by Theorem \ref{intro:regmin} to show that for $\eps>0$ small enough, minimizers of $F^0_{\eps,g}$ have essentially the same structure away from the limiting vortices.

\begin{theorem}\label{intro:theoregeps}
Assume that $\Omega$ is convex. Let $\eps_h\to 0$, and let $u_h$ be a minimizer of $F_{\eh,g}^0$ over ${\mathcal G}_g(\Omega)$. Assume that $u_h\to u$ in $L^1(\Omega)$ as $h\to\infty$ for some minimizer $u$ of $F_{0,g}$. Setting $\mu:=\curl j(u^m)$, for every $\sigma>0$ small enough and every $h$ large enough,  the following holds:  
\begin{enumerate}
\item[(i)] $J_{u_h}\setminus B_\sigma(\mu)$ is a compact subset of $\Omega\setminus  B_\sigma(\mu)$ made of finitely many  segments,  meeting by three at an angle of  $120^{\circ}$ (i.e., triple junctions). 
\vskip3pt
\item[(ii)] $u_h\in C^\infty\big(\overline \Omega \setminus(B_\sigma(\mu)\cup J_{u_h}) \big)$ and $u_h=g$ on $\partial \Omega$. 
\end{enumerate}
In addition, 
\begin{enumerate}
\item[(iv)] $ J_{u_h}$ converges in the Hausdorff distance to $J_{u}$. 
\vskip3pt
\item[(v)] $u_h\to u$ in $C^k_{\rm loc}(\Omega\setminus J_{u})\cap C^{1,\alpha}_{\rm loc}(\overline \Omega\setminus J_{u})$  for every $k\in \mathbb{N}$ and $\alpha\in(0,1)$. 
\end{enumerate}
\end{theorem} 
In proving Theorem \ref{intro:theoregeps}, we actually show a  stronger result that we now briefly describe (see Theorem \ref{theoregeps}, Remark~\ref{betdescripthm3}, and Section \ref{sec:sketch}). In each (sufficiently small) ball $B_r(x)\subset \Omega\setminus  B_\sigma(\mu)$ and $\eps$ small enough,  $u_\eps$ is bounded away from zero, and it can
be decomposed as $u_\eps=\phi_\eps w_\eps $ where $\phi_\eps\in SBV^2(B_r(x))$ and 
$w_\eps$ is minimizing the classical Ginzburg-Landau energy $E_\eps(\cdot,B_r(x))$ with respect to its own boundary condition (and as a consequence, $w_\eps$ is smooth).     
The proof of this decomposition relies on the energy splitting discovered by Lassoued \& Mironescu \cite{LasMi}. Combined with the classical Wente estimate \cite{Wente,BrezCor}, it leads to a  lower expansion of the energy of the form 
\[
 F_\eps^0(u_\eps,B_r(x))\geq E_\eps(w_\eps,B_r(x)) +\frac{1}{\alpha}\left(\int_{B_r(x)} |\nabla \phi_\eps|^2\,dx +\alpha{\mathcal H}^1\big(J_{\phi_\eps}\cap B_r(x)\big)\right)\,,
\]
for some constant $\alpha>0$ (see Proposition \ref{propLM}). Using suitable competitors, we  deduce that  $\phi_\eps$ is a Dirichlet minimizer the Mumford-Shah functional in $B_r(x)$. Applying the calibration results of \cite{AlBouDal,Mora}, we infer that $\phi_\eps$ takes values into the finite set ${\bf G}_m$, reducing the problem to a minimal partition problem in $B_r(x)$.  The classical regularity results on two dimensional minimal clusters 
then yield the announced  geometry of the jump set. 
\vskip5pt

The paper is organized as follows. Section~\ref{S2} is devoted to a full set of preliminary results. First, we present some fine properties of the $BV$-functions under investigation, 
and then we prove existence of minimizers for $F_\eps^\eta$ and $F^0_{\eps,g}$. In a third part, we provide all the material and results concerning the Ginzburg-Landau energy that we shall use. The  $\Gamma$-convergence result of Theorem \ref{intro:theoGamma} is the object of Section~\ref{S3}. In Section~\ref{S4}, we prove Theorem \ref{intro:regmin}  and give the aforementioned examples of $\Lambda(\mu)$-minimizers.  In the last Section \ref{S5}, we return to the analysis of minimizers of $F^0_{\eps,g}$, and prove Theorem~\ref{intro:theoregeps}.

%%%%%%%%%%%%%%%%%%%%%%%%%%%%%%%%%%%%%%%%%%%%%%%%%%%%%%%
%%%%%%%%%%%%%%%%%%%%%%%%%%%%%%%%%%%%%%%%%%%%%%%%%%%%%%%
   								       					 %%%%%%%%%%%%%%%%%%%%
\section{Preliminaries} \label{preliminaries}\label{S2}	                         %%%%%%%%%%%%%%%%%%%%
							 						%%%%%%%%%%%%%%%%%%%
%%%%%%%%%%%%%%%%%%%%%%%%%%%%%%%%%%%%%%%%%%%%%%%%%%%%%%%
%%%%%%%%%%%%%%%%%%%%%%%%%%%%%%%%%%%%%%%%%%%%%%%%%%%%%%%

\subsection{Conventions and notation}
\label{S21}
Throughout the paper we identify the complex plane $\C$ with~$\R^2$. 
We say that a property holds a.e. if it holds outside a set of Lebesgue measure zero.

\begin{itemize}
\item For $a,b\in\R^2$, we write $a\wedge b:=\det(a,b)$;
\item For $a\in\R^2$ and $M=(b_1,\ldots,b_n)\in\mathcal{M}_{2\times n}(\R)$, we write  
$$a\wedge M:=\,^t(a\wedge b_1,\ldots, a\wedge b_n)\in\R^n\,;$$
\item For $M\in\mathcal{M}_{d\times n}(\R)$, we write $|M|:=|{\rm tr}(M\,^tM)|^{1/2}$;
\item For $a=(a_1,a_2)\in \R^2$ we let $a^\perp:=(-a_2,a_1)$
\item for a set $\Omega\subset \R^2$, we call $\nu$ its external normal and $\tau$ its tangent chosen so that $(\nu,\tau)$ is a direct basis (in particular $\nu^\perp=\tau$ and $\partial \Omega$ is oriented counterclockwise); 
\item The distributional derivative is denoted by $Df$;
\item For $v\in \R^n$, we let $\partial_v f:=D f (v)$ be the partial derivative of $f$ in the direction $v$ and if $v=e_l$ is a vector of the canonical basis of $\R^n$ then we simply write $\partial_l f:=\partial_{e_l} f$;
\item $\nabla f=(\partial_l f_k)_{k,l}$ is the Jacobian matrix of the vector valued function $f$;
\item For $j=(j_1,j_2)$, we denote by $\curl j:= \partial_1 j_2-\partial_2 j_1$ the rotational of $j$;
\item For $A\subset\R^n$, we denote by $B_r(A)$ the tubular neighborhood of $A$ of radius $r$. For a measure $\mu$, we simply write $B_r(\mu):= B_r(\spt\,  \mu)$;  
\item In most of the paper, we work with $\Omega$ a given bounded open and simply connected set. Nevertheless, since in sections \ref{S4} and \ref{S5} we will require that $\Omega$ is convex, we will repeat at the beginning of each section the hypothesis we are making on $\Omega$;
\item We shall not relabel subsequences if no confusion arises. 
\end{itemize}

\subsection{Finite subgroups of $\Ss^1$ and isometric cones.}
\label{S22}
Given an integer $m\geq 2$, we denote by $\mathbf{G}_m$ the subgroup of $\Ss^1$ made of all $m$-th roots of unity, i.e., 
\[
\mathbf{G}_m=\big\{1,\mathbf{a},\mathbf{a}^2,\ldots,\mathbf{a}^{m-1}\big\}\quad \text{with }\mathbf{a}:=e^{2i\pi/m}\,.
\]
We consider the quotient space $\C/\mathbf{G}_m$ endowed with the canonical distance
\[
{\rm dist}\big([z_1],[z_2]\big):=\min_{z_1\in[z_1],\,z_2\in[z_2]} |z_1-z_2|=\min_{k=0,\ldots,m-1}|z_1-\mathbf{a}^kz_2|\,,
\]
where $[z]$ is the equivalence class of $z\in\C$. We note that $\C/\mathbf{G}_m$ is isometrically embedded into $\R^3\simeq\C\times\R$ by means of the Lipschitz mapping $\P:\C\to \R^3$ given by 
$$\P(z):=\frac{1}{m}\Big(\p(z),|z|\sqrt{m^2-1}\Big)\quad \text{where }  \p(z):=\frac{z^m}{|z|^{m-1}} \,.$$
In this way we  identify $\C/\mathbf{G}_m$ with the round cone of $\R^3$,
\[
\mathcal{N}:=\P(\C)=\Big\{(x,t)\in\R^2\times\R: t=|x|\sqrt{m^2-1}\Big\}\,, 
\]
and one has ${\rm dist}\big([z_1],[z_2]\big)={\rm d}_{\mathcal{N}}\big(\P(z_1),\P(z_2)\big)$ for every $z_1,z_2\in \C$, where ${\rm d}_{\mathcal{N}}$ denotes the geodesic distance on $\mathcal{N}$ induced by the Euclidean metric (in particular, $|\P(z)|=|z|$ for every $z\in\C$). 
Similarly, $\Ss^1/\mathbf{G}_m$ coincides with the horizontal circle
\[
\mathcal{S}:=\big\{(x,t)\in\mathcal{N}: |x|=1/m, t=\sqrt{1-1/m^2}\big\}=\P(\Ss^1) \,.
\]
Note that in the case $m=2$, $\mathcal{S}\simeq \Ss^1/\{\pm1\}$ is the real projective line $\R\bf{P}^1$. Finally, we point out that $\P$ is smooth away from the origin, and since   $\P$ is isometric, 
 \begin{equation}\label{modnabl}
 |\nabla\P(z)v|=|v| \quad\text{for every $v\in\R^2$ and every $z\in\C \setminus\{0\}$}\,,
 \end{equation}
where $\nabla\P(z)\in\mathcal{M}_{3\times2}(\R)$ is the differential of  $\P$ at $z$ represented in real coordinates. Similarly, we write 
$\nabla\p(z)\in \mathcal{M}_{2\times2}(\R)$ for the differential of $\p$ at $z$.

%%%%%%%%%%%%%%%%%%%%%%%%%%%%%%%%%%%%%%%%%%%%%%%%%%%%%%%

\subsection{$BV$ and $SBV$ functions,  weak Jacobians}
\label{S23}
Concerning functions of bounded variations, their  fine properties, and standard notations, we refer to~\cite{AFP}. Let us briefly introduce the main properties and definitions used in the paper. For an open subset $\Omega$ of $\R^2$, we first recall that $BV(\Omega,\R^q)$ is the space of functions of bounded variation in $\Omega$, i.e., functions $u\in L^1(\Omega,\R^q)$ for which the distributional derivative $Du$ is a finite (matrix valued) Radon measure on $\Omega$. 
We recall that for a function $u\in BV(\Omega,\R^q)$, we have the following decomposition
\[
 Du= \nabla u dx + D^j u+ D^c u\,,
\]
where 
\begin{equation}\label{formDju}
D^ju:=(u^+-u^-)\otimes\nu_u\,\,\mathcal{H}^1\res J_u\,.
\end{equation}
The functions $u^\pm$ denote the traces of $u$ on the jump set  $J_u$ which is a countably $\mathcal{H}^1$-rectifiable set. Since all the properties we will consider are oblivious to modifications of $J_u$ on sets of zero $\mathcal{H}^1$ measure, we shall not distinguish between $J_u$ and the singular set of $u$ (usually denoted as $S_u$). 
In particular, when $J_u$ is regular or a finite union of polygonal curves, we will also not distinguish between $J_u$ and its closure so that we shall often consider it as a compact set. Analogously, for sets $E$ of finite perimeter, i.e.,  such that $\chi_E\in BV(\Omega)$, we simply denote by $\partial E$  the reduced boundary. 
\vskip5pt

 The space $SBV(\Omega,\R^q)$ is defined as the subspace of $BV(\Omega,\R^q)$ made of functions $u$ satisfying $D^cu\equiv 0$. For a finite exponent $p\geq 1$, the subspace $SBV^p(\Omega,\R^q)\subset SBV(\Omega,\R^q)$ is defined as 
\[
SBV^p(\Omega,\R^q):=\Big\{u\in SBV(\Omega,\R^q): \nabla u\in L^p(\Omega)\text{ and }\mathcal{H}^{1}(J_u)<\infty\Big\}\,.
\]
\vskip3pt

\begin{remark}[pre-Jacobian]
For a smooth function $u$, we define the pre-Jacobian of $u$ as 
\[j(u):= u \wedge \nabla u=2 \det (\nabla u),\] 
which also writes $j(u)={\rm Re} (iu\nabla \bar u)$ in complex notation. Notice that if $u=\rho e^{i\theta}$ for some smooth functions $\rho$ and $\theta$,  then $j(u)=\rho^2 \nabla \theta$ so that $j(u)$ measures  the variation of the phase.
In particular, if $\Omega$ is simply connected and $u$ takes values into $\Ss^1$, then $\curl j(u)=0$ and we can  write $j(u)=\nabla \theta$, hence  recovering the phase $\theta$.  
\end{remark}

To our purposes,  we need to extend the notion of pre-Jacobian to $BV$-maps.
\begin{definition}
For $u\in BV(\Omega)$, we define the pre-Jacobian of $u$ to be the measurable vector field
\[
j(u):=u\wedge \nabla u \,,
\] 
where $\nabla u$ is the absolutely continuous part of $Du$. It belongs to $ L^1(\Omega)$ whenever $u\in L^\infty(\Omega)$ or $\nabla u\in L^2(\Omega)$ (since $BV(\Omega)$ is continuously embedded in $L^2(\Omega)$). 
\end{definition}

\begin{lemma}\label{struct}
Let $u\in BV(\Omega)$. Then $V:=\P(u)$ and $v:=\p(u)$  are of  bounded variation in~$\Omega$, and
\begin{enumerate}
\item[(i)] $V(x)\in\mathcal{N}$ for a.e. $x\in\Omega$;
\vskip5pt

\item[(ii)] $J_{V}\subset J_u$; 
\vskip5pt

\item[(iii)]  $\big(V^+,V^-, \nu_{V}\big) =\big(\P(u^+),\P(u^-), \nu_{u}\big) $ on  $J_{V}$; 
\vskip5pt

\item[(iv)] $\P\big(u^+(x)\big)=\P\big(u^-(x)\big)$ for every $x\in J_u \setminus J_{V}$; 
\vskip5pt

\item[(v)] $\left|\nabla V\right|= |\nabla u|$ a.e. in $\Omega$;
\vskip5pt

\item[(vi)] $\left|D^cV\right|=|D^cu|$;
\vskip5pt

\item[(vii)] $j(v)=m j(u)$ a.e. in $\Omega$.
\end{enumerate}
\end{lemma}

\begin{proof}
The fact that $V\in BV(\Omega;\R^3)$, as well as items {\it(i)}, {\it(ii)}, {\it(iii)}, and {\it(iv)}, is a direct consequence of the $1$-Lipschitz property of $\P$, see \cite[proof of Theorem 3.96]{AFP}. Moreover, $|D V|\leq |Du|$. It now remains to prove {\it(v)}, {\it(vi)}, and {\it(vii)}. 
Recall that, by \cite[Proposition 3.92]{AFP}, we have $|Du|(Z_u)=0 $ where 
$$Z_u:=\big\{x\in\Omega \setminus J_u: u(x)=0\big\}\,.$$ 
For $k\in \N$, we set 
$$A_0:=\big\{x\in\Omega \setminus J_u: | u(x)|>1\big\}\,, \quad A_k:=\big\{x\in\Omega \setminus J_u: 2^{-k}<| u(x)|\leq 2^{-k+1}\big\}\,,$$
so that $\Omega \setminus Z_u=\cup_kA_k$ with a disjoint union.  
Then, for each $k\in\N$, we consider $\P_k\in C^1(\C;\R^3)$ such that
$\P_k(z)=\P(z)$ whenever $|z|>2^{-k}$. Using the chain-rule formula in $BV$ (see \cite[Theorem 3.96]{AFP}), for $\P_k( u)$ and the locality of the derivative of a $BV$ function (see \cite[Remark~3.93]{AFP}), we readily obtain {\it(v)} and {\it(vi)}.  

To prove {\it(vii)}, we first notice that for $z\in\C \setminus\{0\}$ and $X\in\R^2$, we have 
\begin{equation*}
 \p(z)\wedge (\nabla \p(z)X)=m z\wedge X\,.
\end{equation*} 
Therefore, if $x\in\Omega \setminus Z_u$ is a Lebesgue point for $\nabla u$ and $\nabla V$, we have for each $l\in\{1,2\}$, 
$$v(x)\wedge \partial_lv(x)=\p(u(x))\wedge\Big(\nabla\p(u(x))\partial_lu(x)\Big)=m u(x)\wedge \partial_lu(x) \,,$$
and the proof is complete. 
\end{proof}

\begin{corollary}\label{cor1}
If $u\in BV(\Omega)$  is such that $\P(u)\in W^{1,p}(\Omega;\mathcal{N})$ for some $p\geq 1$, then  $u\in SBV(\Omega)$ and $\nabla u\in L^p(\Omega)$.  Moreover, $u^\pm(x)\neq 0$ for every $x\in J_u$, and $u^+(x)/u^-(x)\in \mathbf{G}_m$. If, in addition, $|u|\geq \delta$ a.e. in $\Omega$\, for some $\delta>0$, 
then $u\in SBV^p(\Omega)$ and $|D^ju|\geq \delta |\mathbf{a}-1|\mathcal{H}^1\res{J_u}$. 
\end{corollary}

\begin{proof}
The fact that $u\in SBV(\Omega)$ and $\nabla u\in L^p(\Omega)$ is a direct consequence of {\it(vi)} and {\it(v)} in Lemma~~\ref{struct}, respectively.  Next, assume that $u^+(x)=0$ for some $x\in J_u$. Then {\it(iv)} in Lemma~~\ref{struct} yields $u^-(x)=0$, so that $x\not\in J_u$. Hence $u^\pm$ does not vanish on $J_u$. 
Moreover from {\it(iv)} in Lemma~~\ref{struct}, we directly infer that  $u^+/u^-\in  \mathbf{G}_m \setminus\{1\}$ on $J_u$. 

Finally, if $|u|\geq \delta>0$ a.e. in $\Omega$, then  $|u^\pm|\geq \delta$ on $J_u$. Therefore, for every $x\in J_u$ we have 
$$|u^+(x)-u^-(x)|\geq \delta |u^+(x)/u^-(x)-1|\geq \delta \min_{k=1,\ldots,m-1} |\mathbf{a}^k-1| = \delta |\mathbf{a}-1|\,,$$
and thus $|D^ju|\geq \delta |\mathbf{a}-1|\mathcal{H}^1\res{J_u}$ by~\eqref{formDju}. In particular, $\mathcal{H}^1(J_u)<\infty$ and $u\in SBV^p(\Omega)$. 
\end{proof}

\begin{definition}[weak Jacobian]
For an open set $\Omega\subset\R^2$ and $u\in BV(\Omega)$ such that $j(u)\in L^1(\Omega)$, the weak Jacobian of $u$ is defined as the distributional curl in $\Omega$ of the vector field $j(u)$. It belongs to $(C^{0,1}_0(\Omega))^*$, and its action on a Lipschitz function $\phi\in C^{0,1}_0(\Omega)$ that vanishes on the boundary is 
$$\langle {\rm curl}\,j(u),\phi\rangle=-\int_\Omega j(u)\cdot\nabla^\perp\phi\,dx\,.$$
\end{definition}

\begin{lemma}\label{lemstructure}
Assume that $\Omega\subset\R^2$ is simply connected. Let $u_1,u_2\in SBV(\Omega;\Ss^1)$ be such that $\p(u_k)\in W^{1,1}(\Omega;\Ss^1)$ for $k=1,2$. Then, the following properties are equivalent
\begin{enumerate}
\vskip3pt
\item[(i)] $\curl\,j(u_1)= \curl\,j(u_2)$ in $\mathcal{D}^\prime(\Omega)$; 
\vskip3pt
\item[(ii)]  there exist $\vhi\in W^{1,1}(\Omega)$ and  a Caccioppoli partition $\{E_k\}_{k=1}^m$ of $\Omega$ (see e.g. \cite[Chapter 4, Section 4.4]{AFP}) such that 
\begin{equation}\label{Caccioprep}
u_2=\left(\sum_{k=1}^m{\bf a}^k\chi_{E_k}\right)e^{i\vhi} u_1 \,.
\end{equation}
\end{enumerate}
In addition, if\, $\P(u_1)=\P(u_2)$ and (i) holds, then $\vhi$ is a multiple constant of $2\pi/m$. 
\end{lemma}

\begin{proof}
Define $\widetilde u:=u_2\overline u_1\in SBV(\Omega;\Ss^1)$ and $\widetilde v:=\p(\widetilde u)\in W^{1,1}(\Omega;\Ss^1)$. By  Corollary~~\ref{cor1}, we have $\mathcal{H}^1(J_{\widetilde u})<\infty$. Then  Lemma~~\ref{struct}, together with the fact that $\p(\tilde u)=\p(u_2)\overline{\p(u_1)}$, leads to 
$$j(\widetilde v)=j\big(\p(u_2)\,\overline{\p(u_1)}\,\big) =j(\p(u_2))-j(\p(u_1))=m\big(j(u_2)-j(u_1)\big)\,.$$

If {\it(i)} holds, then $\curl\,j(\widetilde v)=0$ in $\mathcal{D}^\prime(\Omega)$. By~\cite{Demengel} (see also \cite[Theorem 7]{BMP}) there exists $\vhi\in W^{1,1}(\Omega)$ such that 
$\widetilde v=e^{im \vhi}$. Consequently, $\p(e^{-i\vhi} \widetilde u)=1$ and thus $e^{-i\vhi} \widetilde u\in BV(\Omega;{\bf G}_m)$, so that 
$e^{-i\vhi} \widetilde u= \sum_{k=0}^{m-1}{\bf a}^k\chi_{E_k}$ for some Caccioppoli partition $\{E_k\}_{k=0}^{m-1}$ of $\Omega$. This proves~\eqref{Caccioprep}.
When $\p(u_1)=\p(u_2)$, then $\widetilde v=1$, and we infer that $\vhi(x)\in \frac{2\pi}{m}\Z$ for a.e. $x\in\Omega$. Since $\vhi\in W^{1,1}(\Omega)$ we conclude that $\vhi$ is constant.

If {\it(ii)} holds, then for each $l\in\{1,2\}$, 
$$\partial_l u_2=\left(\sum_{k=0}^{m-1}{\bf a}^k\chi_{E_k}\right)e^{i\vhi} \big(\partial_l u_1+i\partial_l \vhi u_1\big)\quad\text{a.e. in $\Omega$}\,.$$ 
Consequently, $j(u_2)=j(u_1)+\nabla\vhi$ a.e. in $\Omega$, and {\it(ii)} follows. 
\end{proof}

%%%%%%%%%%%%%%%%%%%%%%%%%%%%%%%%%%%%%%%%%%%%%%%%%%%%%%%

\subsection{Energies, functional classes, and existence of minimizers} 
\label{S24}
Throughout this section, we assume that $\Omega\subset \R^2$ is a smooth, bounded, and {\sl simply connected} domain. 
For $q\in\{2,3\}$ and $\varepsilon>0$, we consider the  Ginzburg-Landau functional $E_\varepsilon:W^{1,2}(\Omega;\R^q)\to[0,\infty)$ defined by 
 \[
 E_\varepsilon(u):=\frac{1}{2}\int_\Omega|\nabla u|^2\,dx+\frac{1}{4\varepsilon^2}\int_\Omega(1-|u|^2)^2\,dx \,.
 \]
For any Borel set $A\subset \Omega$, we let 
\[
 E_\varepsilon(u,A):=\frac{1}{2}\int_A|\nabla u|^2\,dx+\frac{1}{4\varepsilon^2}\int_A(1-|u|^2)^2\,dx \,.
\]
We shall use the analogous notation for the localized version of most of the energies under consideration.   

For $u\in SBV^2(\Om)$ such that $v:=\p(u)=u^m/|u|^{m-1}\in W^{1,2}(\Omega)$, we have (by Lemma \ref{struct})
\begin{equation}\label{GLcone}
E_\eps\lt(\P(u)\rt)= \frac{1}{m^2} E_\eps(v)+\frac{m^2-1}{m^2}  E_\eps(|v|)\,=:\, G_\eps(v)\,.
\end{equation}
Equivalently, the functional $G_\eps:W^{1,2}(\Omega)\to [0,\infty)$ can be defined by 
\begin{equation}\label{defGeps}
G_\eps(v)= \frac{1}{2m^2}\int_{\Omega}|\nabla v|^2+(m^2-1)\big|\nabla |v| \big|^2\,dx+\frac{1}{4\eps^2}\int_\Omega(1-|v|^2)^2\,dx\,.
\end{equation}
For the phase field we consider the functionals $I_\eta:W^{1,2}(\Omega;[0,1])\to[0,\infty)$ defined for $\eta>0$ as 
$$I_\eta(\psi):=\frac{\eta}{2}\int_\Omega|\nabla\psi|^2\,dx+\frac{1}{2\eta}\int_\Omega(1-\psi)^2\,dx\,. $$
The classes of functions we are interested in are the following 
\[
\mathcal{H}(\Omega):=\Big\{(u,\psi)\in L^2(\Omega)\times W^{1,2}(\Omega;[0,1]):\P(u)\in W^{1,2}(\Omega;\mathcal{N})
\text{ and }\psi u \in W^{1,2}(\Omega)\Big\}\,,
\]
and 
\[
 \mathcal{G}(\Omega) :=\Big\{ u\in SBV^2(\Omega)  : \ \P(u)\in W^{1,2}(\Omega;\mathcal{N}) \Big\}\,.
\]
Notice that in the definition of ${\mathcal H}(\Om)$, the condition $\psi u\in W^{1,2}(\Om)$ degenerates on the set $\{\psi=0\}$ allowing for discontinuities of $u$. Typically, $u$  may jump through lines where $\psi$ vanishes and (since $\P(u)$ does not jump) the jump satisfies formally the constraint $\P(u^+)=\P(u^-)$ in the spirit of Lemma~~\ref{struct} ~{\it(iv)}.  
\vskip3pt

On the classes $\mathcal{H}(\Omega)$ and $ \mathcal{G}(\Omega)$, we define the functionals $F^\eta_{\varepsilon}:\mathcal{H}(\Omega)\to[0,\infty)$  and $ F^0_\eps:\mathcal{G}(\Omega)\to[0,\infty)$ by
 \begin{equation}\label{defFetaeps&F0eps}
 F^\eta_{\varepsilon}(u,\psi):=E_\varepsilon\big(\P(u)\big)+ I_\eta(\psi)\, \qquad \textrm{and} \qquad  F^0_\eps(u):= E_\eps\big(\P(u)\big) +\mathcal{H}^1(J_u)\,.
 \end{equation}
Note that $F^0_\eps$ is a functional of the type ``Mumford-Shah''. Indeed, by Lemma ~\ref{struct} we have
$$F^0_\eps(u)=\frac{1}{2}\int_\Omega|\nabla u|^2+\frac{1}{2\eps^2}(1-|u|^2)^2\,dx+\mathcal{H}^1(J_u)\,.$$
As already pointed out in the introduction,  $F^\eta_{\varepsilon}$ can be seen as an ``Ambrosio-Tortorelli'' regularization of $F^0_\eps$ (with a coupling between $u$ and $\psi$ in the class $\mathcal{H}(\Omega)$ rather than in the functional itself). 
 \vskip3pt
 
We aim to minimize  $F^\eta_{\varepsilon}$ and $F^0_\eps$ under a given Dirichlet condition on the boundary. We fix a smooth map $g:\partial\Omega\simeq\Ss^1\to\Ss^1$ of topological degree $d\geq 0$. Accordingly, we introduce the subclasses 
\begin{equation}\label{classHg}
\mathcal{H}_g(\Omega):=\Big\{(u,\psi)\in\mathcal{H}(\Omega): \psi=1\text{ and }\psi u=g \text{ on $\partial\Omega$}\Big\}\,,
\end{equation}
and
$$\mathcal{G}_g(\Omega):=\Big\{u\in\mathcal{G}(\Omega): \P(u)=\P(g) \text{ on $\partial\Omega$}\Big\}\,.$$ 
Note that in  $\mathcal{G}_g(\Omega)$ we do not impose the condition $u=g$ on $\partial\Omega$. Instead we penalize deviations from $g$ minimizing over $\mathcal{G}_g(\Omega)$ the functional
\begin{equation}\label{defFeps0g}
F^0_{\eps,g}(u):=F_\eps^0(u)+\mathcal{H}^1\big(\{u\not=g\}\cap\partial\Omega\big)
\end{equation}
in place of $F_\eps^0$. As already mentioned in the introduction, this penalization is necessary to ensure lower semi-continuity (since $F_{\eps,g}^0$ is precisely the $L^1(\Omega)$-relaxation of $F^0_\eps$).
\vskip3pt

As a warm-up, let us prove that the functionals $F^\eta_\eps$ and $F^0_{\eps,g}$ admit minimizers.
\begin{theorem}\label{exist}
The functional $F^\eta_{\varepsilon}$ admits a minimizing pair $(u_\eps,\psi_\eps)$ in $\mathcal{H}_g(\Omega)$. In addition, any such minimizer satisfies $\|u_\eps\|_{L^\infty(\Omega)}\leq 1$. 
\end{theorem}

\begin{theorem}\label{existsharp}
The functional $F^0_{\varepsilon,g}$ has a  minimizer $u_\eps$ in $\mathcal{G}_g(\Omega)$. In addition, any such minimizer satisfies $\|u_\eps\|_{L^\infty(\Omega)}\leq 1$. 
\end{theorem}

The proof of Theorem ~\ref{exist} rests on the following compactness result. 

\begin{proposition}\label{compactetafix}
Let $\Omega$ be a bounded open subset of $\R^2$. Let $\{(u_h,\psi_h)\}_{h\in\N}\subset \mathcal{H}(\Omega)$ be such that 
$$\sup_{h}\left\{\|u_h\|_{L^\infty(\Omega)}+\|\nabla \psi_h\|_{L^{2}(\Omega)}+\big\|\nabla(\P( u_h))\big\|_{L^{2}(\Omega)}\right\}<\infty\,.  $$
Then, there exist a  subsequence and $(u,\psi)\in \mathcal{H}(\Omega)$ such that 
$$\big(\psi_h,\psi_h u_h,\P(u_h)\big)\rightharpoonup \big(\psi,\psi u,\P(u)\big)\quad\text{weakly in $W^{1,2}(\Omega)$}\,.$$ 
\end{proposition}

\begin{proof}
Set $\phi_h:=\psi_h u_h$, $V_h:=\P(u_h)$, and notice that $\P( \phi_h)= \psi_hV_h$.  Therefore, 
$$\nabla(\P( \phi_h))=\psi_h\nabla V_h+\nabla\psi_h \otimes V_h\,. $$
By~\eqref{modnabl}, $|\nabla \phi_h|=\big| \nabla(\P( \phi_h))\big|$ a.e. in $\Omega$, and since $0\leq \psi_h\leq 1$, we infer  that 
$$\int_\Omega|\nabla\phi_h|^2\,dx \leq 2\int_{\Omega}|\nabla V_h|^2\,dx +2 \|u_h\|^2_{L^\infty(\Omega)}\int_\Omega|\nabla\psi_h|^2\,dx\,.$$
Hence $\{\psi_h\}$, $\{V_h\}$, and $\{\phi_h\}$ are bounded in $W^{1,2}(\Omega)$. Thus, we can find a subsequence such that 
$(\psi_h,\phi_h,V_h)\rightharpoonup (\psi,\phi,V)$ weakly in $W^{1,2}(\Omega)$ and a.e. in $\Omega$, for some $(\psi,\phi,V)\in W^{1,2}(\Omega;[0,1])\times W^{1,2}(\Omega)\times W^{1,2}(\Omega;\mathcal{N})$. 
On the one hand, since $\{u_h\}$ is bounded in $L^\infty(\Omega)$, the sequence $\{u_h(x)\}$ is bounded for a.e. $x\in\Omega$, and we deduce that $\phi(x)=\lim_h \psi_h(x)u_h(x)=0$ for a.e. $x\in \{\psi=0\}$. On the other hand, one has
$\lim_h u_h(x)=\phi(x)/\psi(x)$ and $V(x)=\lim_h V_h(x)=\P(\phi(x)/\psi(x))$ for a.e. $x\in\{\psi\not=0\}$.  
Now, we define $u\in L^\infty(\Omega)$ by  setting 
$$ u:= \frac{\phi}{\psi} \chi_{\{\psi\not=0\}}+\boldsymbol{\alpha}(V)\chi_{\{\psi=0\}} \,,$$
where $\boldsymbol{\alpha}:\mathcal{N}\to\C$ is the ``unrolling map" of the cone $\mathcal{N}$, i.e., 
$$\boldsymbol{\alpha}(z,t):=\begin{cases}
m|z|e^{i\theta/m} & \text{for $z=|z|e^{i\theta}\in\C \setminus\{0\}$ with $\theta\in[0,2\pi)$}\,,\\
0 & \text{otherwise}\,.
\end{cases}$$
By construction, we have $\phi=\psi u$ and $V=\P(u)$, and the proof is complete. 
 \end{proof}

\begin{proof}[Proof of Theorem ~\ref{exist}]
Let $\{(u_h,\psi_h)\}\subset \mathcal{H}_g(\Omega)$ be a minimizing sequence for $F^\eta_{\varepsilon}$ in $ \mathcal{H}_g(\Omega)$.  Since $\P(u_h)\in W^{1,2}(\Omega;\R^3)$, we have $|u_h|\in W^{1,2}(\Omega)$ and thus also $[\max(1,|u_h|)]^{-1}\in W^{1,2}(\Omega)\cap L^\infty(\Omega)$.
Since $\psi_h =1$ and $\psi_h u_h=g$ on $\partial\Omega$, we have $|u_h|=1$ on $\partial \Omega$ and then also $[\max(1,|u_h|)]^{-1}=1$ on $\partial \Omega$.
 Consequently, setting 
\begin{equation}\label{deftronc}
 \widehat u_h:=\frac{u_h}{\max(1,|u_h|)}\,,
 \end{equation}
we have $\psi_h \widehat u_h\in W^{1,2}(\Omega)$, $\psi_h \widehat u_h=g$ on $\partial\Omega$, and $\| \widehat u_h\|_{L^\infty(\Omega)}\leq 1$. Since also,
\begin{equation}\label{tronc}
\P(\widehat u_h)=\frac{\P(u_h)}{\max(1,|u_h|)}=\frac{\P(u_h)}{\max(1,|\P(u_h)|)} \in W^{1,2}(\Omega;\mathcal{N})\,,\end{equation}
we have $(\psi_h,\widehat u_h)\in \mathcal{H}_g(\Omega)$. Moreover,~\eqref{tronc} implies that 
$E_\varepsilon(\P( \widehat u_h))\leq E_\varepsilon(\P(u_h)) $ with equality if and only if $|\P(u_h)|\leq 1$ a.e. in $\Omega$ (and since $|\P(u_h)|=|u_h|$, equality holds if and only if $|u_h|\leq 1$ a.e. in $\Omega$). 

As a consequence, $F^\eta_{\varepsilon}(\widehat u_h,\psi_h)\leq F^\eta_{\varepsilon}(u_h,\psi_h)$, and thus  $\{(\widehat u_h,\psi_h)\}$ 
is also a minimizing sequence  for $F^\eta_{\varepsilon}$ in $\mathcal{H}_g(\Omega)$. Since $\|\widehat u_h\|_{L^\infty(\Omega)}\leq 1$, we can apply Proposition~~\ref{compactetafix} to find a subsequence 
such that $(\psi_h,\psi_h\widehat u_h,\P(\widehat u_h))\rightharpoonup (\psi_\eps,\psi_\eps u_\eps,\P(u_\eps))$ weakly in $W^{1,2}(\Omega)$ for some
$(u_\eps,\psi_\eps)\in\mathcal{H}(\Omega)$ with $|\P(u_\eps)|=|u_\eps|\leq 1$ a.e. in $\Omega$.  
From the continuity of the trace operator, we deduce that $\psi_\eps=1$  and $\psi_\eps u_\eps=g$ on $\partial\Omega$, that is $(u_\eps,\psi_\eps)\in\mathcal{H}_g(\Omega)$. Finally, the functional $F^\eta_{\varepsilon}$ being clearly lower semi-continuous with respect to the weak convergence in $W^{1,2}(\Omega)$, we conclude that $(u_\eps,\psi_\eps)$ minimizes $F^\eta_{\varepsilon}$ over $\mathcal{H}_g(\Omega)$. Since the truncation argument above shows that any minimizer satisfies the announced $L^\infty$ bound, the proof is complete.  
\end{proof}

\begin{proof}[Proof of Theorem ~\ref{existsharp}]
The truncation argument is identical to the one above so we may reduce ourselves to the class of functions $u$ satisfying $\|u\|_{L^\infty(\Om)}\le 1$. Let $\{u_h\}\subset \mathcal{G}_g(\Omega)$ be a minimizing sequence for $F^0_{\varepsilon,g}$.

We fix some $r_0>0$ small enough in such a way that 
\begin{equation}\label{extdomom}
\widetilde \Omega:=\big\{x\in\R^2: {\rm dist}(x,\Omega)<r_0\big\}
\end{equation}
defines a smooth domain, and that the nearest point projection on $\partial\Omega$, denoted by ${ \Pi}$, is well defined and smooth in $\big\{x\in\R^2: {\rm dist}(x,\partial\Omega)<2r_0\big\}$. We extend each $u_h$ 
to $\widetilde\Omega$ by setting $u_h(x)=g\big({\Pi}(x)\big)$ for $x\in\widetilde\Omega \setminus\Omega$. Then we have  $J_{u_h}\cap\widetilde\Omega=(J_{u_h}\cap\Omega)\cup(\{u_h\neq g\}\cap\partial\Omega)$, so that 
$$F^0_{\varepsilon}(u_h,\widetilde\Omega)=F^0_{\varepsilon,g}(u_h)+C_g \,,$$ 
for a constant $C_g$ depending only on $g$, $r_0$, and $\Omega$. Since $|\nabla u_h|=|\nabla(\P(u_h))|$ by Lemma ~\ref{struct}, we deduce that $\{\nabla u_h\}$ is bounded in $L^2(\widetilde \Omega)$. Hence we can apply \cite[Theorem 4.7 \& 4.8]{AFP} to find a  subsequence such that $u_h\rightharpoonup u_\eps$ weakly* in $BV(\widetilde\Omega)$ and a.e. in $\Omega$ to some $u_\eps\in SBV^2(\widetilde\Omega)$. From the a.e. convergence, we deduce that $u_\eps(x)=g\big({\Pi}(x)\big)$ for $x\in\widetilde\Omega \setminus\Omega$. Then, still by \cite[Theorem 4.7]{AFP},  
\begin{multline}\label{malad1148}
\liminf_{h\to\infty} \mathcal{H}^1(J_{u_h}\cap\Omega)+\mathcal{H}^1\big(\{u_h\not=g\}\cap\partial\Omega\big)=\liminf_{h\to\infty} \mathcal{H}^1(J_{u_h}\cap\widetilde\Omega)\\
\geq  
 \mathcal{H}^1(J_{u_\eps}\cap\widetilde\Omega) =\mathcal{H}^1(J_{u_\eps}\cap\Omega)+\mathcal{H}^1\big(\{u_\eps\not=g\}\cap\partial\Omega\big)\,.
 \end{multline}
Since $\{\P(u_h)\}$ is bounded in $W^{1,2}(\Omega)$ and $\P(u_h)\to \P(u_\eps)$ a.e. in $\Omega$, we infer that  $\P(u_h)\rightharpoonup \P(u_\eps)$ weakly in $W^{1,2}(\Omega)$.  
As a consequence, $u_\eps\in\mathcal{G}_g(\Omega)$. Finally, the lower semi-continuity of $E_\eps$ with respect to the weak $W^{1,2}$-convergence, together with~\eqref{malad1148}, leads to $F^0_{\eps,g}(u_\eps)\leq \liminf_hF^0_{\eps,g}(u_h)$. Hence $u_\eps$ is a minimizer of $F^0_{\eps,g}$ in  $\mathcal{G}_g(\Omega)$. 
 \end{proof}

%%%%%%%%%%%%%%%%%%%%%%%%%%%%%%%%%%%%%%%%%%%%%%%%%%%%%%%

\subsection{Asymptotic for the Ginzburg-Landau functional} 
\label{S25}
The aim of this subsection is to recall some classical facts about the asymptotic limit as $\eps\downarrow 0$ of low energy states for the Ginzburg-Landau functional $E_\eps$. In this section  we still assume that $\Omega\subset\R^2$  is a smooth, bounded, and {\sl simply connected} domain. 
Some of the material below can be found with greater details in \cite{BBH, SS, AlicPon} and the references therein. We start with the notion of renormalized energy originally introduced in \cite{BBH}.

\subsubsection{The renormalized energy and canonical harmonic maps}

Let us denote by $\mathcal{A}_{d}$ the set of all finite positive measures $\mu$ of the form 
\begin{equation}\label{defAd}
\mu=2\pi\sum_{k=1}^{md} \delta_{x_k}\,,
\end{equation}
for some $md$ distinct points $\{x_1,\ldots,x_{md}\}\subset \Omega$.

Given $\mu\in \mathcal{A}_{d}$, the canonical harmonic map $v_\mu:\Omega \setminus{\rm spt}\,\mu\to \C$ associated to $\mu$ is the map defined by 
\begin{equation}\label{definivmu}
v_\mu(x):=e^{i\vhi_\mu(x)}\prod _{k=1}^{md}\frac{x-x_k}{|x-x_k|}\,,
\end{equation}
with 
$$\begin{cases}
\Delta\vhi_\mu=0 &\text{in $\Omega$}\,,\\
v_\mu=g^m & \text{on $\partial\Omega$}\,.
\end{cases}$$
Note that $\vhi_\mu$ is a smooth function in $\overline\Omega$ uniquely determined up to constant multiple of $2\pi$. The canonical map $v_\mu$ is a smooth harmonic map from $\Omega \setminus{\rm spt}\,\mu$ into $\Ss^1$. It satisfies
\begin{equation*}
\begin{cases}
{\rm div}\, j(v_\mu) = 0\\
\curl\, j(v_\mu)=\mu
\end{cases}\quad\text{in $\mathcal{D}^\prime(\Omega)$}\,.
\end{equation*}
It turns out that  $v_\mu\in W^{1,p}(\Omega)$ for every $p\in [1,2)$, but fails to be in $W^{1,2}(\Omega)$. However, the Dirichlet energy of $v_\mu$ still have a well defined 
 finite part called {\it the renormalized energy} given by 
\begin{equation}\label{formulW}
\mathbb{W}(\mu):=-\pi\sum_{k\neq l}\log|x_k-x_l| +\frac{1}{2}\int_{\partial\Omega}g^m\wedge\frac{\partial (g^m)}{\partial\tau} \,\Phi_\mu  \,d\mathcal{H}^{1}-\pi\sum_{k=1}^{md}R_\mu(x_k)\,, 
\end{equation}
where $\Phi_\mu$ is the solution of 
$$\begin{cases}
\Delta\Phi_\mu=\mu & \text{in $\Omega$}\,,\\[8pt]
\displaystyle \frac{\partial\Phi_\mu}{\partial\nu}= g^m\wedge\frac{\partial (g^m)}{\partial\tau} & \text{on $\partial\Omega$}\,,\\[8pt]
\displaystyle \int_{\partial\Omega}\Phi_\mu \,d\mathcal{H}^1=0\,,
\end{cases}$$
and $R_\mu(x):=\Phi_{\mu}(x)-\sum_k\log|x-x_k|$\,. Note that $R_\mu$ is an harmonic function in $\Omega$, smooth up to~$\partial\Omega$. The function $\Phi_\mu$ is related to the harmonic map $v_\mu$ through the relation 
\begin{equation}
\label{jvmu}
j(v_\mu)=\nabla^\perp\Phi_\mu\,,
\end{equation} 
and $\mathbb{W}(\mu)$ is the finite part of the Dirichlet energy of $v_\mu$ in the sense that 
\begin{equation}\label{asymptcanharm}
\lim_{r\downarrow 0}\left\{\frac{1}{2} \int_{\Omega \setminus B_r(\mu)}|\nabla v_\mu|^2\,dx-\pi md|\log r|\right\}=\mathbb{W}(\mu)\,.
\end{equation}
\vskip5pt

\subsubsection{Asymptotic for low energy states}

We are now ready to state the following compactness result, which is a slight improvement of \cite[Th. 6.1]{AlicPon}. The proof is postponed to the end of this section.

\begin{theorem}\label{compAlicPon}
 For a sequence  $\eh\downarrow0$, let  $\{v_{h}\}\subset W^{1,2}_{g^m}(\Omega)$ be such that $\{v_h\}$ is bounded in $L^\infty(\Omega)$, and
\begin{equation}\label{asymptenerg}
E_{\eh}(v_h) \leq \pi m d |\log\eps_h| +O(1)\quad\text{as $h\to\infty$}\,.
\end{equation}
There exist a subsequence,  a measure $\mu\in\mathcal{A}_d$, and a phase $\vhi\in W^{1,2}(\Omega)$ such that 
\begin{enumerate}
\item[(i)] $v_h \rightharpoonup e^{i\vhi}v_\mu $ weakly in $W^{1,p}(\Omega)$ for every $p\in [1,2)$;
\vskip3pt
\item[(ii)] $v_h \rightharpoonup e^{i\vhi}v_\mu $ weakly in $W^{1,2}_{\rm loc}(\overline\Omega \setminus{\rm spt}\,\mu)$;
\vskip3pt
\item[(iii)] $e^{i\vhi}=1$ on $\partial\Omega$; 
\item[(iv)] for $r>0$ small enough,
\begin{equation}\label{lowerbBGAlicPon}
 \liminf_{h\to\infty} \left\{ E_{\eps_h}\big(v_h,B_r(\mu)\big) -\pi m d \log\frac{r}{\eps_h}\right\}\geq C_*\,,
\end{equation}
for a constant $C_*$ independent of $r$,  and 
\begin{equation*}
 \liminf_{r\downarrow 0} \liminf_{h\to \infty} \left\{\frac{1}{2} \int_{\Omega\setminus B_r(\mu)} |\nabla v_h|^2\,dx-\pi m d |\log r|\right\}\geq \frac{1}{2} \int_{\Omega} |\nabla \vhi|^2\,dx+\mathbb{W}(\mu)\,.
\end{equation*}
\end{enumerate}
Moreover, $\mu_h:=\curl\,j(v_h)\in L^1(\Omega)$ converges to $\mu=\curl\,j(e^{i\vhi}v_\mu)$ in the weak* topology of $(C^{0,1}_0(\Omega))^*$. 
\end{theorem}

The proof of this theorem relies on the following two auxiliary results. In particular, Lemma ~\ref{lemboundWq} provides an a priori $W^{1,p}$-bound for sequences of low Ginzburg-Landau energy. We believe that Proposition~\ref{keypropcomp} and Lemma \ref{lemboundWq} are already well known to experts (see in particular \cite[Theorem 1.4.4]{ColJer}). Since we did not find
clear statements and proofs in the existing literature, we have decided to provide here  (mostly) self-contained proofs. 

\begin{proposition}\label{keypropcomp}
Let $v\in W^{1,1}(\Omega;\Ss^1)$ and $\mu\in\mathcal{A}_d$ be such that 
$$\begin{cases}
\curl\,j(v)=\mu &\text{in $\mathcal{D}^\prime(\Omega)$}\,,\\
v=g^m & \text{on $\partial\Omega$}\,.
\end{cases}$$ 
If $ v\in W^{1,2}_{\rm loc}(\overline\Omega \setminus {\rm spt}\,\mu)$ and 
\begin{equation}\label{condenergren}
\liminf_{r\downarrow 0}\left\{\frac{1}{2} \int_{\Omega \setminus B_r(\mu)}|\nabla v|^2\,dx-\pi m d|\log r|\right\}<\infty\,,
\end{equation}
then $v=e^{i\vhi}v_\mu$  for some $\vhi\in W^{1,2}(\Omega)$ such that $e^{i\vhi}=1$ on $\partial\Omega$. In addition, 
$$\lim_{r\downarrow 0}\left\{\frac{1}{2} \int_{\Omega \setminus B_r(\mu)}|\nabla v|^2\,dx-\pi m d|\log r|\right\}=\frac{1}{2}\int_\Omega|\nabla\vhi|^2\,dx+ \mathbb{W}(\mu)\,.$$
\end{proposition}

\begin{proof} 
The fact that $v=e^{i\vhi}v_\mu$ for some $\vhi\in W^{1,1}(\Om)$ with $e^{i\vhi}=1$ on $\partial \Omega$ follows as in the proof of Lemma~~\ref{lemstructure}. Moreover, 
$ v\in W^{1,2}_{\rm loc}(\overline\Omega \setminus {\rm spt}\,\mu;\Ss^1)$ yields $\vhi \in W^{1,2}_{\rm loc}(\overline\Omega \setminus {\rm spt}\,\mu)$. Let us prove that in fact $\vhi\in W^{1,2}(\Omega)$.
First notice that 
\begin{equation}\label{eq:nablav}
|\nabla v|^2= |j(v)|^2=|\nabla \vhi|^2+|j(v_\mu)|^2+2\nabla\vhi\cdot j(v_\mu)=|\nabla\vhi|^2+|\nabla v_\mu|^2+2\nabla\vhi\cdot\nabla^\perp\Phi_\mu\,, 
\end{equation}
where  the last identity follows from \eqref{jvmu}.

For each $k\in \{1,\ldots,md\}$, we set 
$$R_\mu^k(x):=\Phi_\mu(x) -\log|x-x_k| \,,$$
so that $R_\mu^k$ is a smooth harmonic function in $\overline\Omega \setminus\cup_{l\neq k}\{x_l\}$. Notice in particular that 
\[
 \partial_\tau \Phi_\mu=\partial_\tau R_\mu^k \qquad \textrm{on } \partial B_r(x_k). 
\]
Integrating by parts \eqref{eq:nablav} in $\Omega_r:=\Omega \setminus B_r(\mu)$ with $r>0$ small enough, leads to
\begin{align}
 \int_{\Omega_r}|\nabla v|^2\,dx&= \int_{\Omega_r}|\nabla\vhi|^2\,dx+ \int_{\Omega_r}|\nabla v_\mu|^2\,dx+2\sum_{k=1}^{md}\int_{\partial B_{r}(x_k)} \vhi \partial_\tau \Phi_\mu\,d\mathcal{H}^1 \nonumber\\
 &= \int_{\Omega_r}|\nabla\vhi|^2\,dx+ \int_{\Omega_r}|\nabla v_\mu|^2\,dx+2\sum_{k=1}^{md}\int_{\partial B_{r}(x_k)} \vhi \partial_\tau R^k_\mu\,d\mathcal{H}^1\,.\label{devlpmtenv}
 \end{align}
By the boundary trace theorem for $BV$ functions \cite[Theorem 3.87]{AFP}, and the embedding of $W^{1,1}$ into $L^2$,
 \begin{equation}\label{poincphi}
 \int_{\partial B_{r}(x_k)} |\vhi| \,d\mathcal{H}^1\les \int_{B_{r}(x_k)}|\nabla \vhi|+\frac{1}{r}|\vhi| \,dx\les \int_{B_{r}(x_k)}|\nabla \vhi| dx+\left(\int_{B_r(x_k)} |\vhi|^2 dx\right)^{1/2} \mathop{\longrightarrow}\limits_{r\to 0} 0\,.
 \end{equation}
 Using the smoothness of $R^k_\mu$ near $x_k$, we can combine~\eqref{asymptcanharm},~\eqref{condenergren},~\eqref{devlpmtenv}, and~\eqref{poincphi} to deduce that 
 $$\int_{\Omega_r}|\nabla\vhi|^2\,dx=O(1)\quad\text{as $r\to 0$}\,.$$
Therefore $\vhi \in W^{1,2}(\Omega)$. Going back to~\eqref{devlpmtenv}, we subtract  $\pi md|\log r |$ from both sides of this identity, and we let $r\to 0$ to reach the conclusion. 
\end{proof}

\begin{lemma}\label{lemboundWq}
 For a sequence  $\eh\downarrow 0$, let  $\{v_{h}\}\subset W^{1,2}_{g^m}(\Omega,\C)$ be such that $\{v_h\}$ is bounded in $L^\infty(\Omega)$,~\eqref{asymptenerg} holds, and for which $\mu_h:= j(v_h)$ weakly* 
 converges in $(C^{0,1}_0(\Omega))^*$ to some measure $\mu\in \mathcal{A}_d$ as $h\to\infty$. Then $\{v_h\}$ is  bounded in $W^{1,p}(\Omega)$ for every $1\leq p<2$.  
\end{lemma}

The proof of Lemma~~\ref{lemboundWq} rests on the so-called `ball construction'' in \cite[Theorem 4.1]{SS} that we now recall. 

\begin{theorem}[\cite{SS}]\label{thmlwdSS}
For any $\alpha\in(0,1)$ there exists $\eps_0(\alpha)>0$ such that, for any $\eps\in(0,\eps_0(\alpha))$ and any $v\in C^\oo({\Omega})$ satisfying $E_\eps(v)\leq \eps^{\alpha-1}$, the following holds for some universal constants ${ c}_0$, ${ c}_1$, and ${ c}_2$: 
for any $r\in \big[{ c}_0\eps^{\alpha/2},1\big)$ there exists a finite collection $\mathcal{B}_r=\big\{ B_j\big\}_{j\in J}$ of disjoint closed balls such that 
\begin{enumerate}
\item[(i)] $r=\sum_{j}r_j\,$; 
\vskip3pt

\item[(ii)] setting $\Omega_\eps:=\big\{x\in\Omega: {\rm dist}(x,\partial\Omega)>\eps\big\}$ and $V_r^\eps:=\Omega_\eps\cap\big(\cup_j   B_j\big)$, 
$$\Big\{x\in\Omega_\eps: \big||v(x)|-1\big|\geq \eps^{\alpha/4}\Big\}\subset  V_r^\eps\,;$$

\item[(iii)] setting $d_j={\rm deg}(v,\partial B_j)$ if  $ B_j \subset \Omega_\eps$, and $d_j=0$ otherwise, 
\begin{equation}\label{genlwden}
E_\eps(v, V_r^\eps)\geq \pi D_r\left(\log\frac{r}{D_r\eps} -{ c}_1\right) 
\end{equation}
whenever $D_r:=\sum_j|d_j| \neq 0$; 
\vskip3pt
\item[(iv)] the following estimate holds
\begin{equation}\label{bndtotD}
 D_r\leq { c}_2 \frac{E_\eps(v)}{\alpha|\log\eps|}\,.
 \end{equation}
\end{enumerate}
Finally, if $r_1<r_2$, then every ball of $\mathcal{B}_{r_1}$ is contained in a ball of $\mathcal{B}_{r_2}$. 
\end{theorem}

\begin{proof}[Proof of Lemma~~\ref{lemboundWq}]
Since $\Omega$ is a smooth bounded domain and $g$ is smooth, any map in $W^{1,2}_{g^m}(\Om)\cap L^\infty(\Om)$ can be (strongly) approximated in the $W^{1,2}$-sense by a sequence in  $\{v\in C^\infty({\Om}):v=g^m\mbox{ on }\partial \Om\}$ which also remains bounded in $L^\infty(\Om)$. Hence, we can assume $v_h\in C^\infty({\Om})$ for each $h$. Recall that $\mu$ writes $\mu=2\pi\sum_{k=1}^{md} \delta_{x_k}$. Setting 
$$\sigma_0:=\frac{1}{4}\min\Big\{1, \min_k {\rm dist}(x_k,\partial\Omega), \min_{k\neq l}|x_k-x_l|\Big\}\,,$$ 
we may assume without loss of generality that $\sigma_0=1$. We choose $\alpha=1/2$  in Theorem~~\ref{thmlwdSS} (this choice of $\alpha$ is arbitrary).  
By~\eqref{asymptenerg}, we have $E_{\eps_h}(v_h)\leq \eps_h^{-1/2}$ for $\eps_h$ small enough, and we can therefore apply  Theorem~~\ref{thmlwdSS} to $v_h$.  

We claim that for $\eps_h$ sufficiently small,  
\begin{equation}\label{lwdDr}
D_r\geq md \quad\text{for every  $r\in \big[{ c}_0\eps_h^{1/4},1/6\big]$}\,.
\end{equation}
Let us introduce the modified function
$$\widetilde v_h:=\min\bigg\{\frac{|v_h|}{1-\eps_h^{1/8}},1\bigg\} \frac{v_h}{|v_h|}\in W^{1,2}(\Omega)\,.$$
Noticing that 
$$j(\widetilde v_h)= \min\bigg\{\frac{1}{\big(1-\eps_h^{1/8}\big)^2},\frac{1}{|v_h|^2}\bigg\}j(v_h)\,,$$
and setting $\widetilde \mu_h:=\curl\, j(\widetilde v_h)$, we estimate
\begin{align}\nonumber
\|\widetilde \mu_h-\mu_h\|_{(C^{0,1}_0(\Omega))^*}&=\sup_{\|\phi\|_{C^{0,1}_0(\Omega))}\le 1} \int_{\Omega} \lt(\min\bigg\{\frac{1}{\big(1-\eps_h^{1/8}\big)^2},\frac{1}{|v_h|^2}\bigg\}-1\rt) j(v_h)\cdot \nabla^\perp \phi dx\\\nonumber
&\le \int_{\Omega} \lt|\min\bigg\{\frac{1}{\big(1-\eps_h^{1/8}\big)^2},\frac{1}{|v_h|^2}\bigg\}-1\rt| |j(v_h)| dx\\
&\les \eps_h^{1/8} \|j(v_h)\|_{L^1(\Omega)}\mathop{\longrightarrow}\limits_{h\to\infty}0\,,\label{convmutild}
\end{align}
where in the last step we have used that since $j(v_h)=v_h\wedge \nabla v_h$,
\[
 \|j(v_h)\|_{L^1(\Omega)}\les \|v_h\|_{L^\infty(\Omega)}\|\nabla v_h\|_{L^\infty(\Omega)}\les \|v_h\|_{L^\infty(\Omega)} E_{\eps_h}^{1/2}(v_h)\stackrel{\eqref{asymptenerg}}{\les} \|v_h\|_{L^\infty(\Omega)} |\log \eps_h|^{1/2}.
\]
Given $r\in \big[{ c}_0\eps_h^{1/4},1/6\big]$, we set
$$A^{\eps_h}_r:=\big\{t\in[1/2,1-r]: \partial B_t(\mu)\cap V^{\eps_h}_r=\emptyset\big\}\,.$$
By item {\it (i)} in Theorem ~\ref{thmlwdSS}, we have $|A^{\eps_h}_r|\geq 1/2-2r\geq 1/6$.  Then, for each $k=1,\ldots,md$, we define a function $\zeta_k\in C^{0,1}(\Omega)$ compactly supported in $\Omega$ by setting 
$$\zeta_k(x):=\int^1_{\min(1,|x-x_k|)} \chi_{A^{\eps_h}_r}(t)\,dt\,. $$
Notice that $\|\zeta_k\|_{C^{0,1}(\Omega)}\leq 2$. Since $\widetilde \mu_h\to\mu$ in $(C^{0,1}_0(\Omega))^*$ by~\eqref{convmutild}, we have $\|\widetilde \mu_h-\mu\|_{(C^{0,1}_0(\Omega))^*}\leq \pi/12$ for $\eps_h$ small enough.  Consequently, by definition of $\sigma_0$ and for $\eps_h$ small,
$$\langle\widetilde\mu_h,\zeta_k\rangle\geq \langle\mu,\zeta_k\rangle -2\|\widetilde \mu_h-\mu\|_{(C^{0,1}_0(\Omega))^*}\geq 2\pi|A^{\eps_h}_r|-\pi/6\geq \pi/6\,,$$
for each $k=1,\ldots,md$. Moreover, using that for $t\in A^{\eps_h}_r$ and $x\in \partial B_t(x_k)$, $\widetilde{v}_h=\frac{v_h}{|v_h|}$, we have 
\begin{multline*}
\langle\widetilde\mu_h,\zeta_k\rangle=-\int_{B_1(x_k)} j(\widetilde v_h)\cdot\nabla^\perp\zeta_k\,dx=\int_0^1 \chi_{A^{\eps_h}_r}(t) \left(\int_{\partial B_t(x_k)} j(\widetilde v_h)\cdot\tau\,d\mathcal{H}^1\right)\,dt\\
= \int_{A^{\eps_h}_r} \left(\int_{\partial B_t(x_k)} j\lt(\frac{v_h}{|v_h|}\rt)\cdot\tau\,d\mathcal{H}^1\right)\,dt= 2\pi \int_{A^{\eps_h}_r}  {\rm deg}\big(v_h,\partial B_t(x_k)\big)\,dt\,,
\end{multline*}
and we conclude that for $\eps_h$ sufficiently small  (independently of $r$), 
$$ \int_{A^{\eps_h}_r}  {\rm deg}\big(v_h,\partial B_t(x_k)\big)\,dt\geq 1/12 \quad\text{for each $k=1,\ldots,md$}\,.$$
Hence, for each $k=1,\ldots,md$, there exists a radius $\rho_h^k\in A^{\eps_h}_r$ such that $ {\rm deg}\big(v_h,\partial B_{\rho_h^k}(x_k)\big)\neq 0$ whenever $\eps_h$ is small enough (independently of $r$). In turn, it implies the existence,  for each $k=1,\ldots,md$, of an element  $B_h^k(r)\in \mathcal{B}_r$ such that 
$B_h^k(r)\subset B_{\rho_h^k}(x_k) \subset \Omega_{\eps_h}$ and ${\rm deg}\big(v_h, \partial B_h^k(r)\big)\neq 0$, whenever $\eps_h$ is small. By the very  definition of $D_r$, we infer that~\eqref{lwdDr} holds for $\eps_h$ small (independently of $r$).  

Combining~\eqref{asymptenerg} and~\eqref{bndtotD}, we deduce that $D_r\leq C$ for some constant  $C$ independent of $\eps_h$ and $r$. Then,~\eqref{genlwden} yields for $\eps_h$ small enough, 
$$E_{\eps_h}(v_h, V_r^{\eps_h})\geq \pi m d \log\left(\frac{r}{\eps_h}\right) -C\quad\text{for every $r\in \big[{ c}_0\eps_h^{1/4},1/6\big]$}\,, $$
where $C$ is still a constant independent of $r$ and $\eps_h$. In view of~\eqref{asymptenerg}, we thus have 
\begin{equation}\label{upbdreps}
E_{\eps_h}(v_h, \Omega \setminus V_r^{\eps_h})\leq \pi m d |\log r| +C\quad\text{for every $r\in \big[{c}_0\eps_h^{1/4},1/6\big]$}\,. 
\end{equation}
Now we define on the set $\{|v_{h}|>0\}$ the map $\widehat v_{h}:=v_h/|v_h|$. Given   $r\in \big[{c}_0\eps_h^{1/4},1/12\big]$, we have $\big||v_h|-1\big|\leq \eps_h^{1/8}$ on $V_{2r}^{\eps_h} \setminus V_r^{\eps_h}$, and we can apply  \cite[Proposition 4.2]{SS} to deduce that  for $\eps_h$ sufficiently small (independently of $r$),  
$$\frac{1}{2}\int_{V_{2r}^{\eps_h} \setminus V_r^{\eps_h}}|\nabla \widehat v_h|^2\,dx\geq \sum_{k=1}^{md}\frac{1}{2}\int_{B_h^k(2r) \setminus V_r^{\eps_h}}|\nabla \widehat v_h|^2\,dx\geq  \pi md \log 2\,.$$
Therefore, if $\eps_h$ is small, using that 
\[
 |\nabla v_h|^2=|\nabla |v_h||^2+|v_h|^2|\nabla \widehat{v}_h|^2\ge |v_h|^2|\nabla \widehat{v}_h|^2;
\]
we obtain
\begin{equation}\label{dyadlwbd}
\frac{1}{2}\int_{V_{2r}^{\eps_h} \setminus V_r^{\eps_h}}|\nabla v_h|^2\,dx\geq \frac{1}{2} \int_{V_{2r}^{\eps_h} \setminus V_r^{\eps_h}}|v_h|^2|\nabla \widehat v_h|^2\,dx\geq \pi md \log 2-C\eps_h^{1/8} \,,
\end{equation}
for some constant $C$ independent of $r$ and $\eps_h$. Then set for $j\in \N$, $r_j:=2^{-j}/6$ and define 
$$J_h:=\max\big\{j\in \N: r_j\geq { c}_0\eps_h^{1/4}\big\}\,.$$
Using the fact that $V_{r_{j+1}}^{\eps_h}\subset V_{r_j}^{\eps_h}$, estimate~\eqref{upbdreps} leads to
\begin{equation}\label{dyadupbd}
\sum_{j=0}^{J_h-1} \frac{1}{2}\int_{V_{r_j}^{\eps_h} \setminus V_{r_{j+1}}^{\eps_h}}|\nabla v_h|^2\,dx\leq E_{\eps_h}(v_h, \Omega \setminus V_{r_{J_h}}^{\eps_h}) \leq  \big(\pi md \log 2\big)J_h +C \,.
\end{equation}
Since  $J_h=O(|\log\eps_h|)$, we infer from~\eqref{dyadlwbd} and~\eqref{dyadupbd} that 
\begin{equation}\label{estisupdyad}
\int_{V_{r_j}^{\eps_h} \setminus V_{r_{j+1}}^{\eps_h}}|\nabla v_h|^2\,dx \leq C(1+J_h\eps_h^{1/8} )\leq C \quad \text{for every $j=0,\ldots,J_h-1$}\,, 
\end{equation}
for a constant $C$ independent of $\eps_h$. 

Finally, fix an arbitrary $p\in [1,2)$. Noticing that $|V^{\eps_h}_{r_j}|=O(r_j^2)$, we  estimate by means of~\eqref{upbdreps},~\eqref{estisupdyad}, and H\"older's inequality, 
\begin{multline*}
\int_{\Omega \setminus V^{\eps_h}_{r_{J_h}}} |\nabla v_h|^p\,dx \leq\int_{\Omega \setminus V^{\eps_h}_{r_{0}}} |\nabla v_h|^p\,dx+\sum_{k=0}^{J_h-1} \int_{V_{r_j}^{\eps_h} \setminus V_{r_{j+1}}^{\eps_h}}|\nabla v_h|^p\,dx\\
\leq  C\left(1+\sum_{k=0}^{J_h-1}r_j^{2-p} \left(\int_{V_{r_j}^{\eps_h} \setminus V_{r_{j+1}}^{\eps_h}}|\nabla v_h|^2\,dx\right)^{p/2}\right)\leq \frac{C}{2^{2-p}-1} \,,
\end{multline*}
for some constant $C$ independent of $\eps_h$ (and $p$). Since, 
$$\int_{ V^{\eps_h}_{r_{J_h}}} |\nabla v_h|^p\,dx \leq C r^{2-p}_{J_h}\left(\int_{\Omega} |\nabla v_h|^2\,dx \right)^{p/2}\leq C  {\eps_h}^{\frac{2-p}{4}}|\log\eps_h|^{p/2}\leq C\,,$$
 we conclude that $\{v_h\}$ is indeed bounded in $W^{1,p}(\Omega)$. 
\end{proof}

\begin{proof}[Proof of Theorem~~\ref{compAlicPon}]
 In view of~\eqref{asymptenerg}, we can  apply \cite[Theorem 6.1]{AlicPon} to find a  subsequence such that $\mu_h\mathop{\rightharpoonup}\limits^* \mu$ weakly* in $(C^{0,1}_0(\Omega))^*$ for some measure $\mu=2\pi\sum_{k=1}^{md} \delta_{x_k} \in\mathcal{A}_d$. 
Moreover, for a radius $r$ satisfying
$$0<r\leq\sigma_0:=\frac{1}{4}\min\Big\{1, \min_k {\rm dist}(x_k,\partial\Omega), \min_{k\neq l}|x_k-x_l|\Big\} \,,$$
 estimate~\eqref{lowerbBGAlicPon} holds  by \cite[Theorem 4.1]{AlicPon}. Consequently,
\begin{equation}\label{bdsigm}
E_{\eps_h}\big(v_h,\Omega  \setminus B_r(\mu)\big) \leq \pi md|\log r| +C\,,
\end{equation}
for a constant $C$ independent of $r$ and $\eps_h$. 
As a consequence of~\eqref{bdsigm}, we can extract a further subsequence  such that $v_h\rightharpoonup v_0$ weakly in $W^{1,2}_{\rm loc}( \Omega \setminus{\rm spt}\,\mu)$
for some $v_0\in W^{1,2}_{\rm loc}( \Omega \setminus{\rm spt}\,\mu;\Ss^1)$. By lower semi-continuity we have 
\begin{equation}\label{lwdv0}
 \liminf_{h\to\infty}E_{\eps_h}\big(v_h,\Omega  \setminus B_r(\mu)\big) \geq \frac{1}{2}\int_{\Omega \setminus B_r(\mu)} |\nabla v_0|^2\,dx\,.
 \end{equation}
In addition, from the continuity of the trace operator we deduce that $v_0=g^m$ on $\partial \Omega$. Thanks to Lemma~~\ref{lemboundWq}, $\{v_h\}$ is also bounded in $W^{1,p}(\Omega)$ for every $1\le p<2$ so that 
$v_h\rightharpoonup v_0$ weakly in $W^{1,p}(\Omega)$ for every $p\in[1,2)$. From this convergence, we easily derive 
$$\langle \mu_h,\zeta\rangle= -\int_{\Omega}j(v_h)\cdot\nabla^\perp\zeta\,dx\mathop{\longrightarrow}\limits_{h\to\infty}  -\int_{\Omega}j(v_0)\cdot\nabla^\perp\zeta\,dx= \langle \curl\,j(v_0),\zeta\rangle\,,$$
for every $\zeta\in \mathcal{D}(\Omega)$, and thus $\curl\,j(v_0)=\mu$ in $\mathcal{D}^\prime(\Omega)$. 
Combining~\eqref{bdsigm} with~\eqref{lwdv0} yields
$$\limsup_{r\downarrow0} \left\{\frac{1}{2}\int_{\Omega \setminus B_r(\mu)} |\nabla v_0|^2\,dx-\pi m d |\log r|\right\}<\infty\,.$$
Hence, we are now in position to apply Proposition~~\ref{keypropcomp} to conclude that $v_0=e^{i\vhi}v_\mu$ for some $\vhi\in W^{1,2}(\Omega)$ satisfying $e^{i\vhi}=1$ on $\partial\Omega$, and the proof is complete.
\end{proof}

%%%%%%%%%%%%%%%%%%%%%%%%%%%%%%%%%%%%%%%%%%%%%%%%%%%%%%%
%%%%%%%%%%%%%%%%%%%%%%%%%%%%%%%%%%%%%%%%%%%%%%%%%%%%%%%
   								       						%%%%%%%%%%%%%%%%%%%%
\section{The $\Gamma$-convergence results}\label{S3}			      %%%%%%%%%%%%%%%%%%%%
								 						%%%%%%%%%%%%%%%%%%%
%%%%%%%%%%%%%%%%%%%%%%%%%%%%%%%%%%%%%%%%%%%%%%%%%%%%%%%
%%%%%%%%%%%%%%%%%%%%%%%%%%%%%%%%%%%%%%%%%%%%%%%%%%%%%%%

In this section, our main objective is to determine the $\Gamma$-limit of the functional $F^\eta_\eps$ defined in~\eqref{defFetaeps&F0eps}  as $\eta\downarrow0$ and $\eps\downarrow0$. We introduce $\widetilde F^\eta_\eps:L^1(\Omega)\times L^1(\Omega)\to (-\infty,\infty]$ given as 
\[ 
\widetilde F_{\eps}^\eta(u,\psi):=\begin{cases}
\displaystyle F_{\eps}^\eta(u,\psi)-\frac{\pi d}{m}|\log\eps| & \text{if $(u,\psi)\in\mathcal{H}_g(\Omega)$}\,,\\
\infty & \text{otherwise}\,, 
\end{cases}
\]
where $F_\eps^\eta$ and the class $\mathcal{H}_g(\Omega)$ are defined in~\eqref{defFetaeps&F0eps},\eqref{classHg} respectively. Throughout this section $\Omega\subset \R^2$ denotes a smooth, bounded, and {\sl simply connected} domain.

In a first part, we shall prove that the domain of the $\Gamma$-limit is determined by the class of functions 
\begin{multline*}
\mathcal{L}_g(\Omega):=\Big\{u\in SBV(\Omega;\Ss^1) : u^m=e^{i\vhi}v_\mu\text{ for some  $\mu\in\mathcal{A}_d$}\\
\text{and $\vhi\in W^{1,2}(\Omega)$ satisfying $e^{i\vhi}=1$ on $\partial\Omega$}\Big\}\,,
\end{multline*}
where $\mathcal{A}_d$ is  the family of measures defined in~\eqref{defAd}, and $v_\mu$ is the canonical harmonic map associated to~$\mu$ through \eqref{definivmu}. 
We emphasize that $\mathcal{L}_g(\Omega)\subset SBV^p(\Omega;\Ss^1)$ for every $p\in[1,2)$ by Corollary ~\ref{cor1}. 
In turn, the $\Gamma$-limit is given by the functional $F_{0,g}:\mathcal{L}_g(\Omega)\to \R$ defined by  
\begin{equation*}
 F_{0,g}(u):=E_0(u)+\mathcal{H}^1(J_u)
 +\mathcal{H}^1\big(\{u\neq g\}\cap\partial\Omega\big)\,,  
 \end{equation*}
 where we have set for $u^m =: e^{i\vhi} v_\mu$, 
 \begin{equation*}
 E_0(u):= \displaystyle \frac{1}{2m^2}\int_\Omega|\nabla\vhi|^2\,dx+\frac{1}{m^2}\mathbb{W}(\mu) 
+m d\bgamma_m\,. 
 \end{equation*}
In the  expression above, $\bgamma_m$ is a structural constant which is usually interpreted as the core energy of a singularity. In our context, it is defined as 
\begin{multline}\label{defbgammVnew}
\bgamma_m:=\lim_{R\to\infty} \min\Bigg\{E_1(w,B_R) -\frac{\pi}{m^2}\log R
 :  w\in W^{1,2}(B_R;\mathcal{N})\,,\\
 w(z)=\frac{1}{m}\Big(\frac{z}{|z|}, \sqrt{m^2-1}\Big) \text{ on $\partial B_R$}\Bigg\}\,.
 \end{multline}
 Existence and finiteness of this limit follows from a classical comparison argument (see Lemma ~\ref{bdinfconstgam}, and \cite[Lemma III.1]{BBH}). 
 We also note that the value of $F_{0,g}(u)$ only depends on $u$ and not on a particular representation $u^m=e^{i\vhi}v_\mu$. Indeed, one always has $\mu={\rm curl}\,j(u^m)$ and  
 $|\nabla \vhi|=|\nabla(\overline v_\mu u^m)|$.  
 \vskip3pt
 
To properly state the $\Gamma$-convergence result, it is now convenient to introduce $\widetilde F_0:L^1(\Omega)\times L^1(\Omega)\to (-\infty,\infty]$ given by
$$\widetilde F_0(u,\psi):=\begin{cases}
F_{0,g}(u) & \text{if $u\in \mathcal{L}_g(\Omega)$ and $\psi\equiv 1$}\,,\\
\infty & \text{otherwise}\,.
\end{cases}$$

\begin{theorem}\label{Gammadiff}
Let $\eps_h\downarrow0$ and $\eta_h\downarrow0$ be arbitrary sequences. The sequence of functionals $\big\{\widetilde F_{\eps_h}^{\eta_h}\big\}$ $\Gamma$-converges in the strong $\big[L^1(\Omega)\big]^2$-topology  to $\widetilde F_0$ as $h\to\infty$. More precisely:    
\vskip3pt

\begin{enumerate}
 \item[(i)] If $\{(u_h,\psi_h)\}\subset\mathcal{H}_g(\Omega)$, $\{u_h\}$ is bounded in $L^\infty(\Omega)$, and  $\sup_h \widetilde F_{\eps_h}^{\eta_h} (u_h,\psi_h)<\infty$, then there exist a  subsequence and $u\in \mathcal{L}_g(\Omega)$ with $u^m=:e^{i\vhi}v_\mu\,$ such that 
 $(u_h,\psi_h)\to (u,1)$ in $L^1(\Omega)$, $v_h:=\p(u_h)\rightharpoonup u^m$ weakly in $W^{1,p}(\Omega)$ for every $p<2$ and weakly in $W^{1,2}_{\rm loc}(\overline\Omega \setminus{\rm spt}\,\mu)$, and 
 the measures $\mu_h:=\curl j(v_h)$ weakly* converge to $\mu= \curl j(u^m)$ in the $(C^{0,1}_0(\Omega))^*$ topology. 
 \vskip3pt
 
\item[(ii)] Under the conclusions of {\it (i)},  
\begin{equation}\label{TheoliminfE}
 \liminf_{h\to \infty}  \left\{ E_{\eps_h} \big({\rm P}(u_h)\big)-\frac{\pi d}{m}|\log\eps_h|\right\}\geq E_0(u)\,,
\end{equation}
and 
\begin{equation}\label{TheoliminfI}
 \liminf_{h\to \infty}  I_{\eta_h} (\psi_h)\geq \mathcal{H}^1(J_u)+\mathcal{H}^1\big(\{u\neq g\}\cap\partial\Omega\big)\,.
\end{equation}
Moreover, if $\widetilde F_0(u,1)=\lim_h\widetilde F_{\eps_h}^{\eta_h} (u_h,\psi_h)<\infty$, then $\p(u_h)\to u^m$ strongly in $W^{1,p}(\Omega)$ for every $p<2$ and strongly  in $W^{1,2}_{\rm loc}(\overline\Omega \setminus{\rm spt}\,\mu)$, 
 \begin{equation}\label{strongconvergence}
 \lim_{h\to\infty} E_{\eps_h}\big({\rm P}(u_h),\Omega \setminus B_r(\mu)\big)= \frac{1}{2m^2}\int_{\Omega \setminus B_r(\mu)}|\nabla(u^m)|^2\,dx\quad\text{for every $r>0$}\,,
 \end{equation}
and 
 \begin{equation}\label{strongconvergencebis}
\lim_{h\to \infty} I_{\eta_h}\big(\psi_h\big)=\mathcal{H}^1(J_u)+\mathcal{H}^1\big(\{u\neq g\}\cap\partial\Omega\big)\,.
 \end{equation}
\item[(iii)]  For every $u\in \mathcal{L}_g(\Omega)$,  
there exists a sequence $\{(u_h,\psi_h)\}\subset \mathcal{H}_g(\Omega)$ such that  $u_h=g$ on $\partial\Omega$, 
$(u_h,\psi_h)\to (u,1)$ in $L^1(\Omega)$, $\p(u_h)\to u^m$ strongly in $W^{1,p}(\Omega)$  for every $p<2$ and strongly in $W^{1,2}_{\rm loc}(\overline\Omega \setminus{\rm spt}\,\mu)$, and satisfying 
\begin{align}
\label{theolimsup01}
 & \lim_{h\to \infty}\,\left\{ E_{\eps_h}\big({\rm P}(u_h)\big) -\dfrac{\pi d}m|\log \eps_h|\right\}  = E_0(u)\,,\\
  \label{theolimsup03}
&\lim_{h\to \infty}  I_{\eta_h}(\psi_h)=\mathcal{H}^1(J_u)+\mathcal{H}^1\big(\{u\neq g\}\cap\Omega\big)\,.
\end{align}
\end{enumerate}
\end{theorem}
\vskip5pt
We can proceed analogously with the sharp interface  functionals $F_{\eps,g}^0$ defined in~\eqref{defFeps0g}, and introduce  
$\widetilde F_\eps^0:L^1(\Omega)\to (-\infty,\infty]$ and $\widetilde F_0:L^1(\Omega)\to (-\infty,\infty]$ defined as
$$\widetilde F_\eps^0(u):=\begin{cases}
\displaystyle F_{\eps,g}^0(u)-\frac{\pi d}{m}|\log\eps| & \text{if $u\in \mathcal{G}_g(\Omega)$}\,,\\[5pt]
\infty & \text{otherwise}\,, 
\end{cases}
\qquad 
\widetilde F_0(u):=\begin{cases}
F_{0,g}(u) & \text{if $u\in \mathcal{L}_g(\Omega)$}\,,\\
\infty & \text{otherwise}\,.
\end{cases}
 $$ 

\begin{theorem}\label{Gammasharp}
Let $\eps_h\downarrow0$  be an arbitrary sequence. The sequence of functionals $\big\{\widetilde F_{\eps_h}^{0}\big\}_{h\in \N}$ $\Gamma$-converges
in the strong $L^1(\Omega)$-topology  to $\widetilde F_0$ as $h\to\infty$. More precisely:
\vskip3pt

\begin{enumerate}
 \item[(i)] If $\{u_h\}\subset\mathcal{G}_g(\Omega)$, $\{u_h\}$ is bounded in $L^\infty(\Omega)$, and  $\sup_h \widehat F_{\eps_h}^0 (u_h)<\infty$, then there exist a  subsequence and $u\in \mathcal{L}_g(\Omega)$ with $u^m=:e^{i\vhi}v_\mu\,$ such that 
 $u_h\to u$ in $L^1(\Omega)$, $v_h:=\p(u_h)\rightharpoonup u^m$ weakly in $W^{1,p}(\Omega)$ for every $p<2$ and weakly in $W^{1,2}_{\rm loc}(\overline\Omega \setminus{\rm spt}\,\mu)$, and 
 the measures $\mu_h:=\curl j(v_h)$ weakly* converge to $\mu= \curl j(u^m)$ in the $(C^{0,1}_0(\Omega))^*$ topology. 
 \vskip3pt
 
\item[(ii)] If $\{u_h\}\subset\mathcal{G}_g(\Omega)$ is such that $u_h\to u$ in $L^1(\Omega)$, then  
\begin{equation}\label{Theoliminfshp}
 \liminf_{h\to \infty} \widetilde F_{\eps_h}^{0} (u_h)\geq \widetilde F_0(u)\,.
\end{equation}

Moreover, if $\widetilde F_0(u)=\lim_h\widetilde F_{\eps_h}^{0} (u_h)<\infty$, then $\p(u_h)\to u^m$ strongly in $W^{1,p}(\Omega)$ for every $p<2$ and strongly in $W^{1,2}_{\rm loc}(\overline\Omega \setminus{\rm spt}\,\mu)$, 
identity \eqref{strongconvergence} holds, and for every open set $A\subset \R^2$ such that $\mathcal{H}^1\big(J_u \cap(\Omega\cap \partial A)\big)+\mathcal{H}^1\big(\{u\neq g\}\cap\partial\Omega\cap \partial A\big)=0$,
 \begin{multline}\label{strongconvergencebisshp}
\lim_{h\to \infty} \mathcal{H}^1\big(J_{u_h} \cap (\Omega\cap A)\big)+\mathcal{H}^1\big(\{u_h\neq g\}\cap\partial\Omega\cap A\big)\\
=\mathcal{H}^1\big(J_u \cap(\Omega\cap A)\big)+\mathcal{H}^1\big(\{u\neq g\}\cap\partial\Omega\cap A\big).
 \end{multline}

\item[(iii)]  For every $u\in \mathcal{L}_g(\Omega)$,   
there exists a sequence $\{u_h\}\subset \mathcal{G}_g(\Omega)$ such that $u_h\to u$ in $L^1(\Omega)$ and 
\begin{equation}\label{theolimsup01shp}
 \lim_{h\to \infty} \widetilde F_{\eps_h}^{0}(u_h) = \widetilde F_0(u)\,.
\end{equation}

\end{enumerate}
\end{theorem}

As a standard consequence of these $\Gamma$-convergence results, we have the following corollaries concerning the minimizers of $F_{\eps}^\eta$ and $F^0_{\eps,g}$, whose existence was proved in Theorems ~\ref{exist} \& ~\ref{existsharp} respectively (together with the uniform $L^\infty$-bound allowing for compactness). 

\begin{corollary}
Let $\eps_h\downarrow 0$  and $\eta_h\downarrow0$ be  arbitrary sequences. For each $h\in\N$, let $(u_h,\psi_h)$  be a minimizer of $F_{\eps_h}^{\eta_h}$ in $\mathcal{H}_g(\Omega)$. There exists a  subsequence and a map $u$ minimizing $F_{0,g}$ over $\mathcal{L}_g(\Omega)$ such that $(u_h,\psi_h)\to (u,1)$ in $L^1(\Omega)$, $\p(u_h)\to u^m$ strongly in $W^{1,p}(\Omega)$ for every $p<2$ and strongly in $W^{1,2}_{\rm loc}(\overline\Omega \setminus{\rm spt}\,\mu)$ where $\mu:={\rm curl}\,j(u^m)$.  In addition, 
$$F^{\eta_h}_{\eps_h}(u_h,\psi_h)= \frac{\pi d}{m}|\log\eps_h| +F_{0,g}(u)+o(1)\quad\text{as $h\to\infty$}\,.$$
\end{corollary}

\begin{corollary}\label{corcvminsharp}
Let $\eps_h\downarrow 0$ be an arbitrary sequence. For each $h\in\N$, let $u_h$  be a minimizer of $F_{\eps_h,g}^0$ in $\mathcal{G}_g(\Omega)$. There exists a  subsequence and a map $u$ minimizing $F_{0,g}$ over $\mathcal{L}_g(\Omega)$ such that $u_h\to u$ in $L^1(\Omega)$, $\p(u_h)\to u^m$ strongly in $W^{1,p}(\Omega)$ for every $p<2$ and strongly in $W^{1,2}_{\rm loc}(\overline\Omega \setminus{\rm spt}\,\mu)$    where $\mu:={\rm curl}\,j(u^m)$. In addition, 
$$F^{0}_{\eps_h,g}(u_h)= \frac{\pi d}{m}|\log\eps_h| +F_{0,g}(u)+o(1)\quad\text{as $h\to\infty$}\,.$$
\end{corollary}

\begin{remark}\label{remminF0}
From the definition of $F_{0,g}$, any minimizer $u$ of $F_{0,g}$ over $\mathcal{L}_g(\Omega)$ satisfies $u^m=v_\mu$ where $\mu:={\rm curl}\,j(u^m)$ (i.e., in any  representation $u^m=e^{i\vhi}v_\mu$, the phase $\vhi$ is a constant multiple of $2\pi$).  As a consequence, 
$$ F_{0,g}(u)=\frac{1}{m^2}\mathbb{W}(\mu) 
+m d\bgamma_m+\mathcal{H}^1(J_{u})
 +\mathcal{H}^1\big(\{u\neq g\}\cap\partial\Omega\big)\,.$$
\end{remark}
\vskip5pt

The rest of this section if devoted to the proofs of Theorems ~\ref{Gammadiff} \& ~\ref{Gammasharp}. Starting with Theorem ~\ref{Gammadiff}, compactness, $\Gamma$-$\liminf$, and $\Gamma$-$\limsup$ parts are proved in Subsections  ~\ref{subseccomp}, ~\ref{sec:gammainf}, and  ~\ref{sec:gammasup} respectively. The proof of Theorem ~\ref{Gammasharp} is the object of Subsection ~\ref{secproofthmgamshp}.

%%%%%%%%%%%%%%%%%%%%%%%%%%%%%%%%%%%%%%%%%%%%%%%%%%%%%%%

\subsection{Proof of Theorem ~\ref{Gammadiff}~{\it(i)}: Compactness}\label{subseccomp}

\begin{proposition}\label{compactthmfield}
Let $\eps_h\downarrow0$ and $\eta_h\downarrow0$ be arbitrary sequences. Let $\{(u_h,\psi_h)\}\subset\mathcal{H}_g(\Omega)$ be such that $\{u_h\}$ is bounded in $L^\infty(\Omega)$, and  
\begin{equation}\label{energbdassump}
\sup_h \left\{F_{\eps_h}^{\eta_h}(u_h,\psi_h)-\frac{\pi d}{m}|\log\eps_h|\right\}<\infty\,. 
\end{equation}
Then, there exist a  subsequence and $u\in \mathcal{L}_g(\Omega)$ with $u^m=:e^{i\vhi}v_\mu$ such that  
\begin{enumerate}
\item[(i)] $(u_h,\psi_h)\to (u,1)$ strongly in $L^1(\Omega)$; 
\vskip3pt
\item[(ii)] $v_h:=\p(u_h) \rightharpoonup u^m $ weakly in $W^{1,p}(\Omega)$ for every $p\in [1,2)$;
\vskip3pt
\item[(iii)] $v_h \rightharpoonup u^m $ weakly in $W^{1,2}_{\rm loc}(\overline\Omega \setminus{\rm spt}\,\mu)$
\end{enumerate}
Moreover,  $\mu_h:= \curl\,j(v_h)\in L^1(\Omega)$ converges to $\mu=\curl\,j(u^m)$ in the weak* topology of~$(C^{0,1}_0(\Omega))^*$. 
\end{proposition}

The proposition above partially rests on the following preliminary lemma. 

\begin{lemma}\label{precompactresult}
Let $\{(u_h,\psi_h)\}\subset \mathcal{H}(\Omega)$ be such that
$\psi_h\to 1$ a.e. in $\Omega$. Assume that for some $p\in(1,2]$, 
$$\sup_{h}\left\{\|u_h\|_{L^\infty(\Omega)}+\|\psi_h-1\|_{L^2(\Omega)}\|\nabla \psi_h\|_{L^{2}(\Omega)}+\big\|\nabla \p( u_h)\big\|_{L^{p}(\Omega)}\right\}<\infty\,.  $$
Then there exist a  subsequence and $u\in SBV^p(\Omega)$ such that $\P(u)\in W^{1,p}(\Omega;\mathcal{N})$, $u_h\to u$ strongly in  $L^1(\Omega)$, and 
$\P(u_h)\rightharpoonup\P(u)$ weakly in $W^{1,p}(\Omega)$. 
\end{lemma}

\begin{proof}
By assumption and Cauchy-Schwarz inequality, we have  
$$\int_\Omega(1-\psi_h)|\nabla\psi_h|\,dx\leq C\,,$$ 
for some constant $C$ independent of $h$. 
According to the co-area formula (see \cite[Theorem~3.40]{AFP}), 
$$\int_\Omega(1-\psi_h)|\nabla\psi_h|\,dx=\int_0^1(1-t)\mathcal{H}^1\big(\partial\{\psi_h<t\}\cap\Omega\big)\,dt
\geq \int_{1/4}^{3/4}(1-t)\mathcal{H}^1\big(\partial\{\psi_h<t\}\cap\Omega\big)\,dt \,.$$
Therefore, we can find a level $t_h\in(1/4,3/4)$ such that 
$$\int_\Omega(1-\psi_h)|\nabla\psi_h|\,dx\geq \frac{1}{4}\mathcal{H}^1(\partial E_h\cap\Omega) \,,\quad \text{with } E_h:=\{\psi_h<t_h\}\,.$$
Notice that $|E_h|\to 0$ since $\psi_h\to 1$ a.e. in $\Omega$. 

Let us now define 
$$\widetilde u_h:= (1-\chi_{E_h})u_h\,. $$
With  our choice of $E_h$, we have that 
$(1-\chi_{E_h})/\psi_h\in SBV^p(\Omega)\cap L^\infty(\Omega)$. 
Since $\psi_hu_h\in W^{1,2}(\Omega)\cap L^\infty(\Omega)$, we deduce that  $\widetilde u_h=(\psi_hu_h)(1-\chi_{E_h})/\psi_n\in SBV^p(\Omega)\cap L^\infty(\Omega)$ with 
 $J_{\widetilde u_h}\, \subset \,\partial E_h$. Since $\P(\widetilde u_h)=(1-\chi_{E_h})\P(u_h)\in SBV(\Omega;\R^3)$, we infer that 
 $$|\nabla \widetilde u_h|= \big|\nabla(\P(\widetilde u_h))\big|=(1-\chi_{E_h})\big|\nabla(\P(u_h))\big|\leq  \big|\nabla(\P(u_h))\big| \quad\text{a.e. in $\Omega$}\,.$$
 Consequently,
\begin{equation}\label{SBVbnd}
\sup_{h}\big\{\|\widetilde u_h\|_{L^\infty(\Omega)}+\|\nabla \widetilde u_h\|_{L^p(\Omega)} +\mathcal{H}^1(J_{\widetilde u_h}) \big\}<\infty\,. 
\end{equation}
Now select a  subsequence such that $\P(u_h) \rightharpoonup w$ weakly in $W^{1,p}(\Omega)$. In view of~\eqref{SBVbnd}, we can  apply Ambrosio's compactness theorem in $SBV$ 
(see e.g. \cite[Theorem 4.8 \& Remark 4.9]{AFP}) to find a further subsequence such that $\widetilde u_h\to u$ strongly in $L^1(\Omega)$ for some
$u\in SBV^p(\Omega)\cap L^\infty(\Omega)$. Then,  
$$\|u_h-u\|_{L^1(\Omega)}\leq \|\widetilde u_h-u\|_{L^1(\Omega)}+ \|u_h\|_{L^\infty(\Omega)}|E_h|\mathop{\longrightarrow}\limits_{h\to\infty} 0\,.$$
 Since $\P$ is $1$-Lipschitz, we have $\|\P(u_h)-\P(u) \|_{L^1(\Omega)}\leq  \|u_h-u\|_{L^1(\Omega)}$, and thus $w=\P(u)$. 
 \end{proof}

\begin{proof}[Proof of Proposition ~\ref{compactthmfield}]
Let us first recall~\eqref{GLcone}, that is $E_{\eps_h}(\P(u_h))= G_{\eh}(v_h)$ with $v_h:=\p(u_h)$. 
In view of~\eqref{GLcone}, assumption~\eqref{energbdassump}  implies 
\begin{equation}\label{ineqenerGL}
E_{\eps_h}(v_h) \leq \pi m d |\log\eps_h| +C \,, 
\end{equation}
that is ~\eqref{asymptenerg} holds.
In turn, we deduce from~\eqref{energbdassump} that
\begin{equation*}
\|\psi_h-1\|_{L^2(\Omega)}\|\nabla\psi_h\|_{L^2(\Omega)}\leq \frac{\eta}{2}\int_{\Omega} |\nabla \psi_h|^2 dx +\frac{1}{2\eta}\int_{\Omega} (1-\psi_h)^2 dx= I_{\eta_h}(\psi_h)\leq C\,,
\end{equation*}
for a constant $C$ independent of $\eta_h$. Clearly, it implies that $\psi_h\to 1$ a.e. $\Omega$, at least for a suitable subsequence. We are thus in position to apply first Theorem~~\ref{compAlicPon} to $\{v_h\}$, and then Lemma~~\ref{precompactresult} to $\{(u_h,\psi_h)\}$ to conclude the proof. 
\end{proof}

\begin{remark}\label{remcomp}
We emphasize that, in addition to the conclusions of Proposition~~\ref{compactthmfield}, assumption~\eqref{energbdassump} implies $\sup_h I_{\eta_h} (\psi_h)<\infty$. 
\end{remark}

%%%%%%%%%%%%%%%%%%%%%%%%%%%%%%%%%%%%%%%%%%%%%%%%%%%%%%%

\subsection{Optimal profiles and the constant $\bgamma_m$}\label{secoptprof}

In this subsection, we first study the core energy associated with one vortex in the Ginzburg-Landau energies $\{E_\eps\}$.  We consider for $R>0$ and $\eps>0$,  the minimum value
\begin{multline}\label{defgammameps}
\gamma_m(\eps,R):= \min\Bigg\{E_\eps(w,B_R) -\frac{\pi}{m^2}\log \frac{R}{\eps}
 :  w\in W^{1,2}(B_R;\mathcal{N})\,,\\
 w(z)=\frac{1}{m}\Big(\frac{z}{|z|}, \sqrt{m^2-1}\Big) \text{ on $\partial B_R$}\Bigg\}\,.
\end{multline}
In view of identity~\eqref{defGeps} defining the functional $G_\eps$, the value $\bgamma_m(\eps,R)$ can be written as 
$$\gamma_m(\eps,R)= \min\Bigg\{G_\eps(v,B_R) -\frac{\pi}{m^2}\log \frac{R}{\eps}  :  v\in W^{1,2}(B_R)\,,\;v(z)=\frac{z}{|z|} \text{ on $\partial B_R$}\Bigg\}\,.$$
Notice that, by homogeneity,
\begin{equation}\label{homogenbgamma}
 \gamma_m(\eps,R)=\gamma_m(1,R/\eps)=:\gamma_m(R/\eps)\,.
\end{equation}

We start by proving that $\gamma_m$ admits a limit as $R\to \infty$. This will be needed for the lower bounds \eqref{TheoliminfE} and \eqref{Theoliminfshp}.

\begin{lemma}\label{bdinfconstgam}
The function $R\mapsto\gamma_m(R)$ is non increasing and  the limit 
\begin{equation}\label{defgammam}
\bgamma_m:=\lim_{R\to \infty} \gamma_m(R)
 \end{equation}
is finite.
\end{lemma} 
\begin{proof}
The proof of this lemma closely follows the proof of \cite[Lemma III.1]{BBH}.  Let us first show that $\gamma_m$ is non increasing. Let $0<R_1<R_2$, and consider an admissible competitor $w$ for the minimization problem defining $\gamma_m(R_1)$. We extend $w$ by $0$-homogeneity in the annulus  $B_{R_2} \setminus B_{R_1}$, i.e., $w(x)=w(R_1x/|x|)$ for $x\in B_{R_2} \setminus B_{R_1}$. 
 By construction $w\in  W^{1,2}(B_{R_2};\mathcal{N})$, and it is an admissible competitor for $\gamma_m(R_2)$. Elementary computations then yield
\[
E_1(w, B_{R_2}) = E_{1}(w;B_{R_1})+\frac{\pi}{m^2}\log\frac{R_2}{R_1}\,.
\]
Hence, $\gamma_m(R_2)\leq E_{1}(w;B_{R_1}) -\frac{\pi}{m^2}\log R_1$. Taking the infimum with respect to $w$ yields $\gamma_m(R_2)\leq \gamma_m(R_1)$, so that 
$\gamma_m$ is indeed non increasing. 
Next, for an arbitrary $w\in W^{1,2}(B_1;\mathcal{N})$ satisfying $w(x)=(1/m)(x, \sqrt{m^2-1})$ on $\partial B_1$, we have $w=(1/m)(v,\sqrt{m^2-1}|v|)$ with $v\in W^{1,2}(B_1)$ satisfying $v(x)=x$ on $\partial B_1$.  Consequently, for $\eps>0$ we have  
\[
E_\eps(w,B_1)=G_\eps(v,B_1)\geq \dfrac1{m^2}E_\eps(v,B_1) \geq \frac{\pi}{m^2} \log \frac{1}{\eps} -C\,,
\]
for some universal constant $C$ by~\cite{BBH}. In view of~\eqref{homogenbgamma}, we infer that $\gamma_m(1/\eps)$ is bounded from below, and thus $\bgamma_m>-\oo$.
\end{proof}
The following lemma  and its subsequent corollary will allow us to construct a recovery sequence close to the vortices.
\begin{lemma}\label{approxgamma} 
 For every $v\in W^{1,2}(B_1)$ satisfying $v(x)=x$ on $\partial B_1$, there exists a sequence $\{u_k\}\subset SBV^2(B_1)$ such that  
 \begin{enumerate}
\item[(i)] $\p(u_k)\in W^{1,2}(B_1)$ and $\p(u_k)(x)=x$ in a neighborhood of $\partial B_1$; 
\item[(ii)] $J_{u_k}\subset\Sigma_k$ where $\Sigma_k$ is a smooth simple curve (i.e., a smooth image of $[0,1]$) contained in $ \overline{B_1}$;
\item[(iii)] $\p(u_k) \to v$ strongly in $W^{1,2}(B_1)$. 
\end{enumerate}
\end{lemma}
   
\begin{proof}
{\it Step 1.} Since the function $\widetilde  v:z\mapsto v(x)-x$ belongs to $W^{1,2}_0(B_1)$,  for each $k\in\N$ we can find $\phi_k\in C^\infty_c(B_1)$ such that 
$\|\widetilde v-  \phi_k\|_{W^{1,2}(B_1)}\leq 2^{-k} $. Implicitly, we extend  $\phi_k$ by $0$ outside $B_1$. Then, we select a sequence of radii $\{r_k\}\subset(3/4,1)$ such that  $r_k\to 1$ as $k\to\infty$, and ${\rm spt}\,\phi_k\subset B_{r_k}$. By Morse-Sard Theorem, we can find $\{c_k\}\subset \C$ with $|c_k|<(1-r_k)^2$ such that $c_k$ is a regular value of the mapping $x\mapsto x+\phi_k(x)$ for each $k\in\N$. Next, we consider for each $k\in\N$ a cut-off function $\chi_k\in C^\infty(\R^2;[0,1])$ satisfying $\chi_k(x)=1$ for $|x|\leq r_k$, $\chi_k(x)=0$ for $|x|\geq (1+r_k)/2$, and $(1-r_k)|\nabla\chi_k|\leq C$ for a constant $C$ independent of $k$. Now we define the smooth function 
$$ v_k(x):=x+\phi_k(x)-c_k \chi_k(x)\,,$$
which satisfies $v_k(x)=x$ in a neighborhood of $\partial B_1$. 
We estimate 
\[
\| v_k - v\|_{W^{1,2}(B_1)} \leq \|\widetilde v- \phi_k\|_{W^{1,2}(B_1)}
  +|c_k| \| \chi_k\|_{W^{1,2}(B_1)} 
 \leq 2^{-k} +C(1-r_k)\,.
 \]
Therefore $v_k\to v$ strongly in $W^{1,2}(B_1)$ as $k\to \infty$. 
\vskip5pt

\noindent{\it Step 2.}  Let us fix an index $k\in\N$. To complete the proof, we shall produce a map $u_k\in SBV^2(B_1)$ such that $\p(u_k)=v_k$ and   $J_{u_k}$ is contained in a closed and smooth simple curve.  First notice that our choice of $c_k$ and the fact that $\phi_k=0$ on $B_1\setminus B_{r_k}$ imply that 
$$|v_k|\geq r_k-|c_k|\geq 11/16\quad \text{in $B_1 \setminus B_{r_k}$}\,,$$
so that $\{v_k=0\}=\{v_k=0\}\cap \overline B_{r_k}=\{x+\phi_k(x)=c_k\}\cap\overline B_{r_k} $ is a finite set. Hence we can find a smooth simple curve $\Sigma_k$ contained in $\overline B_1$ such that $\Sigma_k\cap\partial B_1=\{a_k,b_k\}$ with $a_k$, $b_k$ the distinct endpoints of $\Sigma_k$, $\Sigma_k$ meets $\partial B_1$ orthogonally at $a_k$ and $b_k$, and $\{v_k=0\}\subset \Sigma_k$. In this way, $B_1 \setminus\Sigma_k=A_k^1\cup A_k^2$ where $A_k^1$ and $A_k^2$ are disjoint simply connected open sets with Lipschitz boundary. Since $v_k$ does not vanish in each $A_k^j$, it admits a smooth $m$-th square root $u_k^j$ in each $A_k^j$, which is continuous up to $\partial A_k^j$. We define
$$u_k(x):=u_k^j(x)\quad\text{if $x\in A_k^j$}\,. $$
It is then elementary to check that $u_k\in SBV^2(B_1)$. By construction, items {\it (i)}, {\it (ii)}, and {\it (iii)} hold. 
\end{proof}   

\begin{corollary}\label{corogam0}
Let $\eps_h\downarrow0$ be an arbitrary sequence. There exist $\{u_h\}\subset SBV^2(B_1)\cap L^\infty(\Omega)$ with $\|u_h\|_{L^\infty(\Omega)}\leq 1$, and a sequence of  smooth simple curves $\{\Sigma_h\}\subset \overline{B_1}$ such that $J_{u_h}\subset \Sigma_h$   for every $h\in\N$, $\p(u_h)\in W^{1,2}(\Omega)$, $\p(u_h)(x)=x$ in a neighborhood of $\partial B_1$, and 
\[
\lim_{h\to \infty} \left\{E_{\eps_h}\big(\P(u_h),B_1\big) - \frac\pi{m^2} |\log \eps_h| \right\}= \bgamma_m\,.
\]
\end{corollary}  

\begin{proof}
Let us fix an arbitrary $h\in\N$. Consider $w_h\in W^{1,2}(B_1,\mathcal{N})$ a solution of the minimization problem~\eqref{defgammameps} defining $\gamma_m(\eps_h,1)$ (existence easily follows from the direct method of calculus of variations), and write $w_h=(1/m)(v_h,|v_h|\sqrt{m^2-1})$ with $v_h\in W^{1,2}(B_1)$. We apply Lemma~~\ref{approxgamma} to $v_h$ to produce a sequence $\{u_{k}\}$ and curves $\{\Sigma_k\}$. The convergence property {\it (iii)} in Lemma~~\ref{approxgamma} implies that $E_{\eps_h}\big(\P(u_k)\big)\to E_{\eps_h}(w_h)$ as $k\to\infty$. Hence, we can find $k_h\in\N$ such that, setting $\widetilde u_h:=u_{k_h}$ and $\Sigma_{h}:=\Sigma_{k_h}$,
one has $J_{\widetilde u_h}\subset \Sigma_h$ and $E_{\eps_h}\big(\P(\widetilde u_h)\big)\leq E_{\eps_h}(w_h)+\eps_h$. Setting $u_h:=\widetilde u_h/\max\big(1,|\widetilde u_h|\big)$, 
we observe that  $u_h\in SBV^2(B_1)$, $\|u_h\|_{L^\infty(\Omega)}\leq 1$, and $J_{u_h}\subset J_{\widetilde u_h}\subset \Sigma_h$. As in~\eqref{tronc} we have  $\p(u_h)= \p(\widetilde u_h)/\max\big(1,|\p(\widetilde u_h)|\big)$, and we infer that $\p(u_h)\in W^{1,2}(\Omega)$, $\p(u_h) (x)=x$ in a neighborhood of $\partial B_1$, and 
\[
\gamma_m(\eps_h,1)\leq E_{\eps_h}\big(\P(u_h)\big) \leq E_{\eps_h}\big(\P(\widetilde u_h)\big)\leq E_{\eps_h}(w_h)+\eps_h \,.
\]
Since $E_{\eps_h}(w_h)=\gamma_{m}(\eps_h,1)\to \bgamma_m$ as $h\to\infty$ by~\eqref{homogenbgamma} and Lemma~~\ref{bdinfconstgam}, the conclusion follows. 
\end{proof}

To close this section,  we characterize the optimal one dimensional profile related to the energy~$I_\eta$. This will be important to go from a recovery sequence for the sharp interface functional $F^0_{\eps,g}$ to a recovery sequence of the diffuse interface functional $F^{\eta}_\eps$.

\begin{lemma}\label{profil1Dpsi}
For every  $\eta>0$, 
\[
 \min \left\{ I^{1D}_\eta(\psi):=\dfrac\eta2 \int_\R |\psi'(s)|^2\, ds +\dfrac1{2\eta} \int_\R (1-\psi(s))^2\, ds\, :\, \psi(0)=0\right\}=1\,.
\]
The minimum is uniquely achieved by $\psi_\eta(s):=\psi_\star(s/\eta)$ with $\psi_\star(s)=1-e^{-|s|}$.
\end{lemma}

\begin{proof}
By rescaling, we may assume without loss of generality that $\eta=1$. Write 
$$\Phi(t):=(1-|t|)^2/2=\int_{|t|}^1(1-s)\, ds\,.$$ 
Let $\psi:\R\to\R$ be such that $\psi(0)=0$ and $I^{1D}_1(\psi)<\infty$.
First, notice that the condition $I^{1D}_1(\psi)<\infty$ implies that $\lim_{s\to\pm\infty}\psi(s)=1$. Then, by Cauchy-Schwarz and Young inequalities, we have 
$$I^{1D}_1(\psi)\geq\int_\R |1-\psi(s)|\, |\psi'(s)|\, ds= \int_{-\infty}^0\big|(\Phi\circ \psi)'\big|\,ds+\int^{\infty}_0\big|(\Phi\circ \psi)'\big|\,ds \geq 2(\Phi(0)-\Phi(1)) =1\,.$$
We have equality in the above chain of inequalities if and only if $|\psi'|=|1-\psi|$. Using $\psi(0)=0$ and the condition $\int_\R |\psi'|^2\,ds<\infty$ leads to the optimal profile $\psi_\star$.
\end{proof}

%%%%%%%%%%%%%%%%%%%%%%%%%%%%%%%%%%%%%%%%%%%%%%%%%%%%%%%

\subsection{Proof of Theorem ~\ref{Gammadiff}~{\it(ii)}: The $\Gamma$-$\liminf$ inequality} \label{sec:gammainf}  

  \begin{proposition}\label{Gamliminfprop}
   Let $\eps_h\downarrow0$ and $\eta_h\downarrow0$ be arbitrary sequences. If $\{(u_h,\psi_h)\}\subset\mathcal{H}_g(\Omega)$ is such that $(u_h,\psi_h)\to (u,\psi)$ in $L^1(\Omega)$, then
\begin{equation}\label{Theoliminf}
\widetilde F_0(u,\psi)\leq \liminf_{h\to\infty} \widetilde F_{\eps_h}^{\eta_h}(u_h,\psi_h)\,.
\end{equation}
In addition, 
\begin{enumerate}
\item[(i)] if the liminf is finite, then $u\in \mathcal{L}_g(\Omega)$, $\psi\equiv 1$, and~\eqref{TheoliminfE}-\eqref{TheoliminfI} hold; 
\item[(ii)] if equality holds and the liminf is a finite limit, then 
$\p(u_h)$ converges  to $u^m=:e^{i\vhi}v_\mu$ strongly in $W^{1,p}(\Omega)$ for every $p<2$ and strongly in $W^{1,2}_{\rm{loc}}(\overline\Omega \setminus{\rm spt}\,\mu)$, and~\eqref{strongconvergence}-\eqref{strongconvergencebis} hold. 
\end{enumerate}
\end{proposition}

\begin{proof}
Without loss of generality, we may assume that 
 \begin{equation}\label{liminfassump}
 \liminf_{h\to\infty} \widetilde F_{\eps_h}^{\eta_h}(u_h,\psi_h)= \lim_{h\to\infty} \widetilde F_{\eps_h}^{\eta_h}(u_h,\psi_h)<\infty\,.
\end{equation}
We may also assume that $\|u_h\|_{L^\infty(\Omega)}\leq 1$. Indeed, on the first hand~\eqref{liminfassump} clearly implies that $|u|=1$. On the other hand, the truncation argument used in the proof   of Theorem~~\ref{exist} shows that replacing $u_h$ by $\widehat u_h$ given by~\eqref{deftronc} does not increase the energy. Moreover, 
\begin{equation}\label{estiL1difftrunc}
\|u_h-\widehat u_h\|_{L^1(\Omega)}\leq \||u_h| -1\|_{L^1(\Omega)}\leq \|u_h-u\|_{L^1(\Omega)}\,,
\end{equation}  
so that $\widehat u_h\to u$ in $L^1(\Omega)$. Hence $(u_h,\psi_h)$ can be replaced by $(\widehat u_h,\psi_h)$. 

 Next, we apply  Theorem~~\ref{compactthmfield} to extract a further subsequence to obtain all the conclusions of that theorem. As a consequence, $\psi=1$ and $u\in\mathcal{L}_g(\Omega)$ with  $u^m=e^{i\vhi}v_\mu$. We  have to show that 
\begin{equation}\label{prfthmliminfinterm}
F_{0,g}(u)\leq \lim_{h\to\infty} \widetilde F_{\eps_h}^{\eta_h}(u_h,\psi_h)\,. 
\end{equation} 
We shall prove this inequality  in several steps. 
\vskip3pt

\noindent{\it Step 1.} We  first claim that~\eqref{TheoliminfI} holds. 
Notice that by Remark~~\ref{remcomp}, $\sup_h  I_{\eta_h}(\psi_h)<\infty$. We consider the larger domain $\widetilde\Omega$ defined in~\eqref{extdomom}, as
\begin{equation*}
\widetilde \Omega:=\big\{x\in\R^2: {\rm dist}(x,\Omega)<r_0\big\},
\end{equation*}
for $r_0$ small enough and we recall that 
the nearest point projection ${ \Pi}$ on $\partial\Omega$ is well defined and smooth in $\widetilde\Omega \setminus\Omega$. 
We extend $(u_h,\psi_h)$ and $u$ 
to $\widetilde\Omega$ by setting for $x\in\widetilde\Omega \setminus\Omega$, $\psi_h(x)=1$,  $u_h(x)=g\big({\Pi}(x)\big)$, and $u(x)=g\big({\Pi}(x)\big)$.  
Then, it is elementary to check that $(u_h,\psi_h)\in\mathcal{H}(\widetilde\Omega)$ and $u\in SBV^p(\widetilde\Omega)$ for every $p<2$. In addition, $\P(u_h)\rightharpoonup\P(u)$ weakly in $W^{1,p}(\widetilde\Omega)$ for every $p<2$, and 
$$\int_{\widetilde\Omega} (1-\psi_h)|\nabla\psi_h|\,dx\leq \frac{\eta_h}{2}\int_{\widetilde\Omega}|\nabla\psi_h|^2\,dx+\frac{1}{2\eta_h}\int_{\widetilde\Omega}(1-\psi_h)^2\,dx= I_{\eta_h}(\psi_h)\leq C\,.$$
Next, consider some arbitrary $\delta\in(0,1/2)$. Arguing as in the proof of Proposition~~\ref{precompactresult}, 
$$\int_{\delta}^{1-\delta}(1-t)\mathcal{H}^{1}\big(\partial\{\psi_h<t\}\cap\widetilde\Omega\big)\,dt  \leq I_{\eta_h}(\psi_h)\,,$$
 so that we can find a level $t_h\in(\delta,1-\delta)$ such that 
\begin{equation}\label{2236}
\frac{1-2\delta}{2}\mathcal{H}^1(\partial E_h\cap\widetilde\Omega)\leq I_{\eta_h}(\psi_h)\quad \text{with $E_h:=\{\psi_h<t_h\}$}\,.
\end{equation}
Notice that $E_h\subset \overline\Omega$, and $|E_h|\to 0$ since $\psi_h\to 1$ in $L^1(\Omega)$. 

Fix some $p\in (1,2)$. Defining 
$$\widetilde u_h:=(1-\chi_{E_h})u_h\,, $$
we argue as in the proof of Proposition~~\ref{precompactresult} to show that $\widetilde u_h\in SBV^p(\widetilde\Omega)\cap L^\infty(\widetilde\Omega)$ with   $J_{\widetilde u_h}\, \subset \,\partial E_h$, and that $\widetilde u_h\to u$ in $L^1(\widetilde\Omega)$. Moreover, 
$|\nabla \widetilde u_h|\leq \big|\nabla(\P(u_h))\big	|$ a.e. in $\widetilde \Omega$, so that $\{\nabla\widetilde u_h\}$ is bounded in $L^p(\widetilde\Omega)$. By \cite[Theorem 1]{BCS}, we have 
\[
\liminf_{h\to\infty}\left\{\delta\int_{\widetilde\Omega}|\nabla \widetilde u_h|^p\,dx+\frac{1-2\delta}{2}\mathcal{H}^1(\partial E_h\cap\widetilde\Omega) \right\}\geq 
\delta\int_{\widetilde\Omega}|\nabla \widetilde u|^p\,dx+(1-2\delta) \mathcal{H}^1(J_u\cap \widetilde\Omega)\,,
\]
and thus 
$$  \mathcal{H}^1(J_u\cap \widetilde\Omega)\leq \liminf_{h\to\infty} \frac{1}{2}	\mathcal{H}^1(\partial E_h\cap\widetilde\Omega)  +C\delta\,. $$
Inserting this inequality in~\eqref{2236} and letting $\delta\downarrow0$, we conclude that 
$$\liminf_{h\to\infty}  I_{\eta_h}(\psi_h)\geq  \mathcal{H}^1(J_u\cap \widetilde\Omega)\,.$$
Since $u=g\circ{\Pi}$ in $\widetilde \Omega \setminus\overline\Omega$, we have 
$J_u\subset \overline\Omega$, and thus  $ \mathcal{H}^1(J_u\cap \widetilde\Omega)=  \mathcal{H}^1(J_u\cap\Omega)+\mathcal{H}^1\big(\{u\neq g\}\cap\partial\Omega\big)$, 
and inequality~\eqref{TheoliminfI} follows.  
\vskip3pt

\noindent{\it Step 2.} We now prove the lower bound~\eqref{TheoliminfE}. 
Note that putting \eqref{TheoliminfE} and~\eqref{TheoliminfI} together leads to~\eqref{prfthmliminfinterm}. 
First, we recall that~\eqref{liminfassump} implies~\eqref{asymptenerg} with $v_h:={\rm p}(u_h)$ (by means of~\eqref{GLcone}). Then the proof of~\eqref{TheoliminfE} follows very closely the ones of \cite[Theorem 5.3]{AlicPon} and \cite[Lemma 4.1.1]{ColJer} for the classical Ginzburg-Landau functional. We provide a quite detailed proof for the reader's convenience. Write $\mu=2\pi\sum_{k=1}^{md} x_k$, and choose $\sigma>0$ in such a way  that the balls 
$B_{\sigma}(x_k)$ are contained inside $\Omega$ and are pairwise disjoint. Set 
\[
\mathcal{K}:=\left\{ v_\alpha:\C\setminus\{0\}\to\Ss^1 \mbox{ defined as } \ v(z) =\alpha \frac{z}{|z|} :  \alpha\in \Ss^1\right\}\,.
\]
For each $k\in\{1,\ldots,md\}$, our aim is to prove that either $v_h$ is $W^{1,2}$-close to $\mathcal{K}$ on $\partial B_\sigma(x_k)$, or it has  ``large" energy.
We define for $t\in(0,\sigma]$ and $w\in W^{1,2}(B_t \setminus\overline{B_{t/2}})$,
\[d_t(w,\mathcal{K}):=\min\big\{\|w-v\|_{W^{1,2}(B_t \setminus B_{t/2})} :  v\in \mathcal{K} \big\}\,.\]
It is proven in~\cite{AlicPon} that  for a given $\delta\in(0,1)$, there exists a constant $c_\delta>0$ independent of $t$ such that the condition 
$\liminf_{h} d_t(v_h(\cdot+x_k),\mathcal{K})\geq \delta$ implies  
\begin{equation}\label{liminfK}
 \liminf_{h\to \infty} \frac{1}{2} \int_{B_t(x_k) \setminus B_{t/2}(x_k)}|\nabla v_h|^2\,dx\geq \pi \log 2 +c_\delta\,. 
\end{equation}
Now let $L\in \N$ be such that 
\[
 \frac{Lc_\delta}{m^2}\geq \frac{1}{2 m^2} \int_{\Omega}|\nabla \vhi|^2\,dx+\frac{1}{m^2} \mathbb{W}(\mu)+md \bgamma_m -\frac{\pi d}{m} \log \sigma -\frac{C_*}{m^2}\,,
\]
where $c_\delta$ is the constant from~\eqref{liminfK}, and $C_*$ is the constant from~\eqref{lowerbBGAlicPon}. For $l\in \{1,\ldots,L\}$ we write $C_l(x_k):=B_{2^{1-l}\sigma}(x_k) \setminus B_{2^{-l}\sigma}(x_k)$. By the weak $W^{1,2}_{\textrm{loc}}(\overline\Omega \setminus {\rm spt} \mu)$ 
convergence of $v_h$ to $u^m=e^{i\vhi} v_\mu$, we have for each $k$ and $l$,
\begin{equation}\label{lower2}
 \liminf_{h\to \infty} \frac{1}{2}\int_{C_l(x_k)} |\nabla v_h|^2\,dx\geq \frac{1}{2} \int_{C_l(x_k)} |\nabla \vhi|^2+|\nabla v_\mu|^2\,dx\geq \pi \log 2\,. 
\end{equation}
We now have to distinguish two different cases.
\vskip2pt

\noindent{\it Case 1:} For $h$ large enough, and for each $1\leq l\leq L$, there exists at least one $k_l\in \{1,\ldots,md\}$ such that $d_{2^{1-l}\sigma}(u_h(\cdot+x_{k_l}),\mathcal{K})\geq  \delta$. Then, we estimate 
 \begin{multline*}
 E_{\eh}\big(P(u_h)\big)-\frac{\pi d}{m} |\log \eh|\geq \frac{1}{m^2}E_{\eh}(v_h)-\frac{\pi d}{m} |\log \eh|\\
 \geq  \frac{1}{m^2} \Big\{ E_{\eh}(v_h, B_{2^{-L}\sigma}(\mu))-\pi d m |\log \eh|\Big\}+\frac{1}{m^2}\sum_{l=1}^L\sum_{k=1}^{md} \frac{1}{2} \int_{C_l(x_k)} |\nabla v_h|^2\,dx\,.
 \end{multline*}
Taking the liminf in $h$, and combining~\eqref{lowerbBGAlicPon},~\eqref{liminfK}, and~\eqref{lower2} yields 
 \begin{align}
\nonumber \liminf_{h\to \infty} \left\{E_{\eh}\big(P(u_h)\big)-\frac{\pi d}{m} |\log \eh| \right\}
\nonumber &\geq \frac{\pi d }{m} \log \frac{\sigma}{2^L} +\frac{C_*}{m^2}+\frac{L}{m^2}(\pi d m \log 2 +c_\delta)\\
\nonumber  &=\frac{\pi d}{m} \log \sigma +\frac{C_*}{m^2} +\frac{L c_\delta}{m^2}\\
\label{lowerboundoption1}  &\geq  \frac{1}{2 m^2} \int_{\Omega}|\nabla \vhi|^2\,dx+\frac{1}{m^2} \mathbb{W}(\mu)+md \bgamma_m\,,
 \end{align}
and thus~\eqref{TheoliminfE} holds.
\vskip2pt

\noindent{\it Case 2:} For a  subsequence there exists $\overline{l}\in \{1,\ldots,L\}$ such that, setting $\overline{\sigma}:= 2^{1-\overline{l}}\sigma$,
 $d_{\overline{\sigma}}(u_h(\cdot+x_k),\mathcal{K})< \delta$ for every $k\in\{1,\ldots, md\}$. Let us prove that for $\eps_h$ small enough,
\begin{equation}\label{lowerboundcore}
 G_{\eh}\big(v_h, B_{\overline \sigma}(\mu)\big)-\frac{\pi d}{m}\log \frac{\overline \sigma}{\eh}\geq md\gamma_m(\eps_h,\overline \sigma)-C_{\ovsigma}\delta\,,
\end{equation}
where, here and below, $C_{\ovsigma}$ denotes a nonnegative number depending on $\ovsigma$ but not on $h$ or $\delta$.
To establish this inequality, we shall modify $v_h$ in $B_{\ovsigma}(\mu)$ without increasing the energy too much, and in such a way that it is admissible for~\eqref{defgammam}.
We can proceed on each ball $B_{\ovsigma}(x_k)$ separately, and we may assume without loss of generality that $x_k=0$. Up to a rotation, we can even assume that 
\begin{equation}\label{proxdeltav}
 \int_{B_{\overline \sigma} \setminus B_{\overline \sigma/2}} \left|\nabla v_h-\nabla (e^{i\theta})\right|^2\,dx \leq \delta^2\,,
\end{equation}
where $\theta(x)$ denotes the argument of $x/|x|$. As~\eqref{asymptenerg} holds, we infer from~\eqref{lowerbBGAlicPon} in Theorem~\ref{compAlicPon} %(see~\eqref{bdsigm})
that
\[
 \int_{B_{3\ovsigma/4} \setminus B_{ \ovsigma/2}} |\nabla v_h|^2+\frac{1}{\eps_h^2}(1-|v_h|^2)^2\,dx\leq  C_{\ovsigma}\,.
\]
Therefore, for every $h$ we can find  $\widetilde{\sigma}_h\in [ \ovsigma/2, 3\ovsigma/4]$ for which 
\begin{equation}\label{fubinisigmatilde}
 \int_{\partial B_{\widetilde{\sigma}_h}} |\nabla v_h|^2+\frac{1}{\eps_h^2}(1-|v_h|^2)^2\,d\mathcal{H}^1\leq C_{\ovsigma}\,,
\end{equation}
and
\begin{equation}\label{estimH1vtheta}
 \int_{\partial B_{\widetilde{\sigma}_h}} \left|\nabla v_h-\nabla (e^{i\theta})\right|^2\, d\mathcal{H}^1\leq C_{\ovsigma} \delta^2\,.
\end{equation}
From~\eqref{estimH1vtheta}  we first derive 
\begin{align}
\nonumber \int_{\partial B_{\widetilde{\sigma}_h}} |\nabla v_h|^2\, d\mathcal{H}^1&= \int_{\partial B_{\widetilde{\sigma}_h}} \left|\nabla v_h-\nabla( e^{i\theta}) +\nabla( e^{i\theta})\right|^2\, d\mathcal{H}^1\\
\nonumber &\leq  \frac{1+\delta}{\delta}\int_{\partial B_{\widetilde \sigma_h}} \left|\nabla v_h-\nabla (e^{i\theta})\right|^2\, d\mathcal{H}^1+(1+\delta)\int_{\partial B_{\widetilde \sigma_h}} \left|\nabla(e^{i\theta})\right|^2\, d\mathcal{H}^1\\
\label{strongerboundsigmatilde} &\leq C_{\ovsigma}\delta+\frac{2\pi}{\widetilde{ \sigma}_h}(1+\delta)\le C_{\ovsigma}\delta+\frac{2\pi}{\widetilde{ \sigma}_h} \,.
\end{align}
Next, by a scaling argument, one obtains 
\[
 \||u|-1\|^2_{L^{\infty}(\partial B_r)}\les C\eps_h \int_{\partial B_r} |\nabla u|^2+\frac{1}{\eps_h^2} (1-|u|^2)^2\,d\mathcal{H}^1
\]
for every $\eps_h\le r$ and $u\in W^{1,2}(\partial B_r)$.  Hence,~\eqref{fubinisigmatilde} yields  
\begin{equation}\label{inftybound}
 \| |v_h|-1\|_{L^{\infty}(\partial B_{\widetilde{\sigma}_h})}\leq C_{\ovsigma} \eps_h^{1/2}\,.
\end{equation}
We can thus write $v_h= \rho_he^{i\theta_h}$ on $\partial B_{\widetilde{\sigma}_h}$. Moreover, we have 
$\deg(v_h,\partial B_{\widetilde{\sigma}_h})=1$  by~\eqref{estimH1vtheta}, so that 
$\theta_h-\theta$ can be chosen to be single valued. Let us extend $\rho_h$ and $\theta_h$ by zero homogeneity outside $B_{\widetilde{\sigma}_h}$. For $\eps_h$ small enough, we set $\hat{\sigma}_h:= \widetilde{\sigma}_h+\eps_h^{1/2}$, and we define $\widetilde{v}_h$ in $B_{\hat{\sigma}_h}$ as
\[
 \widetilde{v}_h(x):=\begin{cases}
               v_h(x) & \textrm{in } B_{\widetilde{\sigma}_h}\,,\\[5pt]
\displaystyle               \left(\rho_h(x)\frac{\hat{\sigma}_h-|x|}{\hat{\sigma}_h-\widetilde{\sigma}_h}+\frac{|x|-\widetilde{\sigma}_h}{\hat{\sigma}_h-\widetilde{\sigma}_h} \right) e^{i\theta_h(x)} & \textrm{in } B_{\hat{\sigma}_h} \setminus B_{\widetilde{\sigma}_h}\,.
              \end{cases}
\]
From~\eqref{fubinisigmatilde} and~\eqref{inftybound} we infer that 
\begin{equation}\label{boundvepsannulus}
 G_{\eh}(\widetilde{v}_h, B_{\hat{\sigma}_h} \setminus B_{\widetilde{\sigma}_h})\leq C_{\ovsigma} \eps_h^{1/2}\,.
\end{equation}
In turn,~\eqref{strongerboundsigmatilde} and~\eqref{inftybound} yield, for $\eps_h$ small enough, 
\begin{equation}\label{strongerboundsigmahat}
 \int_{\partial B_{\hat{\sigma}_h}} |\nabla \widetilde{v}_h|^2\, d\mathcal{H}^1\leq C_{\ovsigma}\delta+\frac{2\pi}{\hat{\sigma}_h}\,.
\end{equation}
We now are left to define $\widetilde{v}_h$ in $B_{\overline \sigma} \setminus B_{\hat{\sigma}_h}$. We define for  $x\in B_{\overline \sigma} \setminus B_{\hat{\sigma}_h}$, 
\[\hat{\theta}_h(x):=\frac{\overline \sigma-|x| }{\overline \sigma -\hat{\sigma}_h}{\theta}_h+\frac{|x|-\hat{\sigma}_h}{\overline \sigma-\hat{\sigma}_h}\theta\,,\] 
i.e., the linear interpolation between $\theta_h(x)$ and $\theta(x)$. Setting $\widetilde{v}_h(x):=e^{i\hat{\theta}_h(x)}$ in $B_{\overline \sigma} \setminus B_{\hat{\sigma}_h}$,  it can be proven by means of~\eqref{strongerboundsigmahat} and elementary computations (see \cite[Lemma 4.1.1]{ColJer}) that 
\[
 \int_{B_{\overline \sigma} \setminus B_{\hat{\sigma}_h}} |\nabla \widetilde{v}_h|^2\,dx\leq 2\pi \log \frac{\overline \sigma}{\hat{\sigma}_h} +C_{\ovsigma}\delta\,.
\]
Combining this last inequality with~\eqref{boundvepsannulus}, we first derive
\begin{align*}
 G_{\eh}(v_h, B_{\overline \sigma})&= G_{\eh}(\widetilde{v}_h, B_{\overline \sigma})+G_{\eh}(v_h,B_{\overline \sigma} \setminus B_{\widetilde{\sigma}_h})-G_{\eh}(\widetilde{v}_h,B_{\overline \sigma} \setminus B_{\widetilde{\sigma}_h})\\
 &\geq G_{\eh}(\widetilde{v}_h, B_{\overline \sigma})+ \frac{1}{2m^2}\int_{B_{\overline \sigma} \setminus B_{\hat{\sigma}_h}} |\nabla v_h|^2-|\nabla \widetilde{v}_h|^2\,dx-G_{\eh}(\widetilde{v}_h,B_{\hat{\sigma}_h} \setminus B_{\widetilde{\sigma}_h})\\
 &\geq G_{\eh}(\widetilde{v}_h, B_{\overline \sigma})+\frac{1}{2m^2} \lt(\int_{B_{\overline \sigma} \setminus B_{\hat{\sigma}_h}} |\nabla v_h|^2\,dx-2\pi \log \frac{\overline \sigma}{\hat{\sigma}_h}\rt) -C_{\ovsigma}(\delta+\eps_h^{1/2})\,.
\end{align*}
From~\eqref{proxdeltav} we then obtain 
\begin{align*}
 \int_{B_{\overline \sigma} \setminus B_{\hat{\sigma}_h}} |\nabla v_h|^2\,dx&\geq \int_{B_{\overline \sigma} \setminus B_{\hat{\sigma}_h}} \left|\nabla (e^{i\theta})\right|^2\,dx+ \int_{B_{\overline \sigma} \setminus B_{\hat{\sigma}_h}} \left|\nabla v_h-\nabla( e^{i\theta})\right|^2\,dx\\
 & \qquad-2 \left(\int_{B_{\overline \sigma} \setminus B_{\hat{\sigma}_h}} \left|\nabla (e^{i\theta})\right|^2\,dx\right)^{1/2} \left(\int_{B_{\overline \sigma} \setminus B_{\hat{\sigma}_h}} \left|\nabla v_h-\nabla (e^{i\theta})\right|^2\,dx\right)^{1/2}\\
 &\geq 2\pi \log \frac{\overline \sigma}{\hat{\sigma}_h}-C_{\ovsigma} \delta\,,
\end{align*}
and we deduce that for $\eps_h$ small enough, 
\[
 G_{\eh}(v_h, B_{\overline \sigma})\geq G_{\eh}(\widetilde{v}_h, B_{\overline \sigma})-C_{\ovsigma} (\delta+\eps_h^{1/2})\,.
\]
By the  definition of $\gamma_m(\eps, \overline \sigma)$ (see~\eqref{defgammameps}), we conclude that for $\eps_h$ small, 
\[
 G_{\eh}(v_h, B_{\overline \sigma})-\frac{\pi}{m^2} \log \frac{\overline \sigma}{\eps_h}\geq \gamma_m(\eps_h,\overline \sigma)-C_{\ovsigma}(\delta+\eps_h^{1/2})\,,
\]
which proves~\eqref{lowerboundcore}.

We can now complete the proof of~\eqref{TheoliminfE}. Indeed, by~\eqref{lowerboundcore} and the convergence of $v_h$ towards $e^{i\vhi} v_\mu$ in the weak  $W^{1,2}_{\loc}(\Omega \setminus \spt \mu)$ topology,
\begin{multline}
 \liminf_{h\to \infty} \left\{G_{\eh}(v_h)-\frac{\pi d}{m} |\log \eh|\right\}\\ 
 =\ \liminf_{h \to \infty } \left\{G_{\eh}\big(v_h, \Omega  \setminus B_{\overline \sigma}(\mu)\big)-\frac{\pi d}{m} |\log \overline \sigma|\right.
   + \left.G_{\eh}(v_h,  B_{\overline \sigma}(\mu))-\frac{\pi d}{m} \log\frac{\overline \sigma}{ \eh}\right\} \\
 \ \geq\ \liminf_{h\to \infty}\left\{G_{\eh}\big(v_h, \Omega  \setminus B_{\overline \sigma}(\mu)\big)-\frac{\pi d}{m} |\log \overline \sigma|\right\}
 + md \lim_{h\to \infty}  \bgamma_m(\eps_h,\overline \sigma) -C_{\ovsigma}\delta \\
 \ \geq \frac{1}{2 m^2} \int_{\Omega \setminus B_{\overline \sigma}(\mu)} |\nabla (e^{i\vhi} v_\mu)|^2\,dx-\frac{\pi d}{m} |\log \overline \sigma| +md \bgamma_m -C_{\ovsigma} \delta\,,\label{borneinfsepar}
\end{multline}
where the last inequality follows from~\eqref{homogenbgamma} and Lemma ~\ref{bdinfconstgam}. In view of Proposition~~\ref{keypropcomp}, letting first $\delta\downarrow0$ and then $\sigma\downarrow0$ leads to~\eqref{TheoliminfE}. 
\vskip3pt

\noindent{\it Step 3.} In order to complete the proof, let us show that if equality holds in~\eqref{Theoliminf}, then~\eqref{strongconvergence} and~\eqref{strongconvergencebis} hold.  Note that ~\eqref{strongconvergence} rewrites as 
\begin{equation}\label{rew1450}
\lim_{h\to\infty} G_{\eps_h}\big(v_h,\Omega \setminus B_r(\mu)\big)= \frac{1}{2m^2} \int_{\Omega \setminus B_{r}(\mu)} |\nabla (e^{i\vhi} v_\mu)|^2\,dx \quad\forall r>0\,,  
\end{equation} 
 which, combined with the weak convergence of $\{v_h\}$ in $W^{1,2}_{\rm loc}(\overline\Omega \setminus{\rm spt}\,\mu)$ to $e^{i\vhi}v_\mu$, classically leads to its strong convergence in $W^{1,2}_{\rm loc}(\overline\Omega \setminus{\rm spt}\,\mu)$. 

From~\eqref{TheoliminfE} and~\eqref{TheoliminfI} we first infer that 
$$ \lim_{h\to \infty} I_{\eta_h} (\psi_h)=\mathcal{H}^1(J_u)+\mathcal{H}^1\big(\{u\neq g\}\cap\partial\Omega\big) \,,$$
and
$$ \lim_{h\to \infty} \left\{G_{\eh}(v_h)-\frac{\pi d}{m}|\log \eh|\right\}= \frac{1}{2m^2}\int_\Omega|\nabla\vhi|^2\,dx+\frac{1}{m^2}\mathbb{W}(\mu) 
+m d\bgamma_m\,.$$
In view of~\eqref{lowerboundoption1}, Case 2 in Step 2 above must hold. We argue by contradiction assuming that~\eqref{rew1450} does not hold.  
Then we can find a  subsequence, $\sigma_0>0$, and $\eta_0>0$ such that 
\[
 \liminf_{h\to \infty} G_{\eh}\big(v_h, \Omega  \setminus B_{\sigma_0}(\mu)\big)\geq \frac{1}{2m^2} \int_{\Omega \setminus B_{\sigma_0}(\mu)} |\nabla (e^{i\vhi}v_\mu)|^2\,dx+ \eta_0\,.
\]
By lower semi-continuity of the Dirichlet energy,  the same inequality holds for every $\sigma\in(0, \sigma_0)$. Then, for $\sigma$ and $\delta$ small enough, we have by~\eqref{borneinfsepar},   
\begin{multline*}
 \frac{1}{2m^2}\int_\Omega|\nabla\vhi|^2\,dx+\frac{1}{m^2}\mathbb{W}(\mu) +m d\bgamma_m=\lim_{h\to \infty} \left\{G_{\eh}(v_h)-\frac{\pi d}{m}|\log \eh|\right\}\\
 \geq \frac{1}{2 m^2} \int_{\Omega \setminus B_{\overline \sigma}(\mu)} |\nabla (e^{i\vhi} v_\mu)|^2\,dx-\frac{\pi d}{m} |\log \overline \sigma| +md \bgamma_m -C_{\ovsigma} \delta+ \eta_0\,,
\end{multline*}
where $\ovsigma$ is determined from $\sigma$ as in Case 2 above. Using Proposition~~\ref{keypropcomp} again, we let $\delta\to0$ and then $\sigma\to 0$ to reach a contradiction.

To conclude, it only remains to prove that $v_h\to e^{i\vhi} v_\mu$ strongly in $W^{1,p}(\Omega)$ for every $p<2$.  
Fix an arbitrary $p\in(1,2)$. Since $\{|\nabla v_h|^p\}$ is bounded in $L^1(\Omega)$, we can extract a further subsequence such that $|\nabla v_h|^p\rightharpoonup |\nabla (e^{i\vhi} v_\mu)|^p+\nu_*$  weakly* as measures for some non negative finite measure $\nu_*$ on $\Omega$. From the strong convergence of $v_h$ in $W^{1,2}_{\rm loc}(\Omega \setminus{\rm spt}\,\mu)$, 
we infer that ${\rm spt}\,\nu_*\subset {\rm spt}\,\mu$. Since  $\{\nabla v_h\}$ is also bounded in $L^{q}(\Omega)$ for every $q\in(p,2)$, we have by H\"older's inequality
$$\nu_*(\Omega)\leq\liminf_{h\to\infty}\int_{B_r(\mu)} |\nabla v_h|^p\,dx\leq Cr^{2(1-p/q)}\quad\text{for every $r>0$ and $q\in(p,2)$}\,.  $$
 Letting $r\downarrow0$ we deduce that $\nu_*\equiv 0$ which together  with the weak convergence in $W^{1,p}(\Omega)$ of $v_h$ towards $e^{i\vhi} v_\mu$ concludes the proof. 
\end{proof}

%%%%%%%%%%%%%%%%%%%%%%%%%%%%%%%%%%%%%%%%%%%%%%%%%%%%%%%

\subsection{Proof of Theorem ~\ref{Gammadiff}~{\it(iii)}: Construction of recovery sequences}\label{sec:gammasup}

In this section we prove the  $\Gamma-$limsup inequality. 
\begin{proposition}\label{proplimsup}
 Let $\eps_h\downarrow0$ and $\eta_h\downarrow0$ be arbitrary sequences.  For every $u\in \mathcal{L}_g(\Omega)$ with $u^m=:e^{i\vhi} v_\mu$,  
 there exists a sequence $\{(u_h,\psi_h)\}\subset \mathcal{H}_g(\Omega)$ such that  $\{u_h\}\subset SBV^2(\Omega)$, and 
 \begin{enumerate}
 \item[(i)]  $(u_h,\psi_h)\to (u,1)$ strongly in $L^1(\Omega)$; 
 \item[(ii)] $\p(u_h)\to u^m$ strongly in $W^{1,2}_{\rm loc}(\overline\Omega \setminus  {\rm spt}\, \mu)$ and $W^{1,p}(\Omega)$ for every $p<2$;
 \item[(iii)]~\eqref{theolimsup01} and~\eqref{theolimsup03} hold.
 \end{enumerate}
 Moreover, we can choose $u_h$ such that 
 \begin{equation}
 \label{theolimsup02}
\lim_{h\to \infty}  \mathcal{H}^1(J_{u_h}) = \mathcal{H}^1(J_u)+\mathcal{H}^1\big(\{u\neq g\}\cap\partial\Omega\big)\,.\end{equation}
\end{proposition}

The proof of Proposition \ref{proplimsup} relies on a suitable approximation procedure showing that maps in $\mathcal{L}_g(\Omega)$ having a compact jump set lying at a positive distance from the boundary are dense in energy. This is the purpose of the following section. The proof of Proposition \ref{proplimsup} is then performed in Section~\ref{proofofpropgamsup}.

\subsubsection{Some density results}\label{sec:density}

\begin{lemma}\label{lemsmoothphase}
Let $u\in\mathcal{L}_g(\Omega)$ with $u^m=:e^{i\vhi}v_\mu$ and $\mu=2\pi \sum_{k=1}^{md}\delta_{x_k}$. For every  $\delta>0$, there exists $u_*\in \mathcal{L}_g(\Omega)$ and constants $\xi_1,\cdots,\xi_{md}\in \Ss^1$ such that
\vskip3pt
\begin{enumerate}
\item[(i)]  ${\rm dist}(J_{u_*},\partial\Omega)>0$,  $u_*=g$ on $\partial\Omega$, and $\curl j(u_*^m)=\mu$;
\vskip3pt
\item[(ii)] $u_*^m(x)=\xi_k\displaystyle\dfrac{x-x_k}{|x-x_k|}$ in a neighborhood of $x_k$ for each $k\in\{1,\cdots,md\}$; 
\vskip3pt
\item[(iii)] $\|u-u_*\|_{L^1(\Omega)}<\delta$ and $E_0(u_*)< E_0(u) + \delta$; 
\vskip3pt
\item[(iv)] $\mathcal{H}^1(J_{u_*}) <\mathcal{H}^1(J_{u})+\mathcal{H}^1 \lt(\{u\neq g\}\cap \partial\Omega\rt)+\delta$.
\end{enumerate}
\end{lemma}

\begin{remark}\label{rem2054camp}
If one considers $\delta_k\downarrow0$ and $\{u_{k}\}\subset  \mathcal{L}_g(\Omega)$ the corresponding sequence provided by Lemma~~\ref{lemsmoothphase}, then $u_{k}^m\to u^m$ strongly in $W^{1,p}(\Omega)$  for every $p<2$. Indeed, writing $u_{k}^m=e^{i\vhi_k}v_\mu $, item {\it (iii)} implies the 
strong convergence of $\{\nabla \vhi_k\}$ in $L^2(\Omega)$,  which in turn yields the strong convergence of $\{u^m_{k}\}$ in $W^{1,p}(\Omega)$.
\end{remark}

\begin{proof}[Proof of Lemma ~\ref{lemsmoothphase}] 
{\it Step 1.} We start by modifying $u$ in such a way that {\it (i)}, {\it (iii)}, and {\it (iv)} hold. We proceed as follows. We consider the larger domain $\widetilde\Omega$ defined in~\eqref{extdomom}, and we recall that the nearest point projection $\Pi$ on $\partial\Omega$ is well defined and smooth in $\{x:{\rm dist}(x,\partial\Omega)<2r_0\}$. 
Denote by $d_\Omega:\R^2\to\R$ the signed distance to $\partial\Omega$, i.e., $d_\Omega(x):={\rm dist}(x,\partial\Omega)$ for $x\in\R^2 \setminus\Omega$, and $d_\Omega(x):=-{\rm dist}(x,\partial\Omega)$ for $x\in\Omega$. Note that $d_\Omega$ is smooth in $\{x:{\rm dist}(x,\partial\Omega)<2r_0\}$. Next we consider a smooth vector field $X \in C^\infty_c(\widetilde\Omega \setminus {\rm spt}\,\mu;\R^2)$ satisfying $X=\nabla d_{\Omega}$ in $\{x:{\rm dist}(x,\partial\Omega)<r_0/2\}$, and we fix $\sigma\in(0,r_0/2)$ such that $\spt \,X \subset \widetilde\Omega \setminus \overline B_{\sigma}(\mu)$. We denote by $\{\phi_t\}_{t\in\R}$ the integral flow on $\R^2$ generated by $X$. Then ${\rm spt}(\phi_t-{\rm id})\subset  \widetilde\Omega \setminus \overline B_{\sigma}(\mu)$ for every $t\in \R$, so that $\{\phi_t\}_{t\in\R}$ defines a one parameter family of smooth diffeomorphisms from  $\widetilde\Omega \setminus \overline B_{\sigma}(\mu)$ onto $\widetilde\Omega \setminus \overline B_{\sigma}(\mu)$. 

We extend $u$ to $\widetilde\Omega$ by setting $u(x):=g\big(\Pi(x)\big)$ for $x\in \widetilde\Omega \setminus\Omega$. 
Then we set $u_t:=u\circ \phi_{t}$, and we consider the restriction of $u_t$  to $\Omega$. By construction, for every $t>0$ we have 
$ u_t\in\mathcal{L}_g(\Omega)$, $u_t=g$ on $\partial\Omega$, $u_t=u$ in $B_\sigma(\mu)$, and ${\rm dist}(J_{u_t},\partial \Omega)>0$. 
 Moreover, as $t\downarrow 0$, we have $u_t\rightarrow u$ weakly* in $BV(\Omega)$, $\nabla u_t\to\nabla u$ strongly in $L^2(\Omega \setminus B_\sigma(\mu))$, and 
$$\mathcal{H}^1(J_{u_t}) \to \mathcal{H}^1(J_{u})+\mathcal{H}^1 \lt(\{u\neq g\}\cap \partial\Omega\rt)\,.$$
Given $\delta>0$, we can then choose $t>0$ small enough so that $u_t$ satisfies {\it (i)}, {\it (iii)}, and {\it (iv)} of  Lemma~\ref{lemsmoothphase}.  
\vskip3pt

\noindent{\it Step 2.} To complete the proof, we shall modify further $u_t$ in $B_\sigma(\mu)$ to achieve property {\it (ii)}. 
Fix ${\sigma'} \in(0,\sigma)$  such  that $2{\sigma'} < \min\{|x_k-x_l|:1\leq k<l\leq md\}$. From the specific form of $v_\mu$ (see~\eqref{definivmu}),  for each $k\in \{1,\cdots,dm\}$ there exists $\psi_k\in W^{1,2}(B_{\sigma'}(x_k))$ such that 
\[
u^m(x)=e^{im\psi_k(x)} \dfrac{x-x_k}{|x-x_k|}\qquad\mbox{for }x\in B_{{\sigma'}}(x_k)\,.
\]
For $\rho\geq0$ and $\theta\in[0,2\pi)$ we set $\q_m(\rho e^{i\theta}):=\rho e^{i \theta/m}$, and define  $\vartheta_k(x):=\q_m((x-x_k)/|x-x_k|)$ for $x\in B_{\sigma'}(x_k)$. 
Since $\curl j (\vartheta^m_k)=\mu=\curl j(u^m)$ in $B_{\sigma'}(x_k)$ and $\p(\vartheta_k)=\p(u)$, we can invoke Lemma~~\ref{lemstructure} to infer that there exists $\widehat{u}_k\in SBV(B_{\sigma'}(x_k),\mathbf{G}_m)$ such that
\[
u= \widehat{u}_k \vartheta_k e^{i\psi_k}\qquad\mbox{in } B_{\sigma'}(x_k)\,.
\]
We now fix a cut-off function $\chi\in C^\infty_c(B_1,[0,1])$ such that, $\chi\equiv 1 \mbox{ in } B_{1/2}$, and we set $\chi_{k,r}(x):=\chi((x-x_k)/r)$ for $r>0$. We define for $k\in \{1,\cdots,md\}$ and $r\in(0,\sigma')$,
\[
u_{k,r}(x)  := \widehat{u}_k(x) \vartheta_k(x) \exp\Big(i\psi_k(x) +i\chi_{k,r}(x)\big(\overline{\psi}_{k,r} -\psi_k(x)\big) \Big)\qquad\mbox{for }x\in B_{\sigma'}(x_k)\,,
\]
where $\overline{\psi}_{k,r}$ denotes the mean value
$$\overline{\psi}_{k,r}:=\dfrac1{\pi r^2}\int_{B_r(x_k)} \psi_k(x)\,dx\,.$$
By construction, we have
\[
u^m_{k,r}(x)=\xi_{k,r} \dfrac{x-x_k}{|x-x_k|}\qquad\mbox{for } x\in B_{r/2}(x_k)\,,
\]
with $\xi_{k,r}:=\exp(i m\overline {\psi}_{k,r})$, and  $u_{k,r}=u=u_t$ in $B_{\sigma'}(x_k) \setminus B_r(x_k)$. Finally, we set 
\[
u_*(x):=\ \begin{cases}
u_{k,r}(x) & \mbox{ if }x\in  B_{\sigma'}(x_k)\,,\\
u_t(x)&\mbox{ if }x\in  \Omega \setminus  B_{\sigma'}(\mu)\,.
 \end{cases}
\]
By means of Poincar\'e's inequality (and Step 1), it is standard to check that  for $r$ small enough, $u_*$ complies to all the requirements of the lemma.
\end{proof}

Now we show that we can substitute to $u_*$ a mapping $u_\sharp$ with a compact jump set.
\begin{lemma}\label{Alternative2}
Let $u_*\in\mathcal{L}_g(\Omega)$ be such that ${\rm dist}(J_{u_*},\partial\Omega)>0$. For every $\delta>0$, there exist $u_{\sharp} \in \mathcal{L}_g(\Omega)$, a compact set $K\subset \Omega$, and $\sigma>0$ such that  
\begin{enumerate}
\item[(i)] $\|u_\sharp-u_*\|_{L^1(\Omega)} \leq \delta$,  $u_{\sharp}^m=u^m_*$, and $\mathcal{H}^1(J_{u_\sharp} \setminus J_{u_*})=0$;
\vskip3pt
\item[(ii)]  $\mathcal{H}^1(K\triangle J_{u_\sharp})=0$; 
\vskip3pt
\item[(iii)] $\mathcal{H}^1\big(K\cap B_r(x)\big)\geq r/2\,$ for every $x\in K$ and every $r\in(0,\sigma)$.
\end{enumerate}
\end{lemma}

\begin{proof}
{\it Step 1.} Given a parameter $\lambda\geq 1$, we consider the class of maps  
$$\mathscr{A}:=\Big\{ u \in \mathcal{L}_g(\Omega): u^m=u_*^m\,,\; \mathcal{H}^1(J_u \setminus J_{u_*})=0 \Big\}\,, $$
and the functional $A_\lambda:\mathscr{A}\to[0,\infty)$ defined by 
\[
A_\lambda(u):=\mathcal{H}^1(J_u)+\lambda\int_{\Omega}|u-u_*|^2\,dx\,.
\]
 We claim that $A_\lambda$ admits at least one minimizer $u_\lambda\in\mathscr{A}$.  
Indeed,  $|\nabla u|=\big|\nabla(\P(u_*)\big|=|\nabla u_*|$ for every  $u\in\mathscr{A}$, so that $\{\nabla u: u\in\mathscr{A}\}$ is bounded in $L^p(\Omega)$ for every $p<2$. Obviously, if $\{u_k\}\subset\mathscr{A}$ is a minimizing sequence, then $A_\lambda(u_k)$ is bounded.  By the compactness theorem for $SBV$ functions~\cite[Theorem 4.7 \& 4.8, and Remark 4.9]{AFP}, we can find $u_\lambda\in SBV(\Omega)$ and a  subsequence such that $u_k\to u_\lambda$ in $L^1(\Omega)$ and a.e. in $\Omega$, and satisfying $A_\lambda(u_\lambda)\leq\liminf_k A_{\lambda}(u_k)$. From the pointwise convergence we infer that $u_\lambda$ is $\Ss^1$-valued, and $u^m_\lambda=u^m_*$.  Moreover,  the sequence of positive measures $\{\mathcal{H}^1\restr J_{u_k}\}$ weakly* converges towards $\mathcal{H}^1\restr J_{u_\lambda}$. Since $\mathcal{H}^1\restr J_{u_k}\leq \mathcal{H}^1\restr J_{u_*}$,  we have $\mathcal{H}^1\restr J_{u_\lambda}(U)\leq \liminf_k \mathcal{H}^1\restr J_{u_k}(U)\leq \mathcal{H}^1\restr J_{u_*}(U)$ for every open set $U\subset \Omega$. By outer regularity, we deduce that $\mathcal{H}^1\restr J_{u_\lambda} \leq\mathcal{H}^1\restr J_{u_*}$. Hence $u_\lambda\in\mathscr{A}$, and thus $u_\lambda$ is a minimizer of $A_\lambda$.  

Noticing that $A_\lambda(u_\lambda)\leq A_\lambda(u_*)$ and $\mathcal{H}^1(J_{u_\lambda})\leq\mathcal{H}^1(J_{u_*})$, we  now deduce that $u_\lambda\to u_*$ in $L^1(\Omega)$ as $\lambda\to\infty$. Given $\delta>0$, we  choose $\lambda$ large enough so that {\it (i)} holds with $u_\sharp:=u_\lambda$.
To complete the proof, we have to find a compact set $K$ and $\sigma>0$ such that {\it (ii)} and {\it (iii)} hold. This is the purpose of the next steps. 
\vskip3pt

\noindent{\it Step 2.}  Write $\curl j(u^m_*)=:\mu=2\pi \sum_{k=1}^{md}\delta_{x_k}$, and set 
\[
\sigma:=\min\left(\dfrac1{8\lambda} \,,\,{\rm dist}(J_{u_*},\partial \Omega)\,,\,\frac{1}{2}\min\big\{|x_k-x_l|:1\leq k<l\leq md\big\}\right)\,.
\]
We claim that for $\mathcal{H}^1$-a.e. $x\in J_{u_\sharp}$, there holds
\begin{equation}
\label{Alfhors0}
\mathcal{H}^1\big( J_{u_\sharp}\cap B_r(x)\big)\geq r \quad \mbox{for every } r \in (0,\sigma)\,.
\end{equation}
Let us first recall that for $\mathcal{H}^1$-a.e. $x\in J_{u_\sharp}$,
\begin{equation}\label{yadumondeautour}
\mathcal{H}^1\big(J_{u_\sharp}\cap B_r(x)\big)>0\quad \mbox{ for  every }r\in(0,\sigma)\,.
\end{equation}
Hence, it is enough to establish~\eqref{Alfhors0} at every $x\in J_{u_\sharp} \setminus{\rm spt}\,\mu$ such that~\eqref{yadumondeautour} holds. Let us fix such a point  $x$ and without loss of generality let us assume $x_0=0$. Setting 
$\sigma_0:=\min\big(\sigma,{\rm dist}(0,\spt\mu)\big)$, we shall distinguish the two cases $r\in(0,\sigma_0]$ and $r\in(\sigma_0,\sigma)$.

\vskip3pt

\noindent{\it Case 1: $r\in(0,\sigma_0]$.} Since $B_{\sigma_0} \subset \Omega \setminus\spt\, \mu$ and $u_\sharp^m=u_*^m$,  we have  $\curl j(u_\sharp)=0$ in $\mathscr{D}^\prime(B_{\sigma_0})$ by Lemma~~\ref{struct}.  Applying Lemma~~\ref{lemstructure} in the ball $B_{\sigma_0}$ (with $u_1=1$ and $u_2=u_\sharp$), we obtain a function  $\vhi\in W^{1,1}(B_{\sigma_0})$ and  a Caccioppoli partition $\{E_k\}_{k=0}^{m-1}$ of the ball $B_{\sigma_0}$ such that 
\begin{equation*}
u_\sharp =\left(\sum_{k=0}^{m-1}{\bf a}^k\chi_{E_k}\right)e^{i\vhi}\quad\mbox{in }B_{\sigma_0}\,.
\end{equation*} 
Since $\vhi\in W^{1,1}(B_{\sigma_0})$, we have 
$$J_{u_\sharp}\cap B_{\sigma_0}=\bigcup_{k=0}^{m-1}\partial E_k \quad\text{up to an $\mathcal{H}^1$-null set\,,}$$
where $\partial E_k$ denotes the reduced boundary of $E_k$ in $B_{\sigma_0}$.  Moreover (see \cite[Remark 4.22]{AFP}), 
\begin{equation}\label{J2EB}
\mathcal{H}^1(J_{u_\sharp}\cap B_r)=\dfrac12 \sum_{k=0}^{m-1} \mathcal{H}^1(\partial E_k \cap B_r)\quad\text{for every  $r\in(0,\sigma_0)$}\,.
\end{equation}

Before going any further, let us recall a slicing property of sets of finite perimeter (see e.g. \cite[Chapter~3, Section 3.11]{AFP} for further details). 
For a Borel set $E\subset \R^2$, we first denote by $E^{1}$ the (Borel) set of points of density $1$ for $E$, i.e., 
$$E^{1}:=\left\{y\in\R^2: \lim_{\rho\downarrow0}\frac{|E\cap B_\rho(y)|}{\pi\rho^2}=1 \right\} \,.$$
For $r>0$, we write $E_r:=E^{1}\cap \partial B_r$, and we consider $E_r$ as a subset of $\partial B_r$ (in particular, $\partial E_r$ and ${\rm int}(E_r)$ denote the relative boundary and relative interior of $E_r$ in $\partial B_r$, respectively). If $E\subset\R^2$ has a finite perimeter, then for a.e. $r>0$ the following properties hold:  
\begin{equation}\label{2250B}
\text{$\partial E_r$ is finite and is equal to $\partial E\cap \partial B_r$}\,.
\end{equation}

We point out that since $\{E_k\}$ is a Caccioppoli partition of $B_{\sigma_0}$, by \cite[Theorem 4.17]{AFP} 
\begin{equation}\label{2319B}
B_{\sigma_0}=\bigcup_{k=0}^{m-1} E_k^1 \cup\bigcup_{k=0}^{m-1}\partial E_k\quad\text{up to an $\mathcal{H}^1$-null set}\,.
\end{equation}

For almost all radii $r\in(0,\sigma_0)$, \eqref{2250B} holds for each $E_k$, $k\in\{0,\ldots, m-1\}$.  We claim that for such radii
\begin{equation}\label{Claim2}
\sum_{k=0}^{m-1} \mathcal{H}^0(\partial E_k\cap\partial B_r)\geq 1.
\end{equation}
To prove~\eqref{Claim2}, let us fix such a $r$ and assume by contradiction that 
\begin{equation*}
\partial E_k\cap \partial B_r =\partial(E_k)_r=\emptyset \qquad\mbox{ for every $k\in \{0,\cdots,m-1\}$}\,.
\end{equation*}
Then each $(E_k)_r$ is either equal to $\partial B_r$ or empty, in particular  $\mathcal{H}^1(E_k\cap\partial B_r)\in\{0,2\pi r\}$. Since $\mathcal{H}^1(\partial E_k\cap \partial B_r)=0$ for each $k$ by~\eqref{2250B}, we infer from~\eqref{2319B} that  $\mathcal{H}^1(E_{k_r}\cap\partial B_r)=2\pi r$ for exactly one index $k_r\in\{0,\ldots,m-1\}$. Consequently, every point of $\partial B_r$ is a point of density $1$ for $E_{k_r}$, and of density~$0$ for $E_k$ with $k\not=k_r$. 
We then introduce the competitor 
\[
\widetilde u :={\bf a}^{k_r} e^{i\vhi}\quad\mbox{in }B_r\,,\quad \widetilde u :=u_\sharp \quad\mbox{in }\Omega \setminus B_r\,.
\]
By construction  $J_{\widetilde{u}} \setminus \overline B_r=J_{u_\sharp} \setminus \overline B_r$ and $J_{\widetilde u}\cap \overline{B_r}=\emptyset$, so that $\tilde u$ is an admissible competitor. By optimality of $u_\sharp$, we have $A_\lambda(\widetilde u)\geq A_\lambda(u_\sharp)$. Using that 
\[
 |\widetilde{u}-u_*|^2-|u_\sharp -u_*|^2= \lt(|\widetilde{u}-u_*|+|u_\sharp -u_*|\rt)\lt(|\widetilde{u}-u_*|^2-|u_\sharp -u_*|\rt)\le4 |\widetilde{u}- u_\sharp|
\]
(since $|\widetilde{u}|=|u_*|=|u_\sharp|=1$), we compute 
\begin{align*}
0\ \leq A_\lambda(\widetilde u)-A_\lambda(u_\sharp)&=\lambda\int_{B_r}|\widetilde u-u_*|^2-|u_\sharp -u_*|^2\,dx - \mathcal{H}^1(J_{u_\sharp}\cap B_r)\\
&\leq\ 4\lambda\int_{B_r}|\widetilde u-u_\sharp|\,dx - \mathcal{H}^1(J_{u_\sharp}\cap B_r) \\
&= 4\lambda \sum_{k\neq k_r} |1-{\bf a}^{k-k_r}||E_k\cap B_r|   - \dfrac12 \sum_{k=0}^{m-1}  \mathcal{H}^1(\partial E_k\cap B_r)\\
 &\leq  \dfrac12 \sum_{k\neq k_r} \Big\{16\lambda |E_k\cap B_r|   - \mathcal{H}^1(\partial E_k \cap B_r) \Big\}\,.
\end{align*}
Since $\mathcal{H}^1(\partial E_k \cap \partial B_r)=0$ and $(E_k)^1\cap\partial B_r=\emptyset$ for $k\neq k_r$, we infer that $E_k\cap B_r$ has finite perimeter and $\partial(E_k\cap B_r)=\partial E_k\cap B_r$ for $k\neq k_r$. Therefore,  
$$\sum_{k\neq k_r} \Big\{16\lambda |E_k\cap B_r|   - \mathcal{H}^1\big(\partial(E_k \cap B_r)\big) \Big\}\geq 0\,. $$
By the (two dimensional) isoperimetric inequality, we have 
\[
|E_k\cap B_r|\leq \sqrt{|B_r|}\sqrt{|E_k\cap B_r|}\leq \frac{r}{2} \mathcal{H}^1\big(\partial(E_k\cap B_r)\big)\,, 
\]
so that 
\[ (8\lambda r -1) \sum_{k\neq k_r}\mathcal{H}^1\big(\partial ( E_k\cap B_r)\big)\geq 0\,.
\]
Since $r<\sigma\leq  (8\lambda)^{-1}$, the  prefactor above is negative, and we deduce that $\mathcal{H}^1( \partial (E_k\cap B_r))=0$ for each $k\neq k_r$. As a consequence $B_r \subset E_{k_r}^1$, so that $u_\sharp=\widetilde u\,$.
In particular, $\mathcal{H}^1(J_{u_\sharp}\cap B_r)=0$ which  contradicts~\eqref{yadumondeautour}, and~\eqref{Claim2} is proved. 

By~\eqref{Claim2}, we can now infer from the coarea formula  (see \cite[Theorem II.7.7]{Maggi} for instance) that $ r\in(0,\sigma_0]\,$
\[
\sum_{k=0}^{m-1}\mathcal{H}^1(\partial E_k \cap B_r)\geq \sum_{k=0}^{m-1}\int_0^r \mathcal{H}^0(\partial E_k\cap\partial B_t)\,dt
 = \int_0^r \sum_{k=0}^{m-1} \mathcal{H}^0(\partial E_k\cap\partial B_t)\,dt \geq r. 
\]
Combining this inequality with~\eqref{J2EB} yields~\eqref{Alfhors0} for every $r\in(0,\sigma_0]$. 
\vskip3pt

\noindent{\it Case 2: $r\in(\sigma_0,\sigma)$.} In this case, we must have $\sigma_0<\sigma$, so that $\sigma_0={\rm dist}(0,{\rm spt}\,\mu)$ and $B_\sigma\cap {\rm spt}\,\mu\neq\emptyset$. 
Our choice of $\sigma$ then implies that $B_\sigma\cap {\rm spt}\,\mu$ is reduced to a singleton, i.e., there exists $k_0\in\{0,\ldots,m-1\}$ such that $B_\sigma\cap {\rm spt}\,\mu=\{x_{k_0}\}$. Moreover, $x_{k_0}\in \partial B_{\sigma_0}$. 

By the  definition of $\mathcal{L}_g(\Omega)$ and the slicing properties of $BV$-functions (see e.g. \cite[Chapter~3, Section~3.11]{AFP} for details), for a.e. $r\in(\sigma_0,\sigma)$ we have ${\rm spt}\,\mu\cap\partial B_r=\emptyset$, the trace $u_r:={u_\sharp}_{|\partial B_r}$ belongs to $SBV^2(\partial B_r;\Ss^1)$, and $J_{u_r}=J_{u_\sharp}\cap\partial B_r$ is finite. 
We claim that for every such $r$,
\begin{equation}\label{Claim3}
\mathcal{H}^0(J_{u_r})\geq 1.
\end{equation}
Indeed, let $r$ be such a radius and assume by contradiction that $J_{u_r}=\emptyset$. Then $u_r\in W^{1,2}(\partial B_r,\Ss^1)$ and thus $u_r$ is continuous.  In addition, 
the trace $(u^m)_r:={u_\sharp^m}_{|\partial B_r}$ belongs to $W^{1,2}(\partial B_r,\Ss^1)$ and it satisfies $(u_\sharp^m)_r=(u_r)^m$. Hence the topological degree $\ell$ of $(u_\sharp^m)_r$ is equal to $m$ times the topological degree of $u_r$. Moreover, we have $u_\sharp^m(x)=e^{i\psi}(x-x_{k_0})/|x-x_{k_0}|$ in $B_r$ for some $\psi\in W^{1,2}(B_r)$ satisfying $\psi_{|\partial B_r}\in W^{1,2}(\partial B_r)$. Hence  $\ell=1$ which contradicts $\ell\in m\Z$, and~\eqref{Claim3} is proved. 

By~\eqref{Claim3},  we can infer again from the coarea formula that 
$$ \mathcal{H}^1\big( J_{u_\sharp}\cap (B_r \setminus B_{\sigma_0})\big)\geq \int_{\sigma_0}^r \mathcal{H}^0(J_{u_\sharp}\cap\partial B_t)\,dt=\int_{\sigma_0}^r \mathcal{H}^0(J_{u_t})\,dt\geq  r-\sigma_0\quad\text{for every $r\in(\sigma_0,\sigma)$}\,.$$
Since $\mathcal{H}^1( J_{u_\sharp}\cap B_{\sigma_0})\geq \sigma_0$ by Case 1, we deduce that  
\[
\mathcal{H}^1( J_{u_\sharp}\cap B_r)= \mathcal{H}^1\big( J_{u_\sharp}\cap (B_r \setminus  B_{\sigma_0})\big)+\mathcal{H}^1( J_{u_\sharp}\cap B_{\sigma_0})\geq (r-\sigma_0)+\sigma_0=r \quad\text{for every $r\in(\sigma_0,\sigma)$}\,,
\]
and ~\eqref{Alfhors0} is proved  for every $r\in(\sigma_0,\sigma)$. 
\vskip3pt

\noindent{\it Step 3.}  
 We define 
\[
K:=\lt\{x\in\Omega \ : \ \mathcal{H}^1\big(J_{u_\sharp}\cap B_r(x)\big)\geq r \text{ for every } r \in (0,\sigma)
\rt\}\,.
\]
By definition, $K$ is closed and since ${\rm dist}(J_{u_*},\partial \Omega)>0$,  it is a compact subset of $\Omega$.  
 On the one hand, we can deduce from Step 2 that 
$\mathcal{H}^1(J_{u_\sharp}  \setminus K)=0$. On the other hand,
\[
\lim_{r\downarrow 0} \dfrac{\mathcal{H}^1(J_{u_\sharp}\cap B_r(x))}{r}=0\quad \text{for $\mathcal{H}^1$-a.e. $x\in \Omega \setminus J_{u_\sharp}$}\,,
\]
see e.g.~\cite[Sect. 2.9, Theorem 2.56 and (2.42)]{AFP}. In particular,  $\mathcal{H}^1(K \setminus J_{u_\sharp})=0$, and therefore $\mathcal{H}^1(K\triangle J_{u_\sharp} )=0$. 
\end{proof}

\subsubsection{Proof of Proposition ~\ref{proplimsup}}\label{proofofpropgamsup}

\begin{proof} Thanks to  Lemma ~\ref{lemsmoothphase}, Remark ~\ref{rem2054camp}, and Lemma~~\ref{Alternative2} and a diagonal argument, it is enough to make the construction for  a map $u\in \mathcal{L}_g(\Omega)$ with $u^m=:e^{i\vhi}v_\mu$ and $\mu=2\pi \sum_{k=1}^{md}\delta_{x_k}$ and such that
there exists $\sigma>0$ and a compact set $K\subset\Omega$ such that  
\begin{enumerate}
\vskip3pt
\item[(a)] ${\rm dist}(J_{u}, \partial\Omega)\geq 2\sigma$ and $u=g$ on $\partial \Omega$; 
\vskip3pt

\item[(b)] ${\rm dist}({\rm spt}\,\mu, \partial\Omega)\geq 2\sigma$ and $|x_k-x_l|\geq 2\sigma$ for $k\neq l$; 
\vskip3pt

\item[(c)] $u^m(x)=\xi_k(x-x_k)/|x-x_k|$ in each $B_\sigma(x_k)$ for some $\xi_k\in\Ss^1$; 
\vskip3pt

\item[(d)] $\mathcal{H}^1(K\triangle J_{u})=0$;
\vskip3pt

\item[(e)] ${\rm dist}(K, \partial\Omega)\geq 2\sigma$ and $\mathcal{H}^1\big(K\cap B_r(x)\big)\geq r$ for every $x\in K$ and $r\in(0,\sigma)$.  
\end{enumerate}
\vskip3pt

By a further diagonal argument, it is enough  to  fix $\delta>0$ and construct a sequence $\{(u_{h},\psi_{h})\}\subset  \mathcal{H}_g$ such that $\{u_{h}\}\subset \mathcal{G}_g(\Omega)\cap L^\infty(\Omega)$ with $\|u_{h}\|_{L^\infty(\Omega)}\leq 1$, $(u_{h},\psi_{h})\to (u,1)$ in $L^1(\Omega)$ as $h\to \infty$, $\p(u_{h})\to u^m$ strongly in $W^{1,p}(\Omega)$ for every $p<2$, and 
\begin{align}\label{limsupu*2b}
\displaystyle \limsup_{h\to\infty} \left\{E_{\eps_h}\big(\P(u_h)\big)-\frac{\pi d}{m} |\log\eps_h| \right\} &\leq  E_0(u) + \delta\,,\\
\label{limsupJu*b}
\displaystyle \limsup_{h\to\infty} \mathcal{H}^1(J_{u_h})& \leq  \mathcal{H}^1(J_{u})\,,\\
\label{limsuppsib}
\displaystyle \limsup_{h\to\infty} I_{\eta_h}(\psi_h)&\leq   (1+\delta) \mathcal{H}^1(J_{u})\,.
\end{align}
First, by Corollary~~\ref{corogam0} we can find $\eps\in(0,1)$, $\widetilde u\in SBV^2(B_1)\cap L^\infty(B_1)$ with $\|\widetilde u\|_{L^\infty(B_1)}\leq 1$, and a closed smooth curve $\Sigma\subset \overline{B_1}$ such that $\p(\widetilde u)(x)=x$ in a neighborhood of $\partial B_1$, $J_{\widetilde u}\subset\Sigma$, and 
$$E_{\eps}\big(\P(\widetilde u),B_1\big)- \frac\pi{m^2} \log \frac1{\eps} \leq \bgamma_m + \frac{\delta}{md}\,.$$
Then, for each $k\in\{1,\ldots,md\}$ we select $\zeta_k\in\Ss^1$ such that $\zeta^m_k=\xi_k$, and we set for $\eps_h\in(0,\sigma\eps)$,  
$$ u_h(x):=\begin{cases}
u(x) & \mbox{if }x\in \Omega \setminus B_{\eps_h/\eps}(\mu)\,,\\[5pt]
 \zeta_k \widetilde u\left(\frac{\eps}{\eps_h} (x-x_k)\right) \quad
 & \mbox{if }x\in  B_{\eps_h/\eps}(x_k)\,,\; k\in\{1,\cdots,dm\}\,.
\end{cases}$$
By construction, we have $u_h\in \mathcal{G}_g(\Omega)\cap L^\infty(\Omega)$ with $\|u_h\|_{L^\infty(\Omega)}\leq 1$ (since $|u|=1$), and 
$$\|u_h-u\|_{L^1(\Omega)}\leq 2\pi md (\eps_h/\eps)^2 \mathop{\longrightarrow}\limits_{h\to\infty} 0\,.$$
For $p<2$,  we estimate 
$$\|\p(u_h)-u^m\|_{W^{1,p}(\Omega)}\les \|u^m\|_{W^{1,p}(B_{\eps_h/\eps}(\mu))}+(\eps_h/\eps)^{(2-p)/p}\|\p(\widetilde u)\|_{W^{1,p}(B_1)} \mathop{\longrightarrow}\limits_{h\to\infty} 0\,. $$ 
Next, changing variables, we have for each $k\in\{1,\ldots,md\}$,  
$$E_{\eps_h}\big(\P(u_h); B_{\eps_h/\eps}(x_k)\big)= E_{\eps}\big(\P(\widetilde u);B_1\big)\leq  \frac\pi{m^2} \log \frac1{\eps} +\bgamma_m + \frac{\delta}{md}\,. $$
Consequently,
\begin{multline*}
E_{\eps_h}\big(\P(u_h)\big)- \frac{\pi d}{m}\log\frac{1}{\eps_h} \leq E_{\eps_h}\big(\P(u),\Omega \setminus B_{\eps_h/\eps}(\mu)\big)- \frac{\pi d}{m}\log\frac{\eps}{\eps_h} +md\bgamma_m+\delta\\
= \frac{1}{2m^2} \int_{\Omega \setminus B_{\eps_h/\eps}(\mu)}|\nabla(u^m)|^2\,dx - \frac{\pi d}{m}\log\frac{\eps}{\eps_h} +md\bgamma_m+\delta\,,
\end{multline*}
and~\eqref{limsupu*2b} follows from Proposition~~\ref{keypropcomp}. 

By construction, we have 
$$J_{u_h}\subset K_h:= K\cup \left( {\rm spt}\,\mu+\frac{\eps_h}{\eps}(\partial B_1\cup\Sigma)\right) \quad\text{up to an $\mathcal{H}^1$-null set}\,,$$
so that 
$$\mathcal{H}^1(J_{u_h})\leq  \mathcal{H}^1(K)+ md\frac{\eps_h}{\eps} \big(2\pi+ \mathcal{H}^1(\Sigma)\big)\,.$$
Since $\mathcal{H}^1(K)= \mathcal{H}^1(J_{u})$, ~\eqref{limsupJu*b} follows letting $h\to\infty$ in the inequality above. 
\vskip3pt

To produce the sequence $\{\psi_h\}$, we argue in a way similar to \cite[Sec. 5]{AmbTor90}.  We start by introducing a family of smooth profiles approximating the optimal profile $\psi_\star(s)=1-e^{-|s|}$ from Lemma~\ref{profil1Dpsi}. For $\lambda>0$ we define $\psi_\lambda:[0,\infty)\to[0,1]$ as 
\[
\psi_\lambda(t):=\begin{cases}
\displaystyle 1-{\rm exp}\left(-\frac{\lambda t}{\lambda-t}\right) & \text{if $t<\lambda$}\,,\\[5pt]
1 & \text{otherwise}\,.
\end{cases} 
\]
Notice that $1-\psi_\lambda$ and $\psi_\lambda'$ are supported on $[0,\lambda]$. Setting $e_\lambda(t):=\big(\psi^\prime_\lambda(t)\big)^2+\big(1-\psi_\lambda(t)\big)^2$, elementary computations yield 
$${s}_\lambda:=\int_0^\lambda e_\lambda(t)\,dt  \mathop{\longrightarrow}\limits_{\lambda\to\infty} 1\,.  $$ 
Hence we can find $\lambda>0$ such that ${s}_{\lambda}\leq 1+\delta$. Setting $d_h(x):={\rm dist}(x,K_h)$, we define $\psi_h$ as 
$$\psi_h(x):=\psi_{\lambda}\left(\frac{d_h(x)}{\eta_h}\right)\,. $$
The function $\psi_h$ is Lipschitz continuous and for $\eps_h\in(0,\sigma\eps)$ we have ${\rm dist}(K_h,\partial\Omega)>\sigma$ so that $\psi_h=1$ on $\partial \Omega$ whenever $\eta_h\in(0,\sigma/\lambda)$.  
Since $u_h\in W^{1,2}(\Omega \setminus K_h)$ and $\psi_h=0$ on $K_h$, we infer that $(u_h,\psi_{h})\in\mathcal{H}_g(\Omega)$.  

To estimate $I_{\eta_h}(\psi_h)$, we first notice that $d_h$ is $1$-Lipschitz. This leads to 
$$I_{\eta_h}(\psi_h)\leq \frac{1}{2\eta_h}\int_\Omega e_\lambda\big(d_h(x)/\eta_h\big)\,dx=\frac{1}{2\eta_h}\int_{K_h+B_{\lambda\eta_h}}e_\lambda\big(d_h(x)/\eta_h\big)\,dx\,.  $$
We use Fubini's theorem to obtain 

 \[\int_{K_h+B_{\lambda\eta_h}}e_\lambda\big(d_h(x)/\eta_h\big)\,dx=-\int_{K_h+B_{\lambda\eta_h}}\left(\int_{d_h(x)/\eta_h}^{\lambda}e_\lambda^\prime(t)\,dt \right)\,dx%\\
 =-\int_0^{\lambda}e_\lambda^\prime(t)|K_h+B_{t\eta_h}|\,dt\,.\]
Noticing that $e_\lambda^\prime(t)\leq 0$ and 
$$|K_h+B_{t\eta_h}|\leq |K+B_{t\eta_h}|+md\left|\frac{\eps_h}{\eps}(\partial B_1\cup\Sigma)+B_{t\eta_h}\right| \,,$$ 
we derive 
\begin{equation}\label{1849camp}
I_{\eta_h}(\psi_h)\leq -\int_0^{\lambda} te_\lambda^\prime(t)\frac{|K+B_{t\eta_h}|}{2t\eta_h}\,dt-md\int_0^{\lambda}  te_\lambda^\prime(t) \frac{\left|\frac{\eps_h}{\eps}(\partial B_1\cup\Sigma)+B_{t\eta_h}\right|}{2t\eta_h}\,dt=:J^1_h+J^2_h\,.  
\end{equation} 
By the Ahlfors regularity assumption on $K$ stated in (e), we have (see e.g. \cite[Theorem~2.104]{AFP}) 
$$\lim_{r\downarrow0} \frac{|K+B_r|}{2r}=\mathcal{H}^1(K)\,. $$
Hence, by dominated convergence 
\begin{equation}\label{1850camp}
\lim_{h\to\infty} J^1_h=-\mathcal{H}^1(K)\int_0^{\lambda}te_\lambda^\prime(t) \,dt={s}_{\lambda} \mathcal{H}^1(J_{u})\leq (1+\delta) \mathcal{H}^1(J_{u})\,. 
\end{equation}

We now claim that 
\begin{equation}\label{1851camp}
\lim_{h\to\infty} J^2_h=0\,.
\end{equation}
To prove~\eqref{1851camp} we may argue on subsequences if necessary. We distinguish the two complementary cases $\lim_h(\eps_h/\eta_h)<\infty$ and $\lim_h(\eps_h/\eta_h)=\infty$. If $\lim_h(\eps_h/\eta_h)<\infty$, then we estimate for $t\in(0,\lambda)$
$$\left|\frac{\eps_h}{\eps}(\partial B_1\cup\Sigma)+B_{t\eta_h}\right|\leq \left|B_{t\eta_h+\eps_h/\eps}\right| \les  t^2\eta_h^2+ (\eps_h/\eps)^2\leq C \lambda^2 \eta_h^2\,,$$
so that $J^2_h=O(\eta_h)$ as $h\to\infty$. We now assume that $\lim_h(\eps_h/\eta_h)=\infty$, and we write for $t>0$, 
$$\frac{\left|\frac{\eps_h}{\eps}(\partial B_1\cup\Sigma)+B_{t\eta_h}\right|}{2t\eta_h}= \left(\frac{\eps_h}{\eps}\right) \frac{\left|(\partial B_1\cup\Sigma)+B_{t\eps\eta_h/\eps_h}\right|}{2t\eps\eta_h/\eps_h}\,. $$
By smoothness of $\partial B_1$ and $\Sigma$, we have 
$$\lim_{r\downarrow 0}\frac{\left|(\partial B_1\cup\Sigma)+B_r\right|}{2r}=\mathcal{H}^1(\partial B_1\cup\Sigma)\,, $$
and we conclude that $J^2_h=O(\eps_h)$ as $h\to \infty$. In both cases~\eqref{1851camp} holds true.\smallskip\\
Eventually, putting together~\eqref{1849camp},~\eqref{1850camp}, and~\eqref{1851camp} leads to~\eqref{limsuppsib}. 
\end{proof}

%%%%%%%%%%%%%%%%%%%%%%%%%%%%%%%%%%%%%%%%%%%%%%%%%%%%%%%

\subsection{Proof of Theorem ~\ref{Gammasharp}}\label{secproofthmgamshp}

The proof of Theorem ~\ref{Gammasharp} is very similar to the proof of Theorem \ref{Gammadiff} but we include a sketch of proof for the reader's convenience. It is divided as
usual into three parts: compactness, the $\Gamma$-$\liminf$ inequality, and the construction of  recovery sequences. Concerning recovery sequences, we note that they are actually 
provided by   Proposition ~\ref{proplimsup} since~\eqref{theolimsup01} and~\eqref{theolimsup02} clearly lead to~\eqref{theolimsup01shp}. We proceed with the proof of points~{\it(i)} and~{\it(ii)} of the theorem.
\begin{proof}[Compactness, proof of (i)]
For the compactness part, we argue as in the proof of Proposition ~\ref{compactthmfield} and apply in particular Theorem \ref{compAlicPon}. We can argue  in particular that $\{\nabla u_h\}$ is bounded in $L^p(\Omega)$ and that  $\mathcal{H}^1(J_{u_h})$ is bounded. Therefore, since we also know that $\|u_h\|_{L^\infty(\Omega)}$ is bounded, we may apply \cite[Theorem 4.8]{AFP} to conclude the proof as in Theorem \ref{existsharp}.\end{proof}

\begin{proof}[The $\Gamma$-$\liminf$ inequality, proof of (ii)] Without loss of generality, we can assume that  
$$ \liminf_{h\to \infty} \widetilde F_{\eps_h}^{0}(u_h) = \lim_{h\to \infty} \widetilde F_{\eps_h}^{0}(u_h) <\infty\,.$$
Moreover, the truncation argument in the proof of Theorem ~\ref{existsharp}  together with~\eqref{estiL1difftrunc} shows that $u_h$ can be replaced by $\widehat u_h$ defined in~\eqref{deftronc}. Hence we can also assume that $\|u_h\|_{L^\infty(\Omega)}\leq 1$, and the proof of {\it (i)} above applies. In particular $v_h:=\p(u_h)$ satisfies~\eqref{asymptenerg}, and we reproduce verbatim the proof of Theorem ~\ref{Gammadiff} to show that~\eqref{TheoliminfE} holds. 

Moreover, up to extending $u_h$ and $u$ to a larger domain $\widetilde \Omega$ as in the proof of Theorem \ref{existsharp}, we may assume that $\mathcal{H}^1(\{u_h\neq g\}\cap \partial \Omega)=\mathcal{H}^1(\{u\neq g\}\cap \partial \Omega)=0$.
 We consider the sequence of non negative finite measures on $ \Omega$
$$\nu_h:=\mathcal{H}^1\res (J_{u_h\cap\Omega})\,.$$
Since $\nu_h(\Omega)=\mathcal{H}^1(J_{u_h})$ is bounded, we can extract a further subsequence such that $\nu_h\rightharpoonup\nu_*$ weakly* as measures  for some non negative finite measure $\nu_*$ on $ \Omega$. Since ${\rm spt}\,\nu_h\subset\overline\Omega$, we have ${\rm spt}\,\nu_*\subset\overline\Omega$ and $\nu_h(\Omega)\to\nu_*(\Omega)$.\\ 
We claim that 
\begin{equation}\label{ber1119}
\nu_*\geq \mathcal{H}^1\res(J_u\cap\Omega) \,.
\end{equation}
Before proving~\eqref{ber1119}, we observe that it implies 
\begin{equation}\label{ber1528}
\lim_{h\to\infty}\mathcal{H}^1(J_{u_h})=\lim_{h\to\infty} \nu_h(\Omega)=\nu_*(\Omega)
\geq \mathcal{H}^1(J_u)=\mathcal{H}^1(J_{u})\,,
\end{equation}
which, combined with~\eqref{TheoliminfE}, leads to~\eqref{Theoliminfshp}. 

To prove~\eqref{ber1119}, we fix an open set $A\subset \Omega$. Consider an arbitrary compact set $K\subset A$, and choose another open set $B$ such that $K\subset B\subset\overline B\subset A$. By the proof of {\it (i)} above, we can apply~\cite[Theorem~4.7]{AFP} in the open set $B$ to derive
$$\nu_*(A)\geq \nu_*(\overline B)\geq\liminf_{h\to\infty}\nu_h(B)=\liminf_{h\to\infty} \mathcal{H}^1(J_{u_h}\cap B)\geq \mathcal{H}^1(J_{u}\cap B)\,. $$
Hence $\nu_*(A)\geq  \mathcal{H}^1(J_{u}\cap K)$, and by inner regularity it implies that $\nu_*(A)\geq  \mathcal{H}^1(J_{u}\cap A)$. By outer regularity, we conclude that~\eqref{ber1119} holds. 

Let us now assume that $F_0(u)=\lim_{h}\widehat F^0_\eps(u_h)$. In view of~\eqref{TheoliminfE} and~\eqref{ber1528}, we have 
\begin{equation}\label{ppffbercpa}
\lim_{h\to \infty} \left\{G_{\eh}(v_h)-\frac{\pi d}{m}|\log \eh|\right\}= \frac{1}{2m^2}\int_\Omega|\nabla\vhi|^2\,dx+\frac{1}{m^2}\mathbb{W}(\mu) 
+m d\bgamma_m\,,
\end{equation}
and 
\begin{equation}\label{coucoucpas}
\nu_*(\Omega)=\lim_{h\to\infty}\mathcal{H}^1(J_{u_h})
=\mathcal{H}^1(J_{u})\,.
\end{equation}
From~\eqref{ppffbercpa}, we can argue exactly as in the proof of Proposition ~\ref{Gamliminfprop}, Step 3, to deduce that $v_h\to u^m$ strongly in $W^{1,p}(\Omega)$ for every $p<2$ and  strongly in $W^{1,2}_{\rm loc}(\Omega \setminus{\rm spt}\,\mu)$, and that~\eqref{strongconvergence} holds.
Then, we infer from~\eqref{ber1119} and~\eqref{coucoucpas} that $\nu_*\equiv \mathcal{H}^1\res J_u$. 
As a consequence, if $A\subset \R^2$ is an open set such that $\nu_*(\partial A)=0$, then 
\[
\mathcal{H}^1(J_{u}\cap A)=\nu_*(A)=\lim_{h\to\infty}\nu_h(A)
=\mathcal{H}^1(J_{u_h}\cap A)\,,
\]
which proves~\eqref{strongconvergencebisshp} since $\nu_*(\partial A)=\mathcal{H}^1(J_{u}\cap \partial A)$. 
\end{proof}

%%%%%%%%%%%%%%%%%%%%%%%%%%%%%%%%%%%%%%%%%%%%%%%%%%%%%%%
%%%%%%%%%%%%%%%%%%%%%%%%%%%%%%%%%%%%%%%%%%%%%%%%%%%%%%%
   								       						%%%%%%%%%%%%%%%%%%%%
\section{The limiting problem}\label{S4}				                          %%%%%%%%%%%%%%%%%%%%
								 						%%%%%%%%%%%%%%%%%%%
%%%%%%%%%%%%%%%%%%%%%%%%%%%%%%%%%%%%%%%%%%%%%%%%%%%%%%%
%%%%%%%%%%%%%%%%%%%%%%%%%%%%%%%%%%%%%%%%%%%%%%%%%%%%%%%

The aim of this section is to study minimizers over $\mathcal{L}_g(\Omega)$ of the limiting functional $F_{0,g}$. To avoid some technical issues (at the boundary), we shall assume for simplicity that $\Omega$ is a smooth bounded {\sl convex} set. 

By Remark ~\ref{remminF0} and Lemma ~\ref{struct}, minimizers of $F_{0,g}$ over $\mathcal{L}_g(\Omega)$ coincide with solutions of 
\begin{multline}\label{minproblim}
\min\bigg\{\frac{1}{m^2} \W(\mu)+\mathcal{H}^1(J_u)+\mathcal{H}^1\big(\{u\neq g\}\cap\partial \Omega): u\in SBV(\Omega;\Ss^1)\,,\\
 \mu:=m\,{\rm curl}\,j(u)\in\mathcal{A}_d\text{ and } u^m=v_\mu \bigg\}\,.
\end{multline}
In turn,~\eqref{minproblim} amounts to solve for each $\mu\in\mathcal{A}_d$, 
\begin{equation}\label{probNd}
 L(\mu):=\min\Big\{ \mathcal{H}^1(J_u) + \mathcal{H}^1\big(\{u\neq g\}\cap \partial \Omega\big) : u\in SBV(\Omega, \Ss^1)\,,\; u^m=v_\mu \Big\}\,, 
\end{equation}
and then minimize $\frac{1}{m^2}\mathbb{W}(\cdot)+L(\cdot)$ over $\mathcal{A}_d$ which is a finite dimensional optimization problem. Let us however point out that since on the one hand 
Steiner type problems are usually very hard to solve and since on the other hand, the minimization of $\mathbb{W}$ can be rarely explicitly done (see \cite{Radulescu}), 
this finite dimensional problem does not seem   easy to handle.   
  
\vskip3pt

For the rest of this section we fix a measure $\mu\in\mathcal{A}_d$, and focus on problem~\eqref{probNd}. First, we notice that existence of minimizers in~\eqref{probNd} follows as in the proof of Lemma ~\ref{Alternative2} (Step 1) since $|\nabla v_\mu|\in L^p(\Omega)$ for every $p<2$ and $|\nabla u|=\frac{1}{m} |\nabla v_\mu|$ for any admissible competitor $u$ (see Lemma ~\ref{struct}). 
We will prove that minimizers of \eqref{probNd} are related to a variant of the Steiner problem that we now describe. 
\vskip3pt

Write $\mu=:2\pi\sum_{k=1}^{md}\delta_{x_k}$, and recall that the $x_k$'s are distinct points of the domain $\Omega$. We let 
\begin{multline}\label{verydeflambmu}
 \Lambda(\mu):=\min\lt\{ \mathcal{H}^1(\Gamma) \ : \ {\rm spt} \mu \subset \Gamma,\rt.\\
 \lt. \textrm{ and for every connected component } \Sigma \textrm{ of } \Gamma, \, {\rm Card}(\Sigma\cap {\rm spt}\,\mu)\in m\N\rt\}.
\end{multline}
Notice that since we can always remove connected components which do not contain any vortex $x_k$, we can reduce the above minimization problem to sets $\Gamma$ with the property that every connected component contains a positive number of vortices. Of course, this implies that $\Gamma$ has at most $d$ connected components. 
Each of these connected components $\Sigma$ is a competitor for the Steiner problem related to  ${\rm spt} \mu \cap \Sigma$. This shows that minimizers of \eqref{verydeflambmu} exist and are made of (at most $d$) Steiner trees i.e. finite union of segments meeting only at the vortices or at triple junctions (see \cite[Prop. 2.2]{MorMed}).

\begin{definition}\label{defdmumin}
A compact set $\Gamma$ is said to be a $\Lambda(\mu)$-minimizer if it solves \eqref{verydeflambmu}. 
\end{definition}

\begin{remark}\label{remStein}
Since we assumed that  $\Omega$ is convex, any $\Lambda(\mu)$-minimizer $\Gamma$ is contained in the convex hull of ${\rm spt}\,\mu$. More precisely, since projecting on convex sets reduces  distances, if $\Sigma$ is a connected component of 
$\Gamma$, then $\Sigma$ is contained in the convex hull of $\Sigma\cap {\rm spt}\,\mu$.  Since $\Sigma$ is a tree, we also infer the following property: if  $C\subset \Sigma$ is an open  segment, such that  $C\cap {\rm spt}\,\mu=\emptyset $, then $\Sigma \setminus C$ is made of two connected components
$A$ and $B$ satisfying  ${\rm Card}(A\cap {\rm spt}\,\mu)\not\in m\N \setminus\{0\}$ and  ${\rm Card}(B\cap {\rm spt}\,\mu)\not\in m\N \setminus\{0\}$. Otherwise, $\Gamma \setminus C$ would be an admissible competitor for $\Lambda(\mu)$ with strictly lower length, contradicting minimality. 
\end{remark}

We are now ready to prove the main result of this section, which states that  the jump set of any minimizer of~\eqref{probNd} is  a $\Lambda(\mu)-$minimizer.
\begin{theorem}\label{thmLDmu}
 Assume that $\Omega$ is a smooth, bounded, and convex open set. For every $\mu\in\mathcal{A}_d$, it holds
 \begin{equation}\label{equalNd}
   L(\mu)=\Lambda(\mu)\,.
 \end{equation}
Moreover, if $u$ is a minimizer of $L(\mu)$, then $J_u$ is a $\Lambda(\mu)$-minimizer, $u\in C^\infty(\overline\Omega \setminus J_u)$, and $u=g$ on $\partial\Omega$. 
Vice-versa, if $\Gamma$ is a $\Lambda(\mu)$-minimizer, then there exists $u$ minimizing  $L(\mu)$ with $J_u=\Gamma$.
\end{theorem}

\begin{proof}
{\it Step 1.} Let $\mu=2\pi\sum_{k=1}^{md}\delta_{x_k}$ be fixed. We first prove that 
$$L(\mu)\leq \Lambda(\mu)\,.$$
Consider  $\Gamma$ a $\Lambda(\mu)$-minimizer.  Treating each connected component of $\Gamma$ separately, we may assume without  loss of generality that $\Gamma$ is connected. 
Since $\R^2 \setminus \Gamma$ is connected, we can find a smooth injective curve with arc-length parameterization $\gamma_1:[0,\infty)\to\R^2$ satisfying $\gamma_1(0)=x_1$, $|\gamma_1(t)|\to\infty$ as $t\to\infty$, and $\gamma_1(0,\infty)\cap \Gamma=\emptyset$. 
Setting $D_1:=\gamma_1((0,\infty))$, we orient $D_1$ according to its parameterization $\gamma_1$ (i.e., in the direction of increasing $t$'s). Since $\R^2 \setminus \overline D_1$ is simply connected, we can find a smooth map $\vhi_1:\R^2 \setminus \overline D_1\to\R$ 
which is smooth up to $D_1$ from both sides, has a constant (oriented, pointwise defined) jump across $D_1$ equal to $2\pi$, and such that   $e^{i\vhi_1(x)}=(x-x_1)/|x-x_1|$. 
We then set $u_1:=e^{i\vhi_1/m}$. 

Since $\Gamma$ is a tree, for each $k\in\{2,\ldots,md\}$ there is a unique injective polygonal curve $\gamma_k:[0,1]\to \Gamma$ such that $\gamma_k(0)=x_k$ and $\gamma_k(1)=x_1$. Setting $D_k:=\gamma_k((0,1))$, we orient $D_k$ according to the curve $\gamma_k$. Notice that for $k\neq l$, the orientation of $D_k$ coincides with the orientation of $D_l$ on $D_k\cap D_l$. Moreover,  one has $\Gamma=\cup_{k\geq 2}\overline D_k$ by minimality of $\Gamma$. As a consequence, $\Gamma$ inherits the orientation induced by the $D_k$'s. 

Since $\R^2 \setminus (\overline D_1\cup\overline D_k)$ is simply connected, we can find a smooth map $\vhi_k:\R^2 \setminus (\overline D_1\cup\overline D_k)\to\R$,  smooth up  to $\overline D_1\cup D_k$ from both sides, with a constant (oriented) jump across $\overline D_1\cup D_k$ equal to $2\pi$, and such that   $e^{i\vhi_k(x)}=(x-x_k)/|x-x_k|$. We consider 
\[
u_{\Gamma}:=\exp \lt(\dfrac im\lt[\vhi_\mu + \sum_{k=1}^{md} \vhi_k\rt]\rt),
\]
where $\vhi_\mu$ is the map defined in ~\eqref{definivmu}. 
By construction, we have $u_\Gamma\in SBV(\Omega;\Ss^1)$ and  $u_\Gamma^m=v_\mu$, i.e., $u_\Gamma$ is an admissible competitor for $L(\mu)$. In addition, $u_\Gamma$ is smooth outside $\Gamma\cup D_1$. Since  $\frac1m\sum_k\vhi_k$ has a constant jump equal to $2\pi d$  across $D_1$, we infer that $u_\Gamma$ is actually smooth in $\overline\Omega \setminus \Gamma$, and $u={\bf a}^jg$ on $\partial \Omega$ for some $j\in\{0,\ldots,m-1\}$. 
Replacing $u_\Gamma$ by ${\bf a}^{-j}u_\Gamma$ if necessary, we can assume that $u_\Gamma=g$ on $\partial \Omega$. Now consider an arbitrary point 
$x\in\Gamma \setminus{\rm spt}\,\mu$. By Remark ~\ref{remStein}, the number of $x_k$'s before $x$, according to the orientation of $\Gamma$, is not a multiple of $m$.
This shows that the jump of $\frac1m\sum_k\vhi_k$ across $\Gamma$ at $x$ is not a multiple of $2\pi$, and consequently $x\in J_{u_\Gamma}$. Therefore $J_{u_\Gamma}=\Gamma$, 
and $L(\mu)\leq \mathcal{H}^1(J_{u_\Gamma}) =\Lambda(\mu)$.  
\vskip5pt

\noindent{\it Step 2.} 
We now  prove that 
\begin{equation}\label{finalstep}L(\mu)\ge \Lambda(\mu).\end{equation}
Let us consider an arbitrary minimizer $u$ of $L(\mu)$, and assume 
without loss of generality that $0\in\Omega$. We shall first prove in this step that 
\begin{equation}\label{claimSubordom}
u=g \text{ on }\partial \Omega\,, \text{ and }{\rm dist}(J_u,\partial \Omega)>0\,.
\end{equation}

To show that \eqref{claimSubordom} holds, we consider $\Gamma$ a $\Lambda(\mu)$-minimizer, and $u_\Gamma$ the map constructed in Step 1. We extend $u$ and $u_\Gamma$ to $\R^2$ by setting 
$\widetilde u(x)=g\circ\Pi(x)$ and $\widetilde u_\Gamma(x)=g\circ\Pi(x)$ for $x\in\R^2 \setminus \Omega$, where $\Pi$ denotes the orthogonal projection on the convex set $\overline\Omega$. 
Then, $J_{\widetilde u}\subset \overline\Omega$, $J_{\widetilde u_\Gamma}\subset \Gamma$, and $\mathcal{H}^1(J_{\widetilde u})= \mathcal{H}^1(J_{u})+\mathcal{H}^1\big(\{u\neq g\}\big)$. We can thus from now on identify $u$ (respectively $u_\Gamma$) and $\widetilde u$  (respectively $\widetilde u_\Gamma$). We define 
$$w:= u /  u_\Gamma \,.$$
Since $ u^m=v_\mu= u_\Gamma^m$, we have $w\in SBV_{\rm loc}(\R^2;\mathbf{G}_m)$ with $w=1$ in $\R^2\setminus\overline\Omega$ and we can find a  Caccioppoli partition $\{E_k\}_{k=0}^{m-1}$ of $\R^2$ such that 
$$w=\sum_{k=0}^{m-1}{\bf a}^k\chi_{E_k} \,,$$ 
with $E_k\subset \overline\Omega$ for $k=1,\ldots,m$, and $\R^2\setminus\overline\Omega\subset E_0$. Since $ u_\Gamma$ is smooth in $\overline\Omega\setminus\Gamma$, we deduce that 
\begin{equation}\label{identjumpawgam}
J_{ u}\cap(\R^2\setminus\Gamma)=\bigcup_{k=0}^{m-1}\partial E_k\setminus\Gamma\quad\text{up to an $\mathcal{H}^1$-null set}\,.
\end{equation}
Let $K$ be the convex envelope of ${\rm spt}\,\mu$. The set $K$ is then a closed polygonal subset of $\Omega$. By an elementary geometric construction, we can find a strictly convex open set 
$\omega$ with $C^1$-boundary satisfying $K\subset \omega$ and $\overline\omega\subset\Omega$, and such that the Hausdorff distance between $K$ and $\overline\omega$ is arbitrarily small. Given such~$\omega$, we consider the mapping $\Phi:\R^2\to\R^2$ defined by $\Phi(x):=\frac{1}{2}(x+\Pi_\omega(x))$, where $\Pi_\omega$ denotes the orthogonal projection on $\overline\omega$. Then $\Phi$ is a (global) $C^1$-diffeomorphism of $\R^2$ satisfying $\Phi(x)=x$ for every $x\in\overline\omega$. 

Consider now the sets $\widehat E_k:= \Phi(E_k)$ for $k=1,\ldots,m-1$, so that $\{\widehat E_k\}_{k=0}^{m-1}$ defines a Caccioppoli partition of $\R^2$. As a consequence, the map 
$$\widehat u:=\left(\sum_{k=0}^{m-1}{\bf a}^k\chi_{\widehat E_k}\right) u_\Gamma $$
is an admissible competitor for $L(\mu)$. 
By the chain rule formula for $BV$-functions and \eqref{identjumpawgam}, we have $J_{\widehat u}\setminus\overline\omega=\Phi(J_{ u}\setminus\overline\omega)$ and $J_{\widehat u}\cap\overline\omega=J_{ u}\cap\overline\omega$. The minimality of $u$ together with the  area formula (see \cite[Theorem~2.91]{AFP}) then leads to 
\begin{equation}\label{vendrfinpresq}
\mathcal{H}^1(J_{ u}\setminus\overline\omega)\leq \mathcal{H}^1(J_{\widehat u}\setminus\overline\omega)\leq\int_{J_{ u}\setminus\overline\omega}|\nabla \Phi|\,d\mathcal{H}^1 \,.
\end{equation}
 Our assumption on $\omega$ implies that  $|\nabla\Phi(x)|\leq 1-c_\omega{\rm dist}(x,\overline\omega)$ for every $x\in \overline\Omega$ and some constant $c_\omega>0$ depending only on $\omega$ and $\Omega$.
 Inserting this estimate in \eqref{vendrfinpresq} shows that $\mathcal{H}^1(J_{ u}\setminus\overline\omega)=0$, which clearly implies  \eqref{claimSubordom}.

\vskip5pt

\noindent{\it Step 3.} In this final step we show \eqref{finalstep}. Let us  fix an arbitrary ball $B_{2r}(y)\subset\Omega\setminus{\rm spt}\,\mu$. 
Since $v_\mu$ is smooth in $B_{2r}(y)$, we can find a smooth function $\varphi$ on $\overline B_{r}(y)$ such that $v_\mu= e^{i\varphi}$ in $\overline B_{r}(y)$. The map $u_\star:=e^{i\varphi/m}$ is then 
smooth on $\overline B_{r}(y)$, and satisfies $u_\star^m=v_\mu$ in $\overline B_{r}(y)$. Arguing as above, it implies that any competitor $u_{\rm comp}$ for $L(\mu)$ can be written as 
$$u_{\rm comp}=\left( \sum_{k=0}^m {\bf a}^k \chi_{F_k}\right) u_\star\quad\text{in $B_r(y)$}\,,$$
for some Caccioppoli partition $\{F_k\}_{k=0}^{m-1}$ of $B_r(y)$, and 
$$J_{u_{\rm comp}}\cap B_r(y)= \bigcup_{k=0}^{m-1}\partial F_k\cap B_r(y)\quad\text{up to an $\mathcal{H}^1$-null set} \,.$$ 
In addition, 
\begin{equation}\label{Caccioppoli}\mathcal{H}^1\big(J_{u_{\rm comp}}\cap B_r(y)\big)=\frac{1}{2}\sum_{k=0}^{m-1}\mathcal{H}^1\big(\partial F_k\cap B_r(y)\big)\,.\end{equation}
As a consequence, the minimizer $u$ of $L(\mu)$ that we consider can be written as  $u=\big(\sum_{k}{\bf a}^k\chi_{E_k}\big)u_\star$, for some  Caccioppoli partition $\{E_k\}_{k=0}^{m-1}$ of $B_r(y)$ minimizing the right-hand side of \eqref{Caccioppoli}. 
By classical results on minimal planar clusters (see for instance \cite[Theorem~5.2]{CLM}), $\cup_k\partial E_k\cap B_{r}(y)$ is locally a finite union of  segments meeting at 
 triple junctions. Since $u_\star$ is smooth in $B_r(y)$, it  implies that $J_u\cap B_{r/2}(y)=\cup_k\partial E_k\cap B_{r/2}(y)$, and $u\in C^\infty(B_{r/2}(y)\setminus J_u)$.  
\vskip3pt

We are now ready to prove that $J_u$ is a $\Lambda(\mu)$-minimizer. Let us fix for a moment $\sigma>0$ such that $\overline B_\sigma(x_k)\cap \overline B_\sigma(x_l)=\emptyset$ for $k\neq l$, and $\overline B_\sigma(\mu)\subset\Omega$. By the discussion above, $u\in C^\infty\big(\overline\Omega\setminus (J_u\cup B_{\sigma/2}(\mu))\big)$, and $J_u\setminus B_{\sigma/2}(\mu)$
is a finite union of  segments. Hence, $J_u\cup \overline B_\sigma(\mu)$ is a compact subset of $\Omega$. Since $J_u\cup \overline B_\sigma(\mu)$ converges to $J_u\cup{\rm spt}\,\mu$ as $\sigma\to 0$ in the Hausdorff distance, 
we infer from  Blaschke's theorem that $J_u\cup{\rm spt}\,\mu$ is a compact subset of $\Omega$.
Moreover, we have $J_u\supset {\rm spt}\,\mu$  since $u^m=v_\mu$. Therefore $J_u$ is a compact subset of $\Omega$, and $u\in C^\infty(\Omega\setminus J_u)$. To complete the proof, it now only remains to prove that
any connected component of $J_u$ contains a multiple of $m$ vortices (possibly equal to zero). Indeed, this would lead to $L(\mu)=\mathcal{H}^1(J_u)\geq \Lambda(\mu)$, and \eqref{finalstep} would be proven. 
Furthermore, by Step 1, we would also obtain that $J_u$ is a minimizer of $\Lambda(\mu)$.
\vskip3pt

Let us consider $\Sigma$ a connected component of $J_u$, and $A\subset\Omega$ a connected smooth open neighborhood of $\Sigma$ such that 
$(J_u\setminus\Sigma)\cap \overline A=\emptyset$.  We may write $A=A_0\setminus \cup_{n=1}^N\overline A_n$ where the $A_n$ are  connected and simply 
connected smooth open sets satisfying $\overline A_n\subset A_0$ for $n=1,\ldots, N$, 
and $\overline A_n$ are pairwise disjoint. Since $v_\mu$ and $u$ are smooth on $\partial A_n$ for $n=0,\ldots, N$, and $u^m=v_\mu$,
\[
 {\rm deg}(v_\mu, \partial A_n)=m {\rm deg}(u, \partial A_n)\in m \N
\]
and thus
\[
 {\rm Card}(\Sigma \cap{\rm spt}\,\mu)={\rm deg}(v_\mu,\partial A_0)-\sum_{n=1}^N{\rm deg}(v_\mu,\partial A_n)\in m \N,
\]
concluding the proof.
\end{proof}

\begin{remark}\label{remsmoothLmumin}
The proof of  Theorem ~\ref{thmLDmu} shows that every $L(\mu)$-minimizer $u$ is smooth on both sides of $J_u$ away from ${\spt}\,\mu$. More precisely, one can find a radius $r>0$
such that if  $\overline B_r(x)\subset \Omega \setminus {\rm spt}\,\mu$, then 
$B_r(x) \setminus J_u$ is made of at most three connected sets and  $v_\mu=e^{i\vhi}$ in $B_r(x)$ for some smooth function $\vhi$. In each connected region of   $B_r(x) \setminus J_u$, we have $u={\bf a}^ke^{i\vhi/m}$ for some $k\in\{0,\ldots,m-1\}$. 
\end{remark}

\begin{remark}
 When $\Omega$ is simply connected but not convex,~\eqref{equalNd} still holds true if one adds the condition $\Gamma\subset \overline{\Omega}$ for the admissible sets for $\Lambda(\mu)$. 
 For minimizers, the set $\Gamma\cap\partial\Omega$ can then be non-empty.  
 The proof would follow the same lines as in the convex case 
 using  boundary regularity of minimizers for the constrained Steiner  and constrained minimal cluster problems. 
\end{remark}
\begin{remark}
Given a reference map $u_\star$ which is an admissible competitor for $L(\mu)$, we have seen that any other competitor $u$ can be written as $u=(\sum_{k=0}^{m-1}{\bf a}^k\chi_{E_k})u_\star$ for some Caccioppoli partition $\{E_k\}_{k=0}^{m-1}$ of~$\Omega$. This allows to rephrase  the minimization problem defining $L(\mu)$ as an optimal partition problem. 
Notice however that $\mathcal{H}^1(J_u)$ does not coincide in general with the boundary length of the partition plus $\mathcal{H}^1(J_{u_\star})$  since 
 possible cancellations have to be taken into account (see Figure ~\ref{fig:cancel}).
\end{remark}

\begin{figure}\begin{center}
  \resizebox{7.2cm}{!}{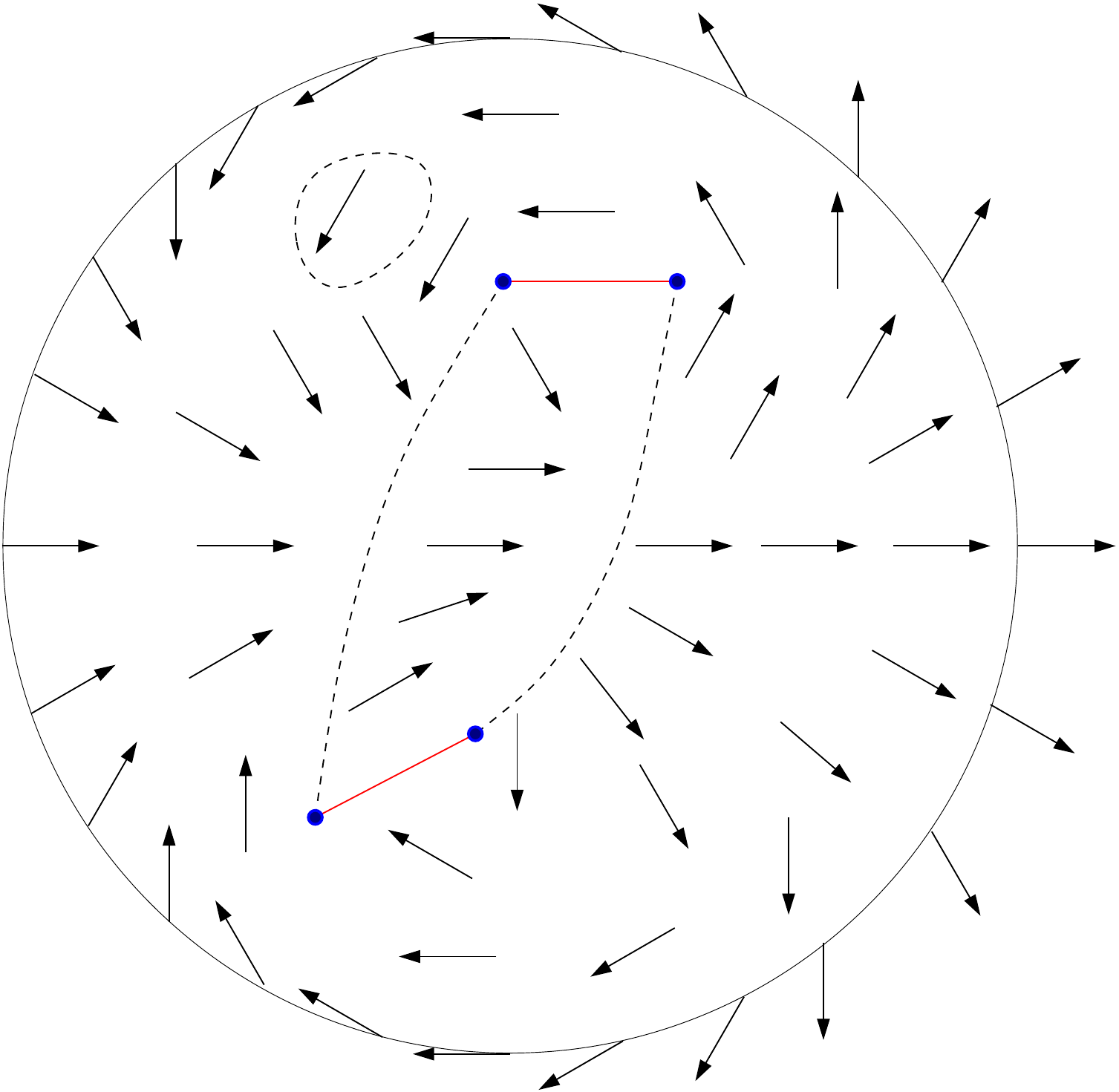} 
  \hspace{1.cm}
   \resizebox{7.2cm}{!}{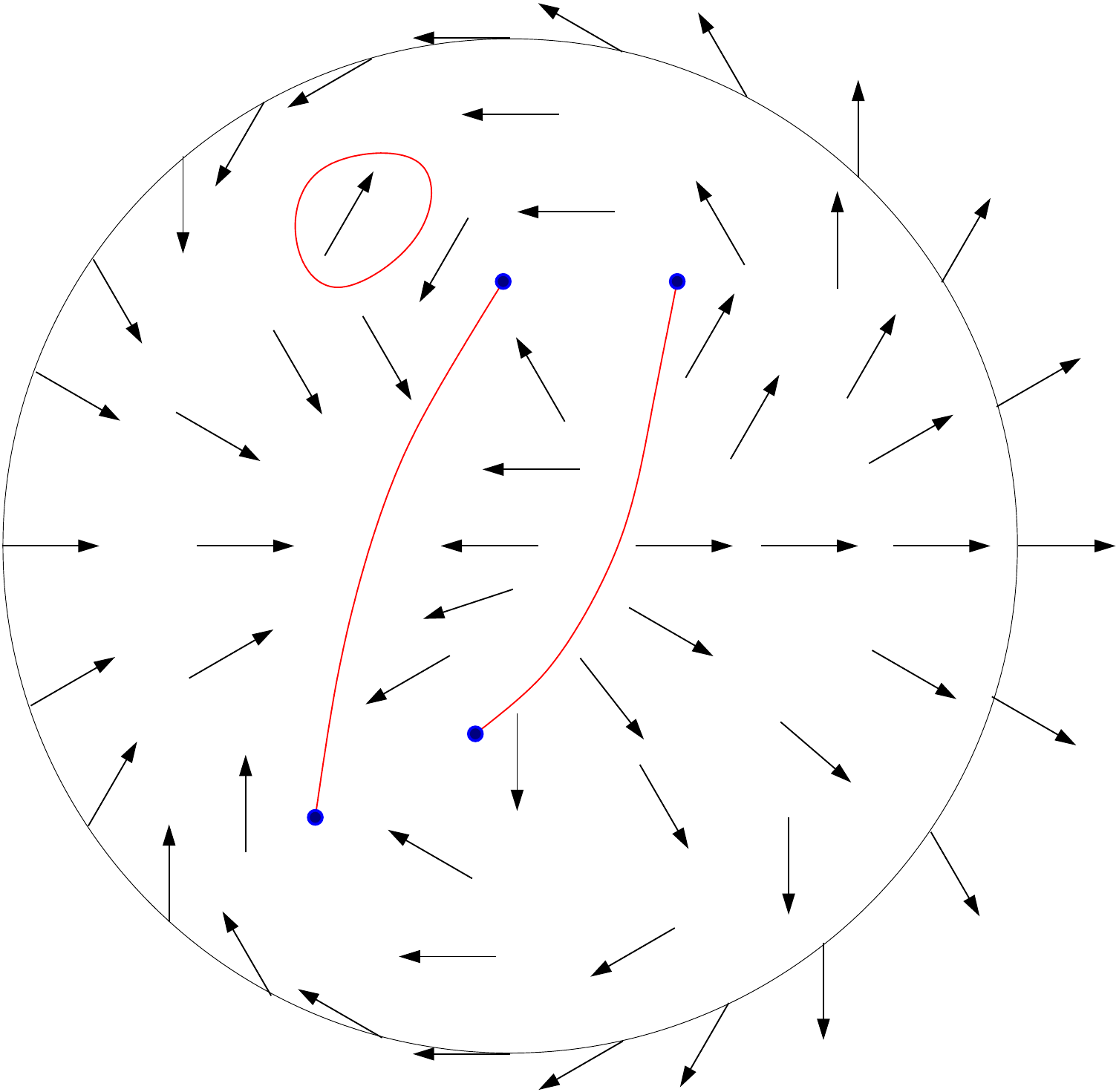}
  \caption{Cancellations in the case $m=2$ ($E_1=E$, $E_0=\Omega \setminus E$).} \label{fig:cancel}
 \end{center}
 \end{figure}

\subsection{Structure of $\Lambda(\mu)$-minimizers.}
We now move on to the study of the $\Lambda(\mu)$-minimizers. In the case $m=2$, it reduces to a variant of the classical minimal connection problem (see for instance~\cite{BreCorLieb}). We recall that if $P:=\{p_1,\dots, p_d\}$ and $Q:=\{q_1,\dots, q_d\}$ are two sets of given points in $\R^2$, then the {\sl length of a minimal connection} between $P$ and $Q$ is defined as 
\[\min_{\sigma} \sum_{k=1}^d |p_k-q_{\sigma(k)}|\,,\]
where the minimum runs over all permutations $\sigma$ of $\{1,\dots, d\}$.

\begin{proposition}\label{propminconnect}
 Assume that $m=2$. Let $\mu\in\mathcal{A}_d$ and $\Gamma$ be a $\Lambda(\mu)$-minimizer. Then, $\Gamma$ is made of exactly $d$ disjoint segments $\Gamma_1,\cdots,\Gamma_d$, and each $\Gamma_k\cap{\rm spt}\,\mu$ contains exactly two points $\{p_k,q_k\}$. 
 In particular, $\Gamma_k=[p_k,q_k]$ for each $k$, and $\mathcal{H}^1(\Gamma)$ is the length of a minimal connection between $P=\{p_1,\cdots,p_d\}$ and $Q=\{n_1,\cdots,q_d\}$. 
\end{proposition}
\begin{proof}
 Let $\Gamma$ be a $\Lambda(\mu)$-minimizer, and let us prove that every connected component $\Gamma_k$ of $\Gamma$ contains exactly two points of ${\rm spt}\,\mu$. It would obviously imply that each $\Gamma_k$ is a  segment, and that $\mathcal{H}^1(\Gamma)$ is the length of minimal connection by the definition \eqref{verydeflambmu} of $\Lambda(\mu)$.  
 
 To prove the claim, we start with the following observation. By Theorem \ref{thmLDmu}, we can find a map $u$ achieving $L(\mu)$ and such that $J_u=\Gamma$. Then, consider an arbitrary open ball $B_{2r}(x)\subset \Omega\setminus {\rm spt}\,\mu$. Since $v_\mu=u^2$ is smooth in that ball and ${\rm deg}(v_\mu,\partial B_{2r}(r))=0$, we can find $u_\star\in C^\infty(B_{2r}(x);\Ss^1)$ such that $u_\star^2=v_\mu$. Arguing as in the proof of Theorem    \ref{thmLDmu}, we infer that $u=(\chi_E-\chi_{E^c})u_\star$ in $B_r(x)$ for some set $E$ having a minimizing perimeter in $B_r(x)$. By minimality, $\partial E\cap B_{r/2}(x)= J_u\cap B_{r/2}(x)=\Gamma\cap B_{r/2}(x)$ is smooth, and thus $\Gamma\cap B_{r/2}(x)$ does not contain triple junctions. Hence $\Gamma$ is a finite union of segments, only intersecting at points of ${\rm spt}\,\mu$.

Let us now consider $\Gamma_k$ a connected component of $\Gamma$. Assume by contradiction that there is a point $x\in \Gamma_k\cap \spt\,\mu$ such that $J\geq 2$ segments meet at  $x$ (if $J=1$ for every point of $ \Gamma_k\cap \spt\,\mu$, then there is nothing to prove).  For $j\in\{1,\ldots,J\}$,
let $C_j:=[x,y_j]$ be the segments in $\Gamma_k$ departing from $x$. 
Denote by $n_j$ the number of points in ${\rm spt}\,\mu$ belonging to the connected component of $\Gamma_k \setminus \{x\}$ containing   $C_j\setminus\{x\}$. Notice that each $n_j$
must be odd (otherwise one could remove the corresponding segment $C_j$ from~$\Gamma$, thus contradicting minimality).
Moreover, the cardinal of ${\rm spt}\,\mu\cap \Gamma_k$ is even, and since it is equal to $1+\sum_{j=1}^J n_j$, we deduce that $J$ is odd.  
Hence $J\geq 3$, and among the segments $C_j$, at least two of them 
are not collinear. Assume without loss of generality that $C_1$ and $C_2$ are not collinear. Then we can replace $C_1$ and 
$C_2$ by the segment $[y_1,y_2]$ to obtain a competitor with strictly lower length than $\Gamma$ (see Figure ~\ref{fig:mindegtwo}),  which again contradicts minimality.  This establishes that  $J=1$, and concludes the proof. 
 \end{proof}
 
  \begin{figure}\begin{center}
 \resizebox{8.cm}{!}{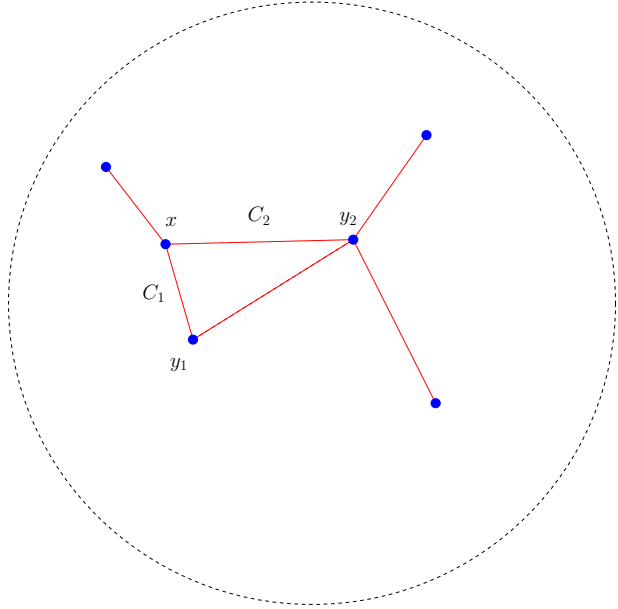} 
   \caption{Construction of a competitor.} \label{fig:mindegtwo}
 \end{center}
 \end{figure}

The case $m\geq 3$ is more involved, and it is no longer true that any $\Lambda(\mu)$-minimizer is a disjoint union of $d$ Steiner trees. 

\begin{proposition}\label{propexmin}
 Assume that  $d\in\{2,3,4\}$ and  $m\geq d+1$. There exists $\mu\in\mathcal{A}_d$ such that every $\Lambda(\mu)$-minimizer is connected. 
\end{proposition}

\begin{proof}
For clarity reason,  we shall start by giving full details of the proof for $d=2$ and $m=3$. We will then explain how to extend this construction to the other cases. Let $Y_1$, $Y_2$, and $Y_3$ be three equidistant  points on the unit circle. The unique solution to the Steiner problem for  connecting  these three points is given by the triple junction $\Sigma$.  
 Given $\eps\ll1$,  let $\{x_1, \dots, x_6\}$ be such that (see Figure~~\ref{ex:Y})
 \[|x_1-Y_1|=|x_2-Y_1|=|x_3-Y_2|=|x_4-Y_2|=|x_5-Y_3|=|x_6-Y_3|=\eps\, .\]
 \begin{figure}
 \resizebox{5.8cm}{!}{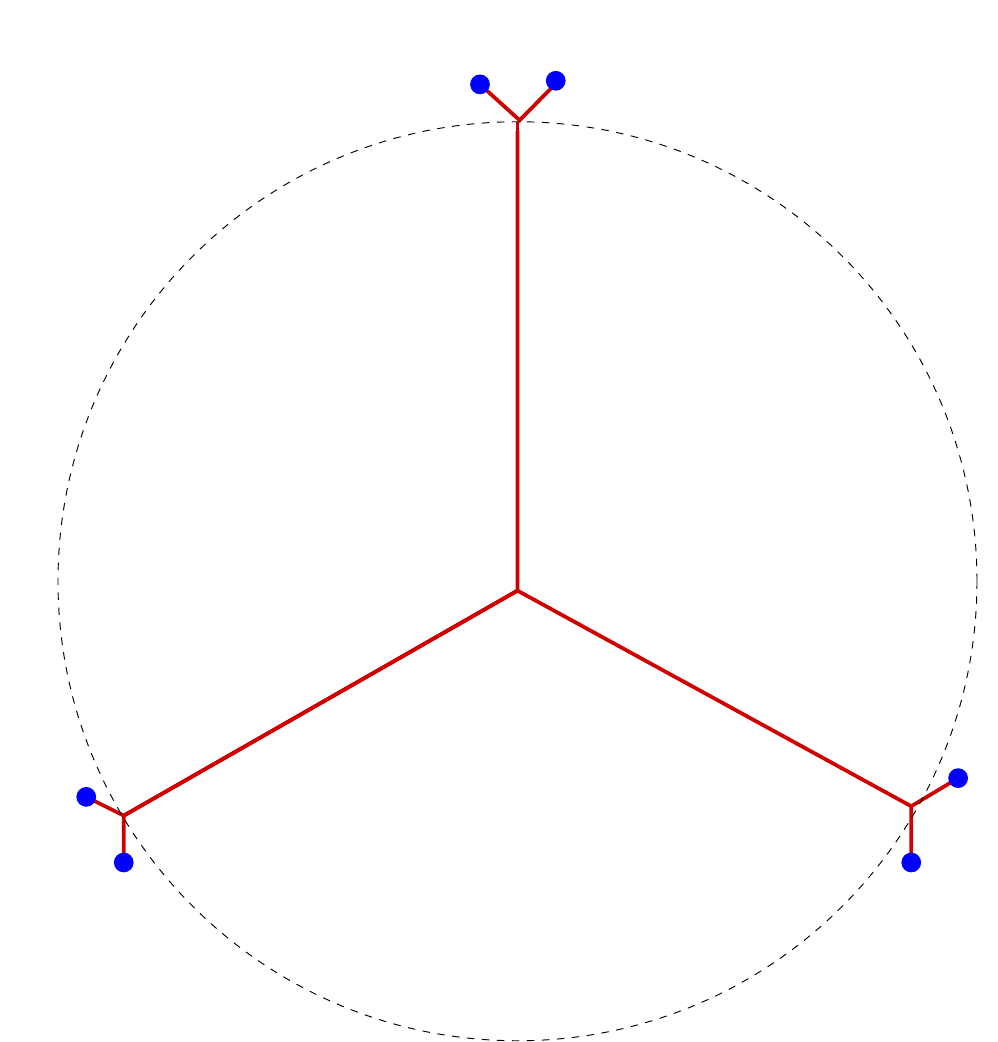} 
 \hspace{2cm}
  \resizebox{6.cm}{!}{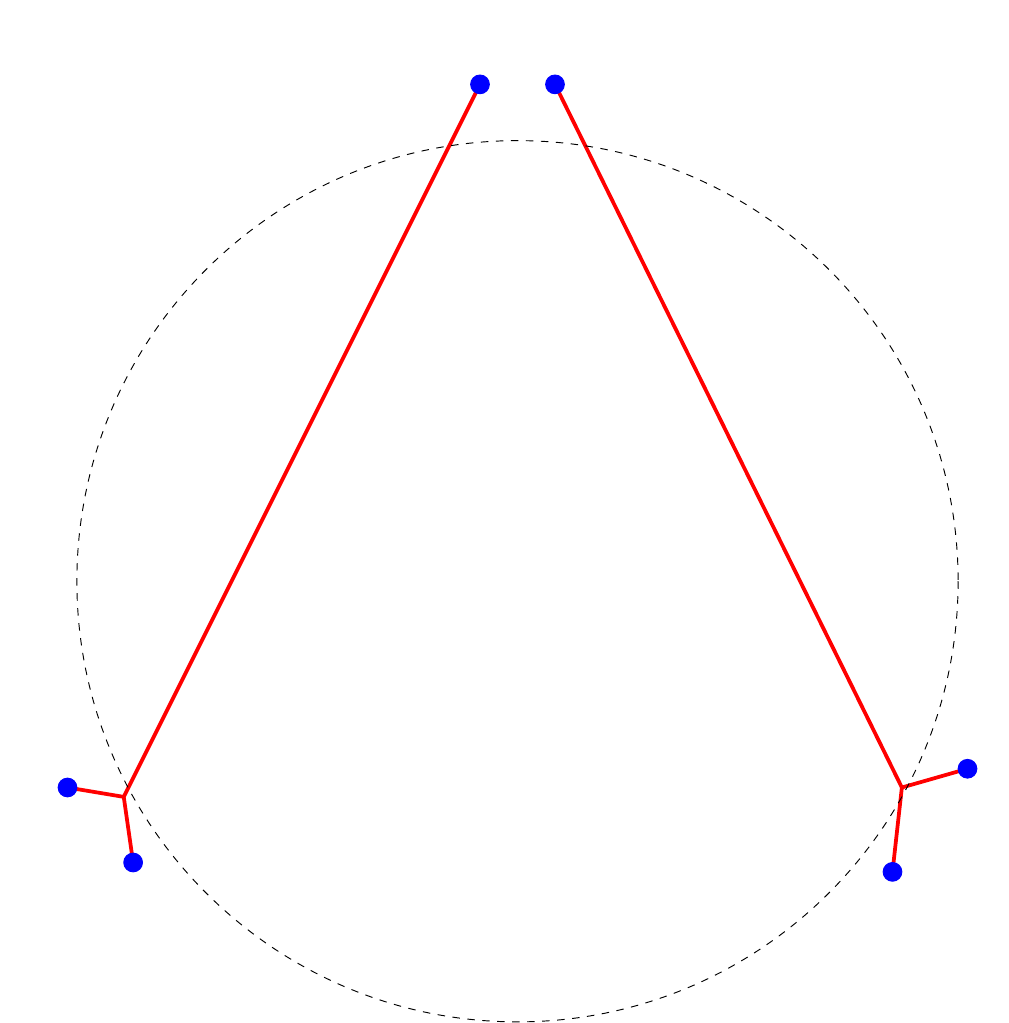}
  \caption{An example of a connected minimizer with six vortices of degree $1/3$.} \label{ex:Y}
 \end{figure}
Consider the measure $\mu_\eps:=2\pi\sum_{k=1}^6\delta_{x_k}$, and let $\Gamma_\eps$ be a $\Lambda(\mu_\eps)$-minimizer. Set $\Sigma_\eps$ to be the union of $\Sigma$ with the segments connecting each $x_j$ to the closest $Y_i$. By minimality, 
 \begin{equation}\label{uperconnectY}
  \mathcal{H}^1(\Gamma_\eps)\leq  \mathcal{H}^1(\Sigma_\eps)\leq \mathcal{H}^1(\Sigma)+6\eps\,.
 \end{equation}
If $\Gamma_\eps$ is not connected, then it is has two connected components $\Gamma_\eps^1$ and $\Gamma_\eps^2$, each of them containing exactly three points among the $x_j$'s. 
Then, at least one of the pairs $\{x_1,x_2\}$, $\{x_3,x_4\}$ and $\{x_5,x_6\}$,  intersects both $\Gamma_\eps^1$ and $\Gamma_\eps^2$, say $\{x_1,x_2\}$. Up to a subsequence, we have that each 
$\Gamma_\eps^i$ converges to a connected set $\Gamma^i$ with $\Gamma:=\Gamma^1\cup \Gamma^2$ admissible for the Steiner problem related to $Y_1$, $Y_2$ and $Y_3$. Therefore, by \eqref{uperconnectY}
\[
 \mathcal{H}^1(\Sigma)\ge \liminf_{\eps\to 0} \mathcal{H}^1(\Gamma_\eps)\ge \mathcal{H}^1(\Gamma)\ge \mathcal{H}^{1}(\Sigma).
\]
Hence, the above inequalities are actually equalities and since
\[
 \liminf_{\eps\to 0} \mathcal{H}^1(\Gamma_\eps)=\mathcal{H}^1(\Gamma^1)+\mathcal{H}^1(\Gamma^2),
\]
we have $\mathcal{H}^1(\Gamma^1\cup\Gamma^2)=\mathcal{H}^1(\Gamma^1)+\mathcal{H}^1(\Gamma^2)$ so that $\Gamma^1$ and $\Gamma^2$ only intersect at $Y_1$.
We have thus obtained a connected graph $\Gamma$ containing $Y_1$, $Y_2$ and $Y_3$ but for which the degree of $Y_1$ is two. Since $\Sigma$ is 
the only minimizer of the Steiner problem for $(Y_1, Y_2, Y_3)$,
\[
 \mathcal{H}^1(\Gamma)>\mathcal{H}^1(\Sigma)\,,
\]
contradicting~\eqref{uperconnectY} for $\eps$ small enough.
\begin{center} 
 \begin{figure}[h]
 \resizebox{8.cm}{!}{\includegraphics{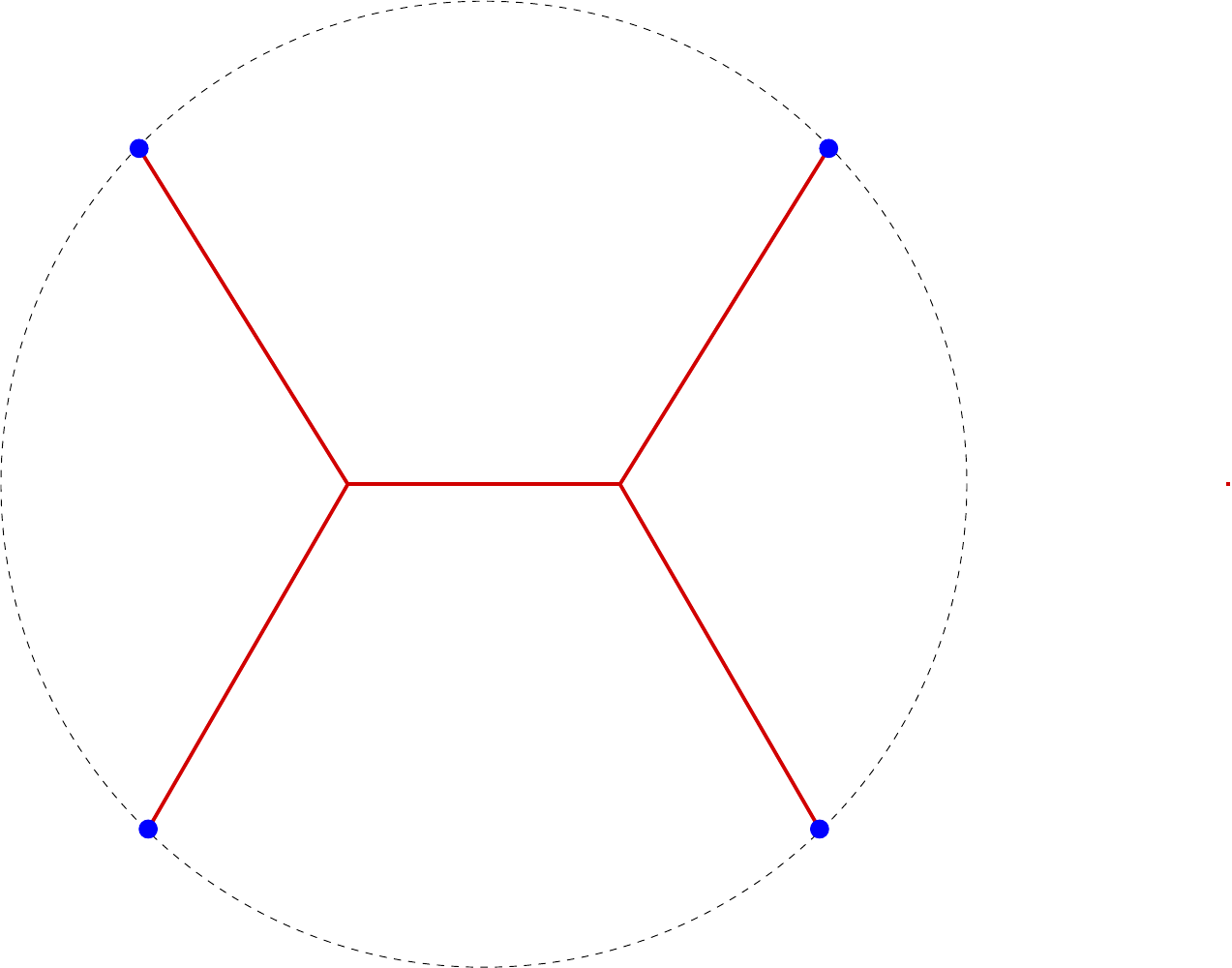}} 
 \hspace{0.1cm}
  \resizebox{6.cm}{!}{\includegraphics{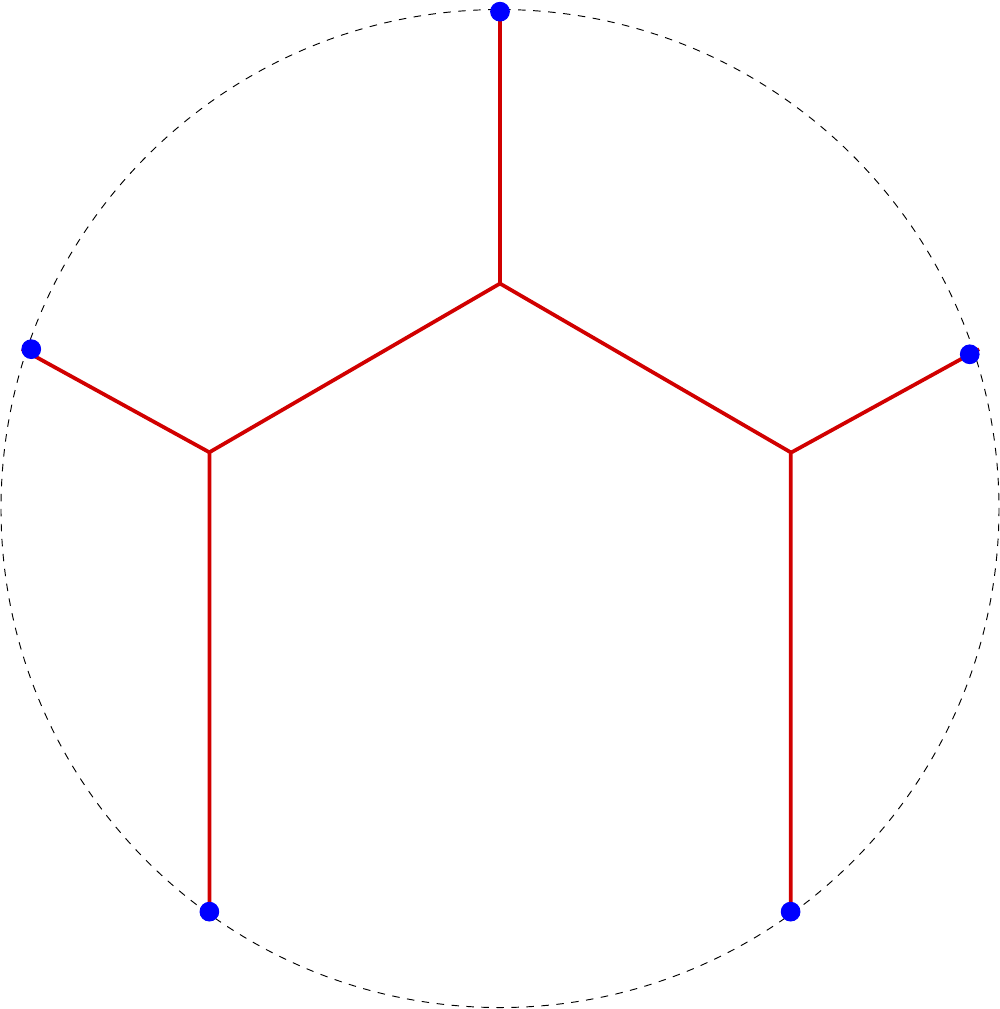}}
  \caption{Minimizers of the Steiner problem for the vertices of the square on the left and of the regular pentagon on the right.} \label{Steinm}
 \end{figure}
 \end{center}
If now $m\geq 3$, we can repeat the same construction placing $m-1$ points close to $Y_1$, $m-1$ close to $Y_2$ and two close to $Y_3$ to construct an example
 where the minimizer is connected for $2m$ points.
 
 For $d=3,4$ and $m=d+1$, the construction is similar to the case $d=2$. For this we let $(Y_1,\dots, Y_m)$ be the vertices of a regular $m-$gone\footnote{That is a square if $m=4$ and a pentagon if $m=5$.} 
 inscribed in the unit circle and consider $(x_1, \dots, x_{md})$ points such that for every $k\in [1,m]$, 
 there are exactly $d$ of these points at distance $\eps$ from $Y_k$. Let $\Sigma$ be the minimizer of the Steiner problem for $(Y_1, \dots, Y_m)$. Then, as for the case $d=2$,
 \begin{equation}\label{uperconnectSig}
  \Lambda(\mu)\leq \mathcal{H}^1(\Sigma)+md \eps\,.
 \end{equation}
 As above, let $\Gamma_\eps$ be a minimizer for $\Lambda(\mu)$ and assume that it is not connected. Let $\widehat{\Gamma}_\eps$ be the set made of $\Gamma_\eps$ and the union of the segments joining the points $x_j$ to the nearest $Y_k$.
 Since every connected component of $\Gamma_\eps$ contains a multiple of $m$ points among the points $x_j$, it must also be the same for the connected components of $\widehat{\Gamma}_\eps$. However, at the 
 same time, it should also be a multiple of $m-1$ since each $Y_k$ is connected to the $m-1$ closest points $x_j$. Therefore, $\widehat{\Gamma}_\eps$ must be connected. Letting $\eps\to 0$, we obtain
 a set $\widehat{\Gamma}$ which is admissible for the Steiner problem for $(Y_1,\dots, Y_m)$ but for which at least one of the points $Y_k$ has degree at least two. Since all the points $Y_k$ have degree one for the minimizer 
 of the Steiner problem for the  $m-$gone  with $m=4 ,5$ (see for instance~\cite{DHW} and Figure~~\ref{Steinm}), we reach a contradiction with~\eqref{uperconnectSig}. 
 The extension to $d=4,5$ and $m>d+1$ is obtained as before by placing $m-1$ points $x_j$ at distance $\eps$ from $d$ points $Y_k$ and $d$ points $x_j$ at distance $\eps$ from the last $Y_k$.  
\end{proof}

\begin{remark}
 For $m\geq 6$, the solution of the Steiner problem for the vertices of a regular $m-$gone is known to be the $m-$gone itself minus one of its side~\cite{DHW}. For this reason our construction does not work for $d\geq 5$.
 It would be interesting to understand if it is possible to find another construction which works for every $d\in \N$.
\end{remark}

In light of Proposition~~\ref{propexmin}, one could conjecture that the maximum number of points that a $\Lambda(\mu)$-minimizer can carry is equal to $m(m-1)$. However, as the following example shows, this is again not the case.

\begin{proposition}\label{propexmin2}
 Assume that $m=d=3$. There exists $\mu\in\mathcal{A}_d$ such that every $\Lambda(\mu)$-minimizer is connected. 
\end{proposition}

\begin{proof}
 The idea is to iterate the construction made above (see Figure~~\ref{ex:9.1}). Let $1\gg \eps\gg \delta$. We   first fix the points $(Y_1,Y_2, Y_3)$ as before and then choose  the points $(X_1, X_2, X_3, X_4)$ so that 
 \[|X_1-Y_1|=|X_2-Y_1|=|X_3-Y_2|=|X_4-Y_2|=\eps\,,\]
 and that all the angles are of $120$\textdegree. Let $(x_1, \dots, x_{8})$ be such that each $X_k$ is at distance $\delta$ of exactly two $x_j$'s and  let finally $x_9=Y_3$. Let $\Gamma$ be a minimizer of the corresponding $\Lambda(\mu)$. 
 As above, by comparing with  a connected competitor, it holds
 \begin{equation}\label{uperconnectYbis}
  \mathcal{H}^1(\Gamma)\leq 3+4\eps +8\delta\,.
 \end{equation}
If $\Gamma$ is not connected, then it can have either two or three connected components.
 Arguing as above, we see that the connected component containing $x_9$ must also contain at least one of the vortices close to $Y_1$, say $x_1$ and one of the vortices close to $Y_2$ say $x_8$. Let $\Gamma_1$ be this connected component. 
 If $\Gamma$ has three connected components, then each of them must contain exactly three points. Up to relabeling, this means that $\Gamma_2$ contains 
 $x_2, x_3$ and $x_4$ and $\Gamma_3$ contains $x_5, x_6$ and $x_7$. Letting $\Sigma_1$ be the triple junction connecting $x_9$, $X_1$ and $X_3$ (see Figure~~\ref{ex:9.1}), 
 we obtain that,
 \[\mathcal{H}^1(\Gamma)\geq \mathcal{H}^1(\Sigma_1)+|X_1-X_2|+|X_3-X_4|+O(\delta)\,.\]
  \begin{figure}
 \resizebox{6.2cm}{!}{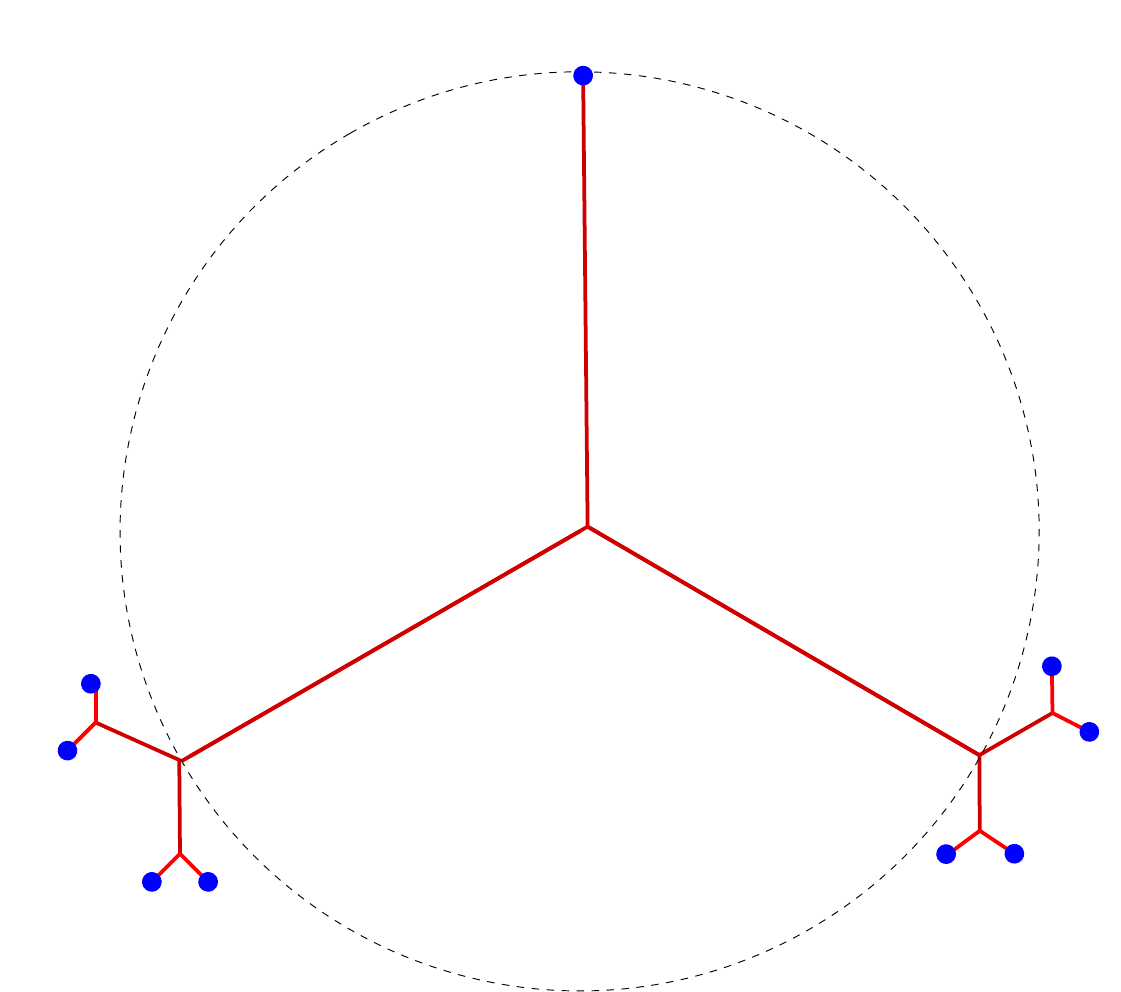} 
 \hspace{2cm}
  \resizebox{6.2cm}{!}{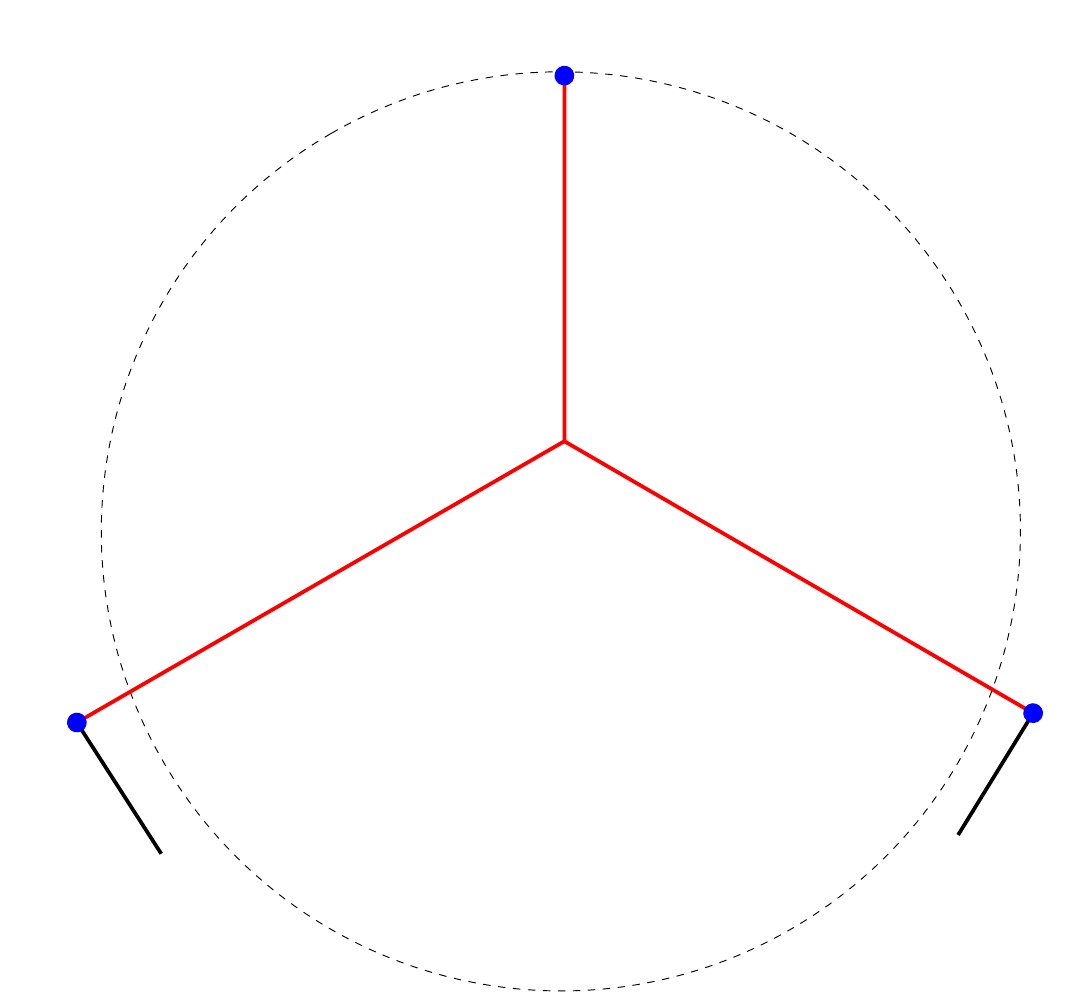}
   \caption{An example of a connected minimizer with nine vortices of degree $1/3$.} \label{ex:9.1}
 \end{figure}
 A simple computation gives $|X_1-X_2|=|X_3-X_4|=\sqrt{3}\eps$ and $\mathcal{H}^1(\Sigma_1)=3+\eps$ so that $\mathcal{H}^1(\Gamma)\geq 3+(1+2\sqrt{3})\eps-O(\delta)$,
 which contradicts~\eqref{uperconnectYbis} for $\delta$ and $\eps$ small enough. The cases when $\Sigma_1$ must be the triple junction connecting $X_2, x_9$ and $X_3$ or $X_1$, $x_9$ and $X_3$ can be treated analogously.
 
 If now $\Gamma$ is made of only two components, then $\Gamma_1$ must contain six points and the other connected component $\Gamma_2$ must contain the remaining three points. Without loss of generality, we can assume that
 $\{x_1, x_2, x_3, x_4, x_8 , x_9\}\subset \Gamma_1$ and $\{x_5, x_6, x_7\}\subset \Gamma_2$. Let $\Sigma_2$ be the optimal Steiner tree connecting $X_1$, $X_2$ $x_9$ and $X_4$ (see Figure~~\ref{ex:9.2}). We then have 
 \[\mathcal{H}^1(\Gamma)\geq \mathcal{H}^1(\Sigma_2)+|X_4-X_3|-O(\delta)\,.\]
 \begin{figure}
 \centering\resizebox{7.4cm}{!}{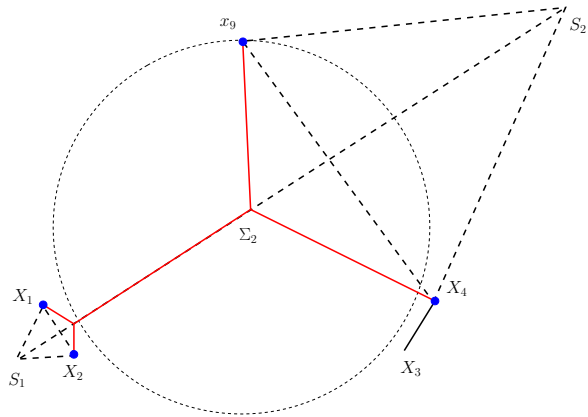}
   \caption{The set $\Sigma_2$.} \label{ex:9.2}
 \end{figure}
 In order to compute $\mathcal{H}^1(\Sigma_2)$, we notice that since at first order, it must have length $3$, it must have at least one triple junction. If it has only one, then we are basically back to the situation of Figure~~\ref{ex:9.1}.
 Otherwise, it has exactly two triple junctions and we can obtain $\mathcal{H}^1(\Sigma_2)$ by constructing the two equilateral triangles $X_1X_2S_1$ and $X_4x_9S_2$ (see Figure~~\ref{ex:9.2}) and computing the distance $S_1S_2$ (see~\cite{pollak}). 
 After some computations using for instance complex numbers, we find that $\mathcal{H}^1(\Sigma_2)= 3+\frac{5}{2} \eps +o(\eps)$ so that 
 \[\mathcal{H}^1(\Gamma)\geq 3+\lt(\frac{5}{2}+\sqrt{3}\rt) \eps +o(\eps)-O(\delta)\,,\]
 contradicting~\eqref{uperconnectYbis} again.
\end{proof}

\begin{remark}
In light of these examples, it would be interesting to understand what is the maximal number of vortices which can be carried by a single tree, given $m\geq3$.
\end{remark}

%%%%%%%%%%%%%%%%%%%%%%%%%%%%%%%%%%%%%%%%%%%%%%%%%%%%%%%
%%%%%%%%%%%%%%%%%%%%%%%%%%%%%%%%%%%%%%%%%%%%%%%%%%%%%%%
   								       						%%%%%%%%%%%%%%%%%%%%
\section{Structure of minimizers at small $\eps>0$}\label{S5}                       %%%%%%%%%%%%%%%%%%%%
								 						%%%%%%%%%%%%%%%%%%%
%%%%%%%%%%%%%%%%%%%%%%%%%%%%%%%%%%%%%%%%%%%%%%%%%%%%%%%
%%%%%%%%%%%%%%%%%%%%%%%%%%%%%%%%%%%%%%%%%%%%%%%%%%%%%%%

The aim of this final section is to use the structure of the minimizers of the limiting functional $F_{0,g}$ given by Theorem \ref{thmLDmu} to prove that  minimizers of $F^0_{\eps,g}$ have the same structure for $\eps>0$  small enough. 
In turn, this gives an improved convergence result for minimizers as $\eps\downarrow 0$ (compare to Corollary \ref{corcvminsharp}). 
Since we will use  some tools developed for the analysis of the Mumford-Shah functional, we will only focus on the sharp interface functional $F_{\eps,g}^0$. It would be interesting to understand if similar results can be obtained for the 
``phase field approximation" $F_\eps^\eta$.

As in the previous section, we shall assume that $\Omega$ is a convex domain. The main results of this section can be summarized in the following theorem. We recall that the $L^1$-convergence of minimizers of $F_{\eh,g}^0$ towards minimizers of $F_{0,g}$ is given 
by Corollary \ref{corcvminsharp}. 

\begin{theorem}\label{theoregeps}
Let $\eps_h\to 0$, and let $u_h$ be a minimizer of $F_{\eh,g}^0$ over ${\mathcal G}_g(\Omega)$. Assume that $u_h\to u$ in $L^1(\Omega)$ as $h\to\infty$ for some minimizer $u$ of $F_{0,g}$. Setting $\mu:=\curl j(u^m)$, for every $\sigma>0$ small enough, and every $h$ large enough,  the following holds:  
\begin{enumerate}
\item[(i)] $J_{u_h}\setminus B_\sigma(\mu)$ is a compact subset of $\Omega\setminus  B_\sigma(\mu)$ made of finitely many  segments,  meeting at triple junctions. 
\vskip3pt
\item[(ii)] $u_h\in C^\infty\big(\overline \Omega \setminus(B_\sigma(\mu)\cup J_{u_h}) \big)$ and $u_h=g$ on $\partial \Omega$. 
\vskip3pt
\item[(iii)] If $B_r(x)\subset \Omega\setminus  B_\sigma(\mu)$, then there exists $\phi_h\in BV(B_r(x);\mathbf{G}_m)$ such that $u_h/\phi_h\in C^\infty(B_r(x))$. 
\end{enumerate}
In addition, 
\begin{enumerate}
\item[(iv)] $ J_{u_h}$ converges in the Hausdorff distance to $J_{u}$. 
\vskip3pt
\item[(v)] $u_h\to u$ in $C^k_{\rm loc}(\Omega\setminus J_{u})\cap C^{1,\alpha}_{\rm loc}(\overline \Omega\setminus J_{u})$  for every $k\in \mathbb{N}$ and $\alpha\in(0,1)$. 
\vskip3pt
\item[(vi)] If $B_r(x)\subset \Omega\setminus  B_\sigma(\mu)$, then there exists $\phi\in BV(B_r(x);\mathbf{G}_m)$ such that 
$u_h/\phi_h\to u/\phi$ in $C^k_{\rm loc}(B_r(x))$ for every $k\in \mathbb{N}$. 
\end{enumerate}
\end{theorem}
\vskip3pt

\begin{remark}\label{betdescripthm3}
In the proof of Theorem \ref{theoregeps}, we are actually going to  prove a stronger result on the structure of $u_h$ (see Section \ref{sec:sketch}).  As a consequence, it solves in $\Omega\setminus B_\sigma(\mu)$ the Ginzburg-Landau system with free discontinuities
\begin{equation*}
\begin{cases}
\displaystyle -\Delta u_h=\frac{1}{\eps_h^2}(1-|u_h|^2)u_h & \text{in $\Omega\setminus(B_\sigma(\mu)\cup J_{u_h})$}\,,\\[8pt]
(u_h^+)^m=(u^-_h)^m & \text{on $J_{u_h}\setminus B_\sigma(\mu)$}\,,\\
u_h= g & \text{on $\partial\Omega$}\,.
\end{cases}
\end{equation*} 
\end{remark}
\vskip3pt

\begin{remark}
As a consequence of items (iii) and (vi) above, we have $\p(u_h)\in C^\infty(\overline \Omega\setminus B_\sigma(\mu))$ for $h$ large enough, and $\p(u_h)\to u^m$ 
in $C^k_{\rm loc }(\Omega\setminus B_\sigma(\mu))\cap C^{1,\alpha}(\overline\Omega\setminus B_\sigma(\mu))$. 
\end{remark}

\begin{remark}
In the case $m=2$,  $J_u$ is made of $d$ disjoint segments connecting points of ${\rm spt}\,\mu$, see Proposition \ref{propminconnect}. In particular, $J_u$ contains no triple junctions. Concerning the minimizer $u_h$ of $F_{\eps,g}^0$, it implies that (for $h$ large enough), the set $J_{u_h}\setminus B_\sigma(\mu)$ is made of $d$ disjoint segments connecting components of $\overline B_\sigma(\mu)$. 
\end{remark}

\begin{remark}
 Let us notice that in the case $\deg(g,\partial \Omega)=0$, Theorem~\ref{theoregeps} shows that for $\eps$ small enough, the minimizer of $F^0_{\eps,g}$ is unique and smooth, i.e. there is no jump set, and it coincides with the unique minimizer of $E_\eps$ over $W^{1,2}_g(\Omega)$ (see \cite{Yezu}). 
 For the classical Mumford-Shah functional, similar results were obtained using calibration methods \cite{AlBouDal} (see also \cite[Th. 3.1]{fuscoreview} for a simple proof originally due to Chambolle).  
\end{remark}

\begin{remark}
Theorem \ref{theoregeps} does not provide regularity results near ${\rm spt}\,\mu$. Nevertheless, repeating verbatim the proof of~\cite[Theorem 3.1]{BucLuc}, one can prove that 
 for every $\eps>0$ and every minimizer $u$ of $F_{\eps,g}^0$, the jump set  is essentially closed, that is $\mathcal{H}^1(\overline{J_u}\setminus J_u)=0$. 
Since $\Omega\subset \R^2$, the proof  of this result only requires the simplest forms of \cite[Lemma 2.3 \& Lemma 2.4]{BucLuc}. 
\end{remark}

\begin{remark}
\label{rem:difficultcore}
It would be interesting to study the behavior of the minimizers $u_h$ close to the vortices, i.e., in $B_\sigma(\mu)$. One could expect that there is only one point in each component of $B_\sigma(\mu)$ where $u_h$ vanishes, and that the jump set of $u_h$ is a union of Steiner trees connecting those zeroes in the spirit of the $\Lambda(\mu)$ minimization problem. In this direction, a first step may consist in 
understanding the optimal profile problem 
\begin{multline*}
\gamma_m^\#(\eps,R):=\min \bigg\{
E_\eps\big(\P(u),B_R\big) + {\mathcal H}^1(J_u \cap B_R) -\dfrac\pi{m^2} \log \dfrac{R}\eps :   \\
u\in\mathcal{G}(B_R)\,,\; \P(u)(z)=\frac{1}{m}\Big(\frac{z}{|z|},\sqrt{m^2-1}\Big)\mbox{ on }\partial B_R
\bigg\}\,.
\end{multline*}
Considering a solution $u$ of this problem, one may ask if $|u|$ is radial, increasing, vanishing at the origin, and  if  $J_u$ is just a segment joining the origin to the boundary. It seems to be  a difficult
question since it combines both issues related to the presence of an expected singularity in the jump set in the spirit of the so called crack tip  (see for instance \cite{David}) for the Mumford-Shah functional, with the fact
that $\P(u)$ should have the same regularity as minimizing harmonic maps with values into the singular  cone $\mathcal{N}$. Such harmonic maps satisfy non standard elliptic equations, and are usually more singular than minimizing harmonic maps with values into a smooth target \cite{Gromov,Lindefects,HardtLin, AlperHL}\footnote{quoting \cite{Lindefects}: ''Unfortunately, the equations
satisfied by $s$ and $u$ are so bad that no existing result can be applied``.}.
\end{remark}

\subsection{Sketch of the proof of Theorem~\ref{theoregeps}.}\label{sec:sketch}
Before starting the proof of Theorem~\ref{theoregeps}, let us explain the strategy. Away from $\spt \mu$, the limiting function $v_\mu$ is smooth. Therefore, 
if we consider a small enough ball $B_r(x)$ outside $B_\sigma(\mu)$, then the oscillation of $v_\mu$ on this ball is very small. By the strong convergence in $W^{1,2}(B_r(x))$ of $v_h$ to $v_\mu$ (recall Theorem~~\ref{Gammasharp}),
this will still be true for $v_h$ on $\partial B_r(x)$ (actually, on $\partial B_{\rho_h}(x)$ for some $\rho_h\sim r$). Hence\footnote{Notice that actually, some care is needed in the choice of $g_h$ to guarantee that no jump is created at the boundary.}, we can  find $g_h\in W^{1,2}(\partial B_r(x))$ 
such that $\p(g_h)=v_h$ on $\partial B_r(x)$. Considering $w_h$ a solution of the Ginzburg-Landau equation
\begin{equation}\label{GLeq}
\lt\{
 \begin{array}{rcl}
  -\Delta w_h&=&\frac{1}{\eh^2}(1-|w_h|^2)w_h\quad \textrm{in } B_r(x)\\
  w_h&=&g_h \quad  \textrm{in }\partial B_r(x),
 \end{array}
 \rt.
\end{equation}
we aim at proving that in $B_r(x)$, $u_h=\lt(\sum_{k=0}^{m-1} \ba^k \chi_{E_{h}^k}\rt)w_h $ where the $E_{h}^k$ are pairwise disjoint and  satisfy (up to a relabeling)
\begin{itemize}
\item[(i)] if $B_{r}(x)\cap J_{u}=\emptyset$, then $E_h^0=B_r(x)$ i.e. $u_h= w_h$ in $B_r(x)$;
\item[(ii)] if $B_{r}(x)\cap J_{u}$ is a segment then 
$u_h=(\chi_{E_h^0}+\ba^k(1-\chi_{E_{h}^0}) )w_h$ for some $ k\neq 0$ with $\partial E_h^0 \cap B_r(x)$ a segment;
\item[(iii)] if $B_{r}(x)\cap J_{u}$ contains a triple point then  $u_h=(\chi_{E_h^0}+\ba^{k_1} \chi_{E_h^{k_1}}+\ba^{k_2} \chi_{E_h^{2}} ) w_h$ with $0<k_1<k_2\leq m-1$ 
and $\partial E_{h}^0\cup \partial E_{h}^{k_1}\cup \partial E_{h}^{k_2}$ a triple junction in $B_r(x)$.
\end{itemize}
In order to show that indeed $u_h$ 
is of this form, a powerful tool that we introduce in the next section is the Lassoued-Mironescu decomposition argument~\cite{LasMi} which allows to conveniently split the energy into a Ginzburg-Landau term and a Mumford-Shah type energy.

 \subsection{Ginzburg-Landau minimizers and energy splitting}
 In this section, we consider a radius $r>0$, a sequence $\eps_h\to0$, and a sequence of boundary conditions $\{g_h\}\subset W^{1,2}(\partial B_r)\cap L^\infty(\partial B_r)$ satisfying 
 \begin{equation}\label{cond1Linftygh}
 \|g_h\|_{L^\infty}(\partial B_r) \leq 1\,,
 \end{equation}
 \begin{equation}\label{condghenergunif}
 \int_{\partial B_r} |\partial_\tau g_h|^2+\frac{1}{\eps^2_h}(1-|g_h|^2)^2\,d\mathcal{H}^1 \leq C\,, 
 \end{equation}
 for some constant $C>0$ independent of $\eps_h$. We further assume that 
 \begin{equation}\label{unifcvgh}
 g_h\to g_\star\quad\text{uniformly on $\partial B_r$ as $h\to \infty$\,,} 
 \end{equation}
 for some $g_\star\in W^{1,2}(\partial B_r;\Ss^1)$ satisfying 
 \begin{equation}\label{conddeg0g*}
 {\rm deg}(g_\star,\partial B_r)=0 \,.
 \end{equation}
 From this last assumption, we can write $g_\star=e^{i\varphi_\star}$ for some harmonic function $\varphi_\star\in W^{1,2}(B_r)$ (which is unique up to a constant multiple of $2\pi$).  As in \cite{BBHarticle}, the map 
 $$w_\star:=e^{i\varphi_\star}\in W^{1,2}_{g_\star}(B_r;\Ss^1) $$
 is the unique solution of the minimization problem 
 \begin{equation}\label{pbharmmaplimtw}
 \min_{W^{1,2}_{g_\star}(B_r;\Ss^1)}\int_{B_r}|\nabla w|^2\,dx\,.
 \end{equation}
 We are now interested in the minimization problem 
 \begin{equation}\label{minpbbdgh}
 \min_{w\in W_{g_h}^{1,2}(B_r)} E_{\eps_h}(w,B_r)\,.
 \end{equation}
 We recall that minimizers of \eqref{minpbbdgh} are in particular solutions of \eqref{GLeq}.
 We shall make an essential use of the following proposition. It constitutes a slight extension of \cite[Theorem~2]{BBHarticle} to the case of a boundary condition which merely belongs to $W^{1,2}(\partial B_r)$.
Since the estimates obtained in  \cite[Theorem~2]{BBHarticle} only depend  on the $W^{1,2}(\partial B_r)$ bounds satisfied by $g_h$, the proof of Proposition~\ref{buildGLcompet}   readily follows  from 
\cite[Theorem~2]{BBHarticle} together with  an approximation argument (to regularize the boundary condition).
 
 \begin{proposition}\label{buildGLcompet}
 Assume that  \eqref{cond1Linftygh}, \eqref{condghenergunif}, \eqref{unifcvgh}, and \eqref{conddeg0g*} hold. There exists $\{w_h\}\subset W^{1,2}(B_r)\cap C^0(\overline B_r)\cap C^\infty(B_r)$ such that $w_h$ solves \eqref{minpbbdgh}, and 
 \begin{equation*}
 w_h\to w_\star\quad\text{strongly in $W^{1,2}(B_r)$}\,,
 \end{equation*}
 \begin{equation*}
|w_h|\to 1 \quad\text{uniformly in $\overline B_r$}\,,
 \end{equation*}
 \begin{equation*}
w_h\to w_\star\quad\text{in $C^k_{\rm loc}(B_r)$ for every $k\in\mathbb{N}$}\,.
 \end{equation*}
 \end{proposition}

For the rest of this subsection, we still denote by $w_h$ a solution of \eqref{minpbbdgh} obtained from Proposition\ref{buildGLcompet}. We continue with a very useful energy decomposition, originally introduced in \cite{LasMi}. 

  \begin{lemma}\label{LM}
Let $u\in \mathcal{G}(B_r)\cap L^\infty(B_r)$ be such that  $\p(u)=\p(g_h)$ on $\partial B_r$.  For $\eps_h$ small enough, we have  $u=w_h\phi$ for some $\phi\in \mathcal{G}(B_r)\cap L^\infty(B_r)$ satisfying $\p(\phi)=1$ on $\partial B_r$, 
   \begin{equation}\label{LMeq}
    E_{\eps_h}\big(\P(u),B_r\big)= E_{\eps_h}(w_h,B_r) +\frac{1}{2} \int_{B_r}  |w_h|^2 |\nabla \phi|^2+ \frac{|w_h|^4}{2\eps_h^2}(1-|\phi|^2)^2 +\frac{2}{m}j(\p(\phi))\cdot j(w_h)\,dx\,,
       \end{equation}
   and $\mathcal{H}^1(J_{u}\cap B_r)=\mathcal{H}^1(J_\phi\cap B_r)$.
  \end{lemma}
\begin{proof}
By Proposition  \ref{buildGLcompet}, we have $|w_h|^2\geq 1/2$ for $\eps_h$ small enough. Setting $\phi:=u/w_h$, we have $\phi \in SBV^2(B_r)\cap L^\infty(B_r)$ and
$\P(\phi)\in W^{1,2}(B_r)$, thus $\phi\in \mathcal{G}(B_r)\cap L^\infty(B_r)$. Since $|w_h|^2$ and $|\phi|^2$ belong to $W^{1,2}(B_r)$, by chain rule we have 
\[|\nabla u|^2= |\phi|^2|\nabla w_h|^2+|w_h|^2|\nabla \phi|^2 +\frac{1}{2}\nabla (|\phi|^2)\cdot \nabla( |w_h|^2) +2 j(\phi)\cdot j(w_h)\quad\text{a.e. in $B_r$}\,.\] 
Recalling that $m j(\phi)= j(\p(\phi))$ (see Lemma~\ref{struct}), and since $|\phi|^2=|\p(\phi)|^2$, we obtain
\begin{equation}\label{startingpointLM}
 \int_{B_r} |\nabla u|^2\,dx= \int_{B_r}  |\p(\phi)|^2|\nabla w_h|^2+|w_h|^2|\nabla \phi|^2
  +\frac{1}{2}\nabla (|\p(\phi)|^2)\cdot \nabla (|w_h|^2)+\frac{2}{m} j(\p(\phi))\cdot j(w_h)\,dx\,.
\end{equation}
Testing equation \eqref{GLeq} with $|\p(\phi)|^2 w_h\in W^{1,2}(B_r)$,  we derive
\begin{multline}\label{LMinter1}
 \int_{B_r} |\p(\phi)|^2|\nabla w_h|^2+\frac{1}{2}\nabla (|\p(\phi)|^2)\cdot \nabla (|w_h|^2) \,dx\\
 =\int_{\partial B_r}\partial_\nu w_h\cdot w_h\,d\mathcal{H}^1+\int_{B_r} \frac{1}{\eps_h^2} (1-|w_h|^2)|w_h|^2 |\p(\phi)|^2\,dx\,,
\end{multline}
where the first integral in the right hand side is understood in the $W^{-1/2,2}-W^{1/2,2}$ sense. Here we have also used the fact $|\p(\phi)|^2=1$ on $\partial B_r$ (so that $|\p(\phi)|^2 w_h=w_h$ on $\partial B_r$). 
Testing now  \eqref{GLeq} with $w_h$ yields 
\[\int_{\partial B_r} \partial_\nu w_h\cdot w_h\,d\mathcal{H}^1=\int_{B_r} |\nabla w_h|^2-\frac{1}{\eps_h^2}(1-|w_h|^2)|w_h|^2\,dx\,.
\]
Putting together  this identity with \eqref{startingpointLM}  and \eqref{LMinter1} leads to 
\begin{multline*}
 \int_{B_r} |\nabla u|^2+\frac{1}{2\eps_h^2}(1-|u|^2)^2\,dx=\int_{B_r} |\nabla w_h|^2+|w_h|^2|\nabla \phi|^2+\frac{2}{m} j(\p(\phi))\cdot j(w_h)\\
 +\frac{1}{\eps_h^2}\lt[ \frac{1}{2} (1-|w_h|^2 |\p(\phi)|^2)^2-(1-|w_h|^2)|w_h|^2+ (1-|w_h|^2)|w_h|^2 |\p(\phi)|^2\rt]\,dx\,.
\end{multline*}
In view of the algebraic identity 
\[\frac{1}{2}(1-a^2b^2)^2-(1-a^2)a^2+(1-a^2)a^2b^2=\frac{1}{2} (1-a^2)^2+\frac{a^4}{2}(1-b^2)^2\quad \text{ for } a,b\geq 0\,,\]
we have obtained 
\begin{multline*}
 \frac{1}{2}\int_{B_r} |\nabla u|^2+\frac{1}{2\eps_h^2}(1-|u|^2)^2\,dx=E_{\eps_h}(w_h,B_r)\\
 +\frac{1}{2} \int_{B_r} \left\{ |w_h|^2 |\nabla \phi|^2+ \frac{|w_h|^4}{2\eps_h^2}(1-|\phi|^2)^2 +\frac{2}{m}j(\p(\phi))\cdot j(w_h) \right\}\,dx\,.
 \end{multline*}
Finally, since $w_h\in W^{1,2}(B_r)$ we have $J_u\cap B_r=J_\phi\cap B_r$ (up to an $\mathcal{H}^1$-null set), and the conclusion follows.
\end{proof}

We now use Lemma~\ref{LM} to derive a a lower bound on the energy. In particular, we want to be able to control the last term in \eqref{LMeq}, which is the purpose of the following lemma.  

\begin{lemma}\label{propwente}
Assume that 
 \begin{equation}\label{hyp1}
\int_{B_r} |\nabla w_h|^2\,dx\leq \delta\,,
\end{equation}
and 
\begin{equation}\label{hyp2}
 r^{1/2}\lt(\int_{\partial B_r} |\partial_\tau g_h|^2+\frac{(1-|g_h|^2)^2}{2\eps_h^2}\,d\mathcal{H}^1\rt)^{1/2}\leq \delta\,,
\end{equation}
for some constant $\delta>0$. Then there exists a universal constant $C_\star>0$ such  that 
 \begin{equation*}
 \lt|\int_{B_r} j(\Phi)\cdot j(w_h)\,dx\rt|\leq C_\star \delta\int_{B_r} |\nabla \Phi|^2\,dx
 \end{equation*}
 for every $\Phi\in W^{1,2}(B_r)$ satisfying $\Phi=1$ on $\partial B_r$. 
\end{lemma}

\begin{proof}
Rescaling variables we may assume that $r=1$. Arguing by approximation as in Proposition~\ref{buildGLcompet}, we may assume  that $w_h$ is smooth. 

First, using equation \eqref{GLeq} we derive that ${\rm div}\,j(w_h)=w_h\wedge\Delta w_h=0$. 
By Hodge decomposition, we can find a smooth scalar function $H$ such that $j(w_h)=\nabla^\perp H$.  
Notice that $H$ is defined up to an additive constant  that we shall fix later on. 

By approximation, we may assume that the test function $\Phi$ is smooth. Since $\Phi$ is constant on~$\partial B_1$, the vector field $j(\Phi)$ satisfies  
$j^\perp(\Phi)\cdot \nu=0$ on~$\partial B_1$. Since $\curl j(\Phi)=2 \det \nabla \Phi$, 
$$  \int_{B_1} j(\Phi)\cdot j(w_h)\,dx =-\int_{B_1} j^\perp(\Phi)\cdot \nabla H\,dx=-\int_{B_1} 2H \det(\nabla \Phi)\,dx\,. $$
We may now estimate 
 \[
  \lt|\int_{B_1} j(\Phi)\cdot j(w_h)\,dx\rt|\leq \int_{B_1} |H| |\nabla \Phi|^2\,dx\leq \|H\|_{L^\infty(B_1)} \int_{B_1} |\nabla \Phi|^2\,dx\,, 
 \]
and we are left to prove that $\|H\|_{L^\infty(B_1)}$ is controlled by $\delta$. We consider the function $H_1$ solving 
\[\begin{cases}
   \Delta H_1  =2\det(\nabla w_h) & \textrm{in } B_1\,,\\
   H_1=0 & \textrm{on } \partial B_1\,,
  \end{cases}\]
and  set $H_2:=H-H_1$. Then, $H_2$ is harmonic in $B_1$ since
$$\Delta H_2=-{\rm div}\, j^\perp(w_h)-2\det(\nabla w_h)={\rm curl}\,j(w_h)-2\det(\nabla w_h)=0\,. $$ 
In addition, 
\begin{equation}\label{eqtocite}\partial_\tau H_2=\partial_\tau H = -\nu\cdot j(w_h)\quad \text{on $\partial B_1$}\,.\end{equation} 
Thanks to Wente's estimate (see~\cite{Wente} or \cite[Lemma A.1]{BrezCor}), there exists a universal constant $C_\sharp>0$ such that
\[\|H_1\|_{L^\infty(B_1)}\leq C_\sharp \int_{B_1} |\nabla w_h|^2\,dx\leq C_\sharp \delta\,.\]
Moreover, by the maximum principle,
\[
\inf_{B_1}H_2 = \inf_{\partial B_1} H_2\quad\text{and}\quad \sup_{B_1} H_2= \sup_{\partial B_1} H_2\,.\]
We now fix the additive constant for $H$ so that $\sup_{B_1} H_2+\inf_{B_1} H_2=0$. This yields 
\[\|H_2\|_{L^{\infty}(B_1)}=\dfrac12 \lt[\sup_{\partial B_1} H_2-\inf_{\partial B_1} H_2\rt]\leq \dfrac14\int_{\partial B_1} |\partial_\tau H_2|\,d\mathcal{H}^1\,.\]
Recalling that $|w_h|\leq 1$, we have  by \eqref{eqtocite} $|\partial_\tau H_2|\leq |w_h||\partial_\nu w_h|\leq |\partial_\nu w_h|$, so that by Cauchy-Schwarz inequality,
\[
\|H_2\|_{L^{\infty}(B_1)}\leq \dfrac{\sqrt{2\pi}}4 \lt(\int_{\partial B_1}|\partial_\nu w_h|^2\,d\mathcal{H}^1\rt)^{1/2}.
\]
Let us now recall that the Pohozaev identity  applied to equation \eqref{GLeq} (see e.g. \cite[(5.2)]{SS}) leads to  
\[
\int_{B_1}\frac{(1-|w_h|^2)^2}{\eps_h^2}\,dx=\int_{\partial B_1} |\partial_\tau w_h|^2-|\partial_\nu w_h|^2+\frac{(1-|w_h|^2)^2}{2\eps_h^2}\,d\mathcal{H}^1\,.
\]
Hence,
\[
 \int_{\partial B_1} |\partial_\nu w_h|^2\,d\mathcal{H}^1\leq \int_{\partial B_1} |\partial_\tau w_h|^2+\frac{(1-|w_h|^2)^2}{2\eps_h^2}\,\,d\mathcal{H}^1\,,
\]
which then implies 
\[\|H_2\|_{L^{\infty}(B_1)}\leq  \dfrac{\sqrt{2\pi}}4 \lt(\int_{\partial B_1} |\partial_\tau w_h |^2+\frac{(1-|w_h|^2)^2}{2\eps_h^2}\,d\mathcal{H}^1 \rt)^{1/2}\stackrel{\eqref{hyp2}}{\leq}\dfrac{\sqrt{2\pi}}4 \delta\,.\]
The conclusion now follows with $C_\star:=C_\sharp+\sqrt{2\pi}/4$. 
\end{proof}

Combining Lemma \ref{LM} and Lemma \ref{propwente} yields the following lower bound for the energy.   

\begin{proposition}\label{propLM}
 Let $C_\star$ be the constant given by Proposition~\ref{propwente}, and let $\delta>0$ be such that $C_\star\delta\leq \frac{1}{16m}$. For $\eps_h$ small enough,  
 if $w_h$ satisfies~\eqref{hyp1} and~\eqref{hyp2}, then 
 \begin{equation}\label{MainestimLM}
 F_\eps^0(u,B_r)\geq  E_{\eps_h}(w_h,B_r) +\frac{1}{8} \int_{B_r}  |\nabla \phi|^2\,dx + \mathcal{H}^1(J_\phi\cap B_r)
 \end{equation}
 for every $u\in \mathcal{G}(B_r)\cap L^\infty(B_r)$ satisfying $\p(u)=\p(g_h)$ on $\partial B_r$, where $\phi:=u/w_h \in \mathcal{G}(B_r)\cap L^\infty(B_r)$. 
\end{proposition}

\begin{proof}
By Proposition \ref{buildGLcompet} and Lemma \ref{LM}, identity \eqref{LMeq} holds and $|w_h|^2\geq 1/2$ for $\eps_h$ small enough.  Applying Lemma \ref{propwente} with $\Phi:=\p(\phi)$, we derive that 
$$E_{\eps_h}\big(\P(u),B_r\big)\geq E_{\eps_h}(w_h,B_r) +\frac{1}{4}\int_{B_r} |\nabla \phi|^2+ \frac{1}{4\eps_h^2}(1-|\phi|^2)^2\,dx-\frac{1}{8m^2}\int_{B_r}|\nabla \p(\phi)|^2\,dx\,.$$
By Lemma \ref{struct} we have 
$$ \frac{1}{m^2}\int_{B_r}|\nabla \p(\phi)|^2\,dx \leq \int_{B_r}|\nabla \P(\phi)|^2\,dx= \int_{B_r}|\nabla \phi|^2\,dx\,,$$
and the conclusion follows. 
\end{proof}

\begin{remark}\label{remuniq}
Under the assumptions of Proposition~~\ref{propLM}, we can obtain that the minimizer of \eqref{minpbbdgh} is unique. 
Indeed, any solution $\widetilde w_h$ of \eqref{minpbbdgh} satisfies $\|\widetilde w_h\|_{L^\infty(B_r)}\leq 1$ by minimality. Applying \eqref{MainestimLM} to $\widetilde w_h$ then yields 
$$E_{\eps_h}(\widetilde w_h,B_r)\geq E_{\eps_h}(w_h,B_r) +\frac{1}{8} \int_{B_r}  |\nabla \phi|^2\,dx  $$
with $\phi:=\widetilde w_h/w_h$. Since $E_{\eps_h}(\widetilde w_h,B_r)= E_{\eps_h}(w_h,B_r)$, we deduce that $\phi\equiv 1$, that is $\widetilde w_h=w_h$. 
 A similar idea was used in~\cite{FarMir} to prove uniqueness results, extending those from~\cite{Yezu}.
\end{remark}

 \subsection{Proof of Theorem~~\ref{theoregeps}}
 \label{S53}
 
 This section is devoted to the proof of Theorem~\ref{theoregeps}. We fix a sequence $\eps_h\to 0$, and minimizers $u_h$ of $F^0_{\eps_h,g}$ over $\mathcal{G}_g(\Omega)$. 
 We assume that $u_h\to u$ strongly in $L^1(\Omega)$ as $h\to\infty$, where $u$ is a minimizer of $F_{0,g}$ over $\mathcal{L}_g(\Omega)$. We recall that by Theorem \ref{Gammasharp} $v_h:=\p(u_h)\to u^m$ strongly in $W^{1,p}(\Omega)$ for every $p<2$ 
 and in $W^{1,2}_{\rm loc}(\overline\Omega\setminus{\rm spt}\,\mu)$, where $\mu:={\rm curl}\,j(u^m)\in \mathcal{A}_d$. 
 According to Section~\ref{S4},  the compact set $\Gamma:=J_u\subset \Omega$ is a $\Lambda(\mu)$-minimizer in the sense of Definition \ref{defdmumin}, and thus a union of at most $d$ Steiner trees. We denote by $T\subset \Omega$ the (finite) set of Steiner points of $\Gamma\setminus {\rm spt}\,\mu$, i.e., the triple junctions of $\Gamma$ away from ${\rm spt}\,\mu$.
 We finally recall that $u\in C^\infty(\overline\Omega\setminus \Gamma)$, $u=g$ on $\partial \Omega$, and that $u^m=v_\mu\in C^\infty(\overline\Omega\setminus{\rm spt}\,\mu)$. 
  
 Writing ${\rm spt}\,\mu=:\{x_1,\ldots,x_{md}\}$ and $T:=\{y_1,\ldots,y_q\}$, we now fix  $\sigma_0>0$ satisfying 
 $$\sigma_0<\frac{1}{2}\min \bigg\{\min_{k\not=l}|x_k-x_l|\,, \min_{k\not=l}|y_k-y_l|\,, {\rm dist}(\Gamma,\partial\Omega)\,, {\rm dist}(T,{\rm spt}\,\mu) \bigg\}\,,$$
 and we set for $\sigma\in(0,\sigma_0)$, 
\begin{equation}\label{Ksigma}K_\sigma:=\|\nabla v_\mu\|_{L^\infty(\Omega\setminus B_{\sigma/4}(\mu))}\,. \end{equation}
 Moreover we fix the positive constant $\delta$ to be  
 $$\delta:=\min\big\{1/(4\sqrt{\pi}m),1/(16m C_\star)\big\}\,,$$
 $C_\star$ being the constant given by Proposition \ref{propwente}. 
For $\sigma\in(0,\sigma_0)$, we finally set 
$$r_\sigma:=\min\Big\{\frac{\sigma}{10}, \frac{\delta}{8\sqrt{\pi}K_\sigma}\Big\} \,.$$
Theorem \ref{theoregeps} is a consequence of a covering argument combined with Proposition \ref{smoothcvawaygamma}, Proposition~\ref{smoothconvontheedge} and Proposition~\ref{propmincloseTnew} which respectively give the structure of $u_h$ away from $\Gamma$, close to $\Gamma$ but away from the triple junctions and at the triple junctions.

 \subsubsection{Smoothness and convergence away from $\Gamma$}
 
 \begin{proposition}\label{smoothcvawaygamma}
 Let $\sigma\in(0,\sigma_0)$. For $h$ large enough, $u_h\in W^{1,2}_g(\Omega\setminus B_\sigma(\Gamma))$  and $u_h$ minimizes $E_{\eps_h}(\cdot,\Omega\setminus B_\sigma(\Gamma))$ under its own boundary condition. 
 In addition, $u_h\in C^{\infty}(\overline\Omega\setminus B_\sigma(\Gamma))$  and $u_h\to u$ in $C^{1,\alpha}(\overline\Omega \setminus B_\sigma(\Gamma))$ and $C^{k}_{\rm loc}(\Omega \setminus B_\sigma(\Gamma))$ for every $\alpha\in(0,1)$ and $k\in\mathbb{N}$. 
 \end{proposition}
 
 The proof of Proposition \ref{smoothcvawaygamma} is a direct consequence of Lemma \ref{LemfarS} and Lemma \ref{smoothcvbdry} below, together with a suitable covering  argument.

\begin{lemma}\label{LemfarS}
For $\sigma\in(0,\sigma_0)$, let $r\in(0,r_\sigma)$ and $x_0\in\Omega$ be such that  $\overline B_{2r}(x_0)\subset\Omega\setminus B_\sigma(\Gamma)$. For $h$ large enough,  $u_h\in W^{1,2}(B_{r}(x_0))$, and $u_h$ minimizes $E_{\eps_h}(\cdot,B_{r}(x_0))$ under its own boundary condition. In addition, $u_h\in C^\infty(B_{r}(x_0))$ and $u_h\to u$ in $C^{k}_{\rm loc}(B_{r}(x_0))$ for every $k\in\mathbb{N}$. 
\end{lemma} 
 
\begin{proof}
{\it Step 1.} Without loss of generality, we may assume that $x_0=0$. Set 
 \[
  \gamma_h:= \int_{B_{2r}}|\nabla v_h- \nabla v_\mu|^2+|v_h-v_\mu|^2\,dx+ \frac{1}{2\eh^2} \int_{B_{2r}}(1-|u_h|^2)^2\,dx\,. 
 \]
By Theorem \ref{Gammasharp} and Corollary \ref{corcvminsharp}, $\gamma_h\to 0$ as $h\to \infty$. Since $\Gamma\cap\overline B_{2r}=\emptyset$, \eqref{strongconvergencebisshp} shows that 
\begin{equation}\label{11483107}
\mathcal{H}^1(J_{u_h}\cap B_{2r})\leq \frac{r}{2}
\end{equation}
for $h$ large enough.  From now on, we assume that \eqref{11483107} holds. 

By the coarea formula  (see  \cite[Theorem II.7.7]{Maggi}), we have 
\[
 \int_{r}^{2r} \mathcal{H}^0(J_{u_h}\cap \partial B_t)\, dt \leq  \mathcal{H}^1(J_{u_h}\cap B_{2r})\leq \frac{r}{2}\,.\]
Setting  
$$A_h:=\Big\{t\in (r,2r)   :  J_{u_h}\cap \partial B_{t}=\emptyset\Big\}\,,$$
we deduce that $\mathcal{H}^1(A_h)\geq r/2$. Notice that $u_h\res \partial B_t \in W^{1,2}(\partial B_t)$ for a.e. $t\in A_h$. Since 
$$\int_{A_h} \lt[\int_{\partial B_t}|\nabla v_h-\nabla v_\mu|^2+|v_h-v_\mu|^2+\frac{1}{2\eh^2} (1-|u_h|^2)^2\,d\mathcal{H}^1\rt]\,dt \leq  \gamma_h\,,$$
we can find a radius $\rho_h\in A_h$ such that $u_h\res \partial B_{\rho_h} \in W^{1,2}(\partial B_{\rho_h};\C)$ and 
\begin{equation}\label{cvvhbordboul}
\int_{\partial B_{\rho_h}}|\nabla v_h-\nabla v_\mu|^2+|v_h-v_\mu|^2+\frac{1}{2\eh^2} (1-|u_h|^2)^2\,d\mathcal{H}^1\leq \frac{2\gamma_h}{r}\,.
\end{equation}
By definition of $\p$, $|\nabla u_h|\le|\nabla v_h|$ and thus
\begin{multline*}
\int_{\partial B_{\rho_h}}|\nabla u_h|^2\,d\mathcal{H}^1\leq\int_{\partial B_{\rho_h}}|\nabla v_h|^2\,d\mathcal{H}^1\leq 2 \int_{\partial B_{\rho_h}}|\nabla v_h-\nabla v_\mu|^2\,d\mathcal{H}^1+4\pi \rho_h \|\nabla v_\mu\|^2_{L^\infty(B_{2r})} \\
\leq \frac{4\gamma_h}{r}+8\pi r K^2_{\sigma}\,,
\end{multline*}
which leads to 
\begin{equation}\label{condbdrhohcpa}
\rho_h^{1/2}\left(\int_{\partial B_{\rho_h}}|\nabla u_h|^2+\frac{1}{2\eh^2} (1-|u_h|^2)^2\,d\mathcal{H}^1\right)^{1/2}\leq \delta
\end{equation}
for $h$ large enough since $\rho_h\leq 2r_\sigma$. 
\vskip5pt

\noindent{\it Step 2.} We select a  subsequence such that $\rho_h\to \rho\in[r,2r]$. Define $g_h(x):=u_h(\rho_h x)$ for $x\in\partial B_1$. Then $g_h\in W^{1,2}(\partial B_1)\cap L^\infty(B_1)$ satisfies $|g_h|\leq 1$, and 
\begin{equation}\label{condghinprf}
\left(\int_{\partial B_1}|\partial_\tau g_h|^2+\frac{1}{2\widetilde{\eps}_h^2} (1-|g_h|^2)^2\,d\mathcal{H}^1\right)^{1/2}\leq \delta\,,
\end{equation}
where $\widetilde{\eps}_h:=\eps_h/\rho_h$. Extracting a further subsequence if necessary, we may thus assume that $g_h\to g_\star$ uniformly of $\partial B_1$ for some $g_\star\in W^{1,2}(\partial B_1;\Ss^1)$. Estimate \eqref{cvvhbordboul} yields 
\begin{equation}\label{gstarmvmu}
g^m_\star(x)=\lim_{h\to \infty}\p(g_h)(x)= \lim_{h\to \infty}v_h(\rho_h x) = v_\mu(\rho x)\quad \forall x\in \partial B_1\,.
\end{equation}
Since ${\rm deg}(v_\mu,\partial B_\rho)=0$, we deduce that ${\rm deg}(g_\star,\partial B_1)=0$. 
We are now in position to apply Proposition~\ref{buildGLcompet} to produce minimizers $w_h$ of $E_{\widetilde\eps_h}(\cdot,B_1)$ over $W^{1,2}_{g_h}(B_1)$. Then $w_h\to w_\star$ strongly in $W^{1,2}(B_1)$ where $w_\star$ is the unique solution of \eqref{pbharmmaplimtw}. We claim that 
$$w^m_\star(x)=v_\mu(\rho x)\quad \forall x\in B_1\,. $$
Indeed, recalling \eqref{definivmu}, $v_\mu(\rho x)=e^{i\psi_\mu(x)}$ for $x\in \overline B_1$ and a smooth harmonic function $\psi_\mu$ (which is unique up to a constant multiple of $2\pi$).
Moreover, $w_\star=e^{i\varphi_\star}$ for some harmonic function $\varphi_\star\in W^{1,2}(B_1)$. In view of \eqref{gstarmvmu}, we have $m\varphi_\star=\psi_\mu+2k\pi$ on $\partial B_1$ for some constant $k\in\mathbb{N}$. By uniqueness of the harmonic extension, we infer that  $m\varphi_\star=\psi_\mu+2k\pi$ in $B_1$, and the claim follows. 

As a consequence of this last identity, we deduce that 
$$\int_{B_1}|\nabla w_\star|^2\,dx=\frac{1}{m^2}\int_{B_{\rho_h}}|\nabla v_\mu|^2\,dx \leq \frac{4\pi r^2K_\sigma^2}{m^2}\leq \frac{\delta}{2}\,. $$
Since $w_h\to w_\star$ strongly in $W^{1,2}(B_1)$, we thus have for $h$ large enough 
\begin{equation}\label{condwhinprf}
\int_{B_1}|\nabla w_h|^2\,dx\leq \delta\,.
\end{equation}
\vskip5pt

\noindent{\it Step 3.} Let us define $\widehat w_h(x):=w_h( x/\rho_h)$, and consider the competitor $\widehat u_h\in \mathcal{G}_g(\Omega)$ given by 
$$\widehat u_h:=
\begin{cases} 
u_h & \text{in $\Omega\setminus B_{\rho_h}$}\,,\\
\widehat w_h& \text{in $B_{\rho_h}$}\,. 
\end{cases}$$
By minimality we have $F^0_{\eps_h,g}(u_h)\leq F^0_{\eps_h,g}(\widehat u_h)$, and since $J_{u_h}\cap \partial B_{\rho_h}=\emptyset$, we deduce that 
$$F^0_{\eps_h}(u_h,B_{\rho_h})\leq E_{\eps_h}(\widehat w_h,B_{\rho_h}) \,.$$
Setting $\widetilde u_h(x):=u_h(\rho_h x)$ and rescaling variables, we obtain 
\begin{equation}\label{fucklabel1}
E_{\widetilde \eps_h}\big(\P(\widetilde u_h),B_1\big)+\rho_h\mathcal{H}^1(J_{\widetilde u_h}\cap B_1)\leq E_{\widetilde \eps_h}( w_h,B_1) \,. 
\end{equation}
 In view of \eqref{condghinprf} and \eqref{condwhinprf} (and our choice of $\delta$), we can apply Lemma \ref{LM} and Proposition \ref{propLM} to derive that 
\begin{equation}\label{fucklabel2}
E_{\widetilde \eps_h}\big(\P(\widetilde u_h),B_1\big)+\rho_h\mathcal{H}^1(J_{\widetilde u_h}\cap B_1)\geq E_{\widetilde \eps_h}( w_h,B_1)+\frac{1}{8} \int_{B_1}  |\nabla \phi_h|^2\,dx+\rho_h\mathcal{H}^1(J_{\phi_h}\cap B_1)
\end{equation}
for $h$ large enough, where $\phi_h:=\widetilde u_h/w_h$ satisfies $\phi_h=1$ on $\partial B_1$. Putting   \eqref{fucklabel1} and \eqref{fucklabel2} together leads to 
$ \int_{B_1}  |\nabla \phi_h|^2\,dx=0=\mathcal{H}^1(J_{\phi_h}\cap B_1)$, and thus $\phi_h\equiv1$. In other words, $\widetilde u_h\equiv w_h$ for $h$ large enough. 

Scaling back to the original variables (and recalling that $u_h\to u$ in $L^1(\Omega)$), we conclude from Proposition \ref{buildGLcompet} that for $h$ large enough, $u_h$ minimizes $E_{\eps_h}(\cdot,B_{r})$ in $W^{1,2}(B_{r})$ under its own boundary condition, $u_h\in C^\infty(B_{r})$ and $u_h\to u$ in $C^k_{\rm loc}(B_{r})$ for every $k\in\mathbb{N}$.
Since the limit is unique, we deduce that these facts actually hold for the full sequence (and not only for a subsequence). 
\end{proof}

The next lemma is devoted to smoothness and convergence of $u_h$ near the boundary of $\Omega$. Since $\partial\Omega$ is assumed to be smooth, we can find a radius $r_\Omega>0$ such that 
\begin{equation}\label{smoothOmega}\mathcal{H}^1(\partial\Omega\cap B_r(x))\leq 3 r\quad\forall r\in(0,r_\Omega)\,,\;\forall x\in \partial\Omega\,.\end{equation} 
For the sake of variety, in the proof below, we do not use the energy splitting argument. Notice that, either way, it could be possible to adapt this alternative argument to prove Lemma~\ref{LemfarS}, or to adapt  the energy splitting approach to treat boundary points.

\begin{lemma}\label{smoothcvbdry}
 For $\sigma\in(0,\sigma_0)$, let $r\in(0,\min\{r_\sigma,r_\Omega\})$ and $x_0\in\partial\Omega$. For $h$ large enough, $u_h\in W^{1,2}(B_{r}(x_0)\cap\Omega)$ with $u_h=g$ on $\partial \Omega\cap B_r(x_0)$  and $u_h$ minimizes $E_{\eps_h}(\cdot,B_{r}(x_0)\cap \Omega)$ under its own boundary conditions.
 In addition, $u_h\in C^\infty(B_{r}(x_0)\cap\overline\Omega)$ and $u_h\to u$ in $C^{1,\alpha}_{\rm loc}(B_{r}(x_0)\cap \overline\Omega)$ and $C^{k}_{\rm loc}(B_{r}(x_0)\cap \Omega)$ for every $\alpha\in(0,1)$ and $k\in\mathbb{N}$.
 \end{lemma}

\begin{proof}
 Without loss of generality, we may assume that $x_0=0$.  As in the proof of Proposition~\ref{LemfarS}, it is enough to find  $\rho_h\in(r,2r)$ such that  \eqref{condbdrhohcpa} holds (with $\partial B_{\rho_h}\cap \Omega$ in place of $\partial B_{\rho_h}$), 
 $J_{u_h}\cap (\overline B_{\rho_h}\cap \Omega)=\emptyset$ (so that $u_h\in W^{1,2}(B_{\rho_h}\cap\Omega)$), and $u_h=g$ on $\partial\Omega\cap B_{\rho_h}$, for $h$ large enough. Indeed, for any $w\in W^{1,2}_{u_h}(B_{\rho_h})$, one can consider the competitor  $\widehat u_h\in \mathcal{G}_g(\Omega)$ given by $\widehat u_h=w$ in $B_{\rho_h}\cap\Omega$, and $\widehat u_h=u_h$ in $\Omega\setminus B_{\rho_h}$. By minimality, $F^0_{\eps_h,g}(u_h)\leq F^0_{\eps_h,g}(\widehat u_h)$, which then leads to 
 $$E_\eh(u_h,B_{\rho_h}\cap \Omega)=F^0_{\eps_h}(u_h,B_{\rho_h}\cap\Omega)\leq  F^0_{\eps_h}(w,B_{\rho_h}\cap\Omega)=E_\eh(w;B_{\rho_h}\cap \Omega)\,.$$
Hence $u_h$ minimizes $E_{\eps_h}(\cdot,B_{\rho_h})$ in $W^{1,2}(B_{\rho_h}\cap\Omega)$ under its own boundary condition. Then the remaining conclusions follow from \cite{BBHarticle} (see also \cite[Theorem A.3]{BBH}) together with the fact that $u_h\to u$ in $L^1(\Omega)$. 

We select the radius $\rho_h$ by repeating the Fubini type argument used in Step 1 of the proof  of Proposition \ref{LemfarS}. The main additional point is to
select $\rho_h$ so that $u_h$ belongs to $W^{1,2}(\partial B_{\rho_h}\cap\Omega)$ with $u_h=g$ on $\partial B_{\rho_h}\cap\partial \Omega$. This is possible via the coarea formula since \eqref{strongconvergencebisshp} implies that for $h$ large enough
$$\mathcal{H}^1(J_{u_h}\cap B_{2r}\cap\Omega)+\mathcal{H}^1(\{u_h\not=g\}\cap\partial\Omega\cap B_{2r}) \leq \frac{r}{2}.$$

By our choice of $\rho_h$, the map $g_h$ defined by $g_h:=u_h$ in $\partial B_{\rho_h}\cap\Omega$, and $g_h:=g$ in $ B_{\rho_h}\cap\partial\Omega$, belongs to $W^{1,2}(\partial (B_{\rho_h}\cap\Omega))$. In view of  \eqref{condbdrhohcpa}, for $h$ large enough we have $|g_h|\geq 1/2$ on $\partial(B_{\rho_h}\cap\Omega)$, and 
\begin{align*}
\mathop{{\rm osc}}\limits_{\partial(B_{\rho_h}\cap\Omega)} g_h & \leq \big(\mathcal{H}^1(\partial(B_{\rho_h}\cap\Omega))\big)^{1/2} \left(\int_{\partial(B_{\rho_h}\cap\Omega)}|\partial_\tau g_h|^2\,d\mathcal{H}^1\right)^{1/2}\\
& \leq \sqrt{2\pi\rho_h} \left(\int_{B_{\rho_h}\cap\partial\Omega}|\partial_\tau g|^2\,d\mathcal{H}^1+\int_{\partial B_{\rho_h}\cap\Omega}|\nabla u_h|^2\,d\mathcal{H}^1\right)^{1/2}\\
& \stackrel{\eqref{Ksigma}, \eqref{smoothOmega} \& \eqref{condbdrhohcpa}}{\leq} \sqrt{2\pi\rho_h} \left(3\rho_hK^2_\sigma+ \delta^2/\rho_h \right)^{1/2}\\
& \leq \sqrt{2\pi} \left( 6 r^2K^2_\sigma+ \delta^2\right)^{1/2}\\
& \leq \frac{1}{2m}\,.
\end{align*}
Rotating coordinates in the image if necessary, we may assume that $g_h(0)=1$, which in turn yields  
\begin{equation}\label{bougpatro}
|\p(g_h(x))-1|\leq \frac{1}{2}\qquad\forall x\in\partial(B_{\rho_h}\cap\Omega)\,.
\end{equation}

We claim that
\begin{equation}\label{eqproblems}
 \min_{v\in W^{1,2}_{\p(g_h)}(B_{\rho_h}\cap \Omega)} G_\eh(v,B_{\rho_h}\cap \Omega)=\min_{u\in W^{1,2}_{g_h}(B_{\rho_h}\cap \Omega)} E_\eh(u,B_{\rho_h}\cap \Omega)\,.
\end{equation}
Before proving this claim, let us show how \eqref{eqproblems} leads to the conclusion. By minimality of $u_h$ (and our choice of $\rho_h$), we have 
\begin{align*}
 \min_{u\in W^{1,2}_{g_h}(B_{\rho_h}\cap \Omega)} E_\eh(u,B_{\rho_h}\cap \Omega)&=\min_{u\in W^{1,2}_{g_h}(B_{\rho_h}\cap \Omega)} F^0_\eh(u,B_{\rho_h}\cap \Omega)\\
 &\geq \min_{u\in \mathcal{G}_{g_h}(B_{\rho_h}\cap \Omega)} \Big\{F^0_\eh(u,B_{\rho_h}\cap \Omega)+\mathcal{H}^1\big(\{u\not= g_h\}\cap\partial(B_{\rho_h}\cap\Omega)\big)\Big\}\\
 &= \begin{multlined}[t][9.5cm]
 G_\eh(v_h,B_{\rho_h}\cap \Omega)+\mathcal{H}^1(J_{u_h}\cap (B_{\rho_h}\cap \Omega))\\
 + \mathcal{H}^1(\{u_h\not= g\}\cap(\partial\Omega\cap B_{\rho_h}))
 \end{multlined}
 \\
 &\geq \min_{v\in W^{1,2}_{\p(g_h)}(B_{\rho_h}\cap \Omega)} G_\eh(v,B_{\rho_h}\cap \Omega)\,.
\end{align*}
Then, \eqref{eqproblems} implies that all the inequalities above are in fact equalities and as a consequence 
$$\mathcal{H}^1(J_{u_h}\cap (B_{\rho_h}\cap \Omega))+ \mathcal{H}^1(\{u_h\not= g\}\cap(\partial\Omega\cap B_{\rho_h}))=0\,.$$ 
Hence $J_{u_h}\cap (B_{\rho_h}\cap \Omega)$ is  empty, and $u_h=g$ on $\partial\Omega\cap B_{\rho_h}$. 

In view of the above chain of inequalities,  to prove~\eqref{eqproblems} it is enough to show that 
\begin{equation}\label{eqproblemsbis}
 \min_{v\in W^{1,2}_{\p(g_h)}(B_{\rho_h}\cap \Omega)} G_\eh(v,B_{\rho_h}\cap \Omega)\geq\min_{u\in W^{1,2}_{g_h}(B_{\rho_h}\cap \Omega)} E_\eh(u,B_{\rho_h}\cap \Omega)\,.
\end{equation}
We consider $\overline v$ a minimizer of the left-hand side. To establish \eqref{eqproblemsbis}, it is enough to construct $\overline u\in W^{1,2}_{g_h}(B_{\rho_h}\cap \Omega)$ 
satisfying $\p(\overline u)=\overline v$ since, in this case, $G_\eh(\overline v,B_{\rho_h}\cap \Omega)= E_\eh(\overline u,B_{\rho_h}\cap \Omega)$.  
Let $\Pi:\C\to \C$ the map defined by $\Pi(z):=|{\rm Re}(z)|+i {\rm Im}(z)$. By \eqref{bougpatro} we have $\Pi(\p(g_h))=\p(g_h)$ and  $\Pi(\overline v)\in W^{1,2}_{\p(g_h)}(B_{\rho_h}\cap\Omega)$. In addition,  $G_\eh(\Pi(\overline v),B_{\rho_h}\cap\Omega)=G_\eh(\overline v,,B_{\rho_h}\cap\Omega)$. 
Replacing $\overline v$ by $\Pi(\overline v)$ if necessary, we may thus assume that the real part of $\overline v$ is nonnegative in $B_{\rho_h}\cap\Omega$. 
Now, let us introduce the map $\q: \C\cap\{{\rm Re}(z)\geq 0\}\to \C$ defined by $\q(z)=|z|e^{i\theta/m}$ for $z=|z|e^{i\theta}$ with $\theta\in[-\pi/2,\pi/2]$. 
Then, $\q$ is Lipschitz continuous left inverse of $\p$. In view of \eqref{bougpatro} we have $\q(\p(g_h))=g_h$, and as a consequence $\overline u:=\q(\overline v)\in W_{g_h}^{1,2}(B_{\rho_h}\cap\Omega)$ with $\p(\overline u)=\overline v$. 
\end{proof}

\subsubsection{Smoothness and convergence away from triple junctions}

We continue our asymptotic analysis by considering the local behavior of $u_h$ near $\Gamma$, but away from $T\cup{\rm spt}\,\mu$.  In the statement below, we understand the convergence of half spaces in the sense of local Hausdorff convergence.   
Let us write
$$c_m:=|1-\ba|^2.$$

\begin{proposition}\label{smoothconvontheedge}
For $\sigma\in(0,\sigma_0)$, let $r\in(0,\min\{r_\sigma,c_m/32 \})$ and $x_0\in\Gamma\setminus B_\sigma(T\cup{\rm spt}\,\mu)$. 
For $h$ large enough,  there exist a half space $H_h$ and $k\in\{1,\ldots,m-1\}$ such that $u_h=:(\chi_{H_h}+{\bf a}^k\chi_{H_h^c})w_h$ with $w_h\in W^{1,2}(B_{r}(x_0))$, and $w_h$ minimizes $E_{\eps_h}(\cdot,B_{r}(x_0))$ under its own boundary conditions.
In addition, $w_h\in C^{\infty}(B_{r}(x_0))$, $H_h\to H$ for some half space $H$ satisfying $\partial H\cap B_{r}(x_0)=\Gamma\cap B_{r}(x_0)$, and $w_h\to (\chi_{H}+{\bf a}^{-k}\chi_{H^c})u$ in $C^{\ell}_{\rm loc}(B_{r}(x_0))$ for every $\ell\in\mathbb{N}$.  
\end{proposition}

\begin{proof}
{\it Step 1.} Once again we may assume that $x_0=0$. We follow the strategy used in the proof of  Lemma~\ref{LemfarS} considering 
\[
  \gamma_h:= \int_{B_{5r}}|\nabla v_h- \nabla v_\mu|^2+ |v_h-v_\mu|^2\,dx+ \frac{1}{2\eh^2} \int_{B_{5r}}(1-|u_h|^2)^2\,dx\mathop{\longrightarrow}\limits_{h\to\infty}0\,. 
 \]
By Theorem \ref{Gammasharp} and Corollary \ref{corcvminsharp} we have \eqref{strongconvergencebisshp}, and thus   $\mathcal{H}^1(J_{u_h}\cap B_{5r})\to \mathcal{H}^1(\Gamma\cap B_{5r})=10r$. As a consequence,  
$ \mathcal{H}^1(J_{u_h}\cap B_{5r})\leq 11r$ for $h$ large enough, which in turn leads to  
\begin{equation}\label{smallmass2new}
\int_{r}^{5r} \mathcal{H}^0(J_{u_h}\cap \partial B_t)\,dt \leq  \mathcal{H}^1(J_{u_h}\cap B_{5r})\leq 11r\,.
\end{equation}
Setting 
$$A_h:=A_h^0\cup A_h^1\cup A_h^2\quad\text{with}\quad A^k_h:=\Big\{t\in (r,5r)   :  \mathcal{H}^0(J_{u_h}\cap \partial B_{t})=k\Big\}\,,$$
we infer from \eqref{smallmass2new} that $\mathcal{H}^1(A_h)\geq r/3$ for $h$ large enough. Notice that $u_h \in W^{1,2}(\partial B_t)$ for a.e. $t\in A^0_h$, and $u_h \in SBV^{2}(\partial B_t)$ for a.e. $t\in A^1_h\cup A^2_h$.

We claim that $\mathcal{H}^1(A_h^0)\leq r/6$ for $h$ large enough. Indeed, assume by contradiction that $\mathcal{H}^1(A_h^0)\geq r/6$ for some subsequence. Then, we could apply the proof of Lemma~\ref{LemfarS} (choosing a good radius $\rho_h\in A^0_h$) to infer that $u_h$ is smooth in $B_r$ for $h$ large enough, and thus that $J_{u_h}\cap B_r=\emptyset$. 
However  \eqref{strongconvergencebisshp} tells us that $\mathcal{H}^1(J_{u_h}\cap B_r)\to 2r$ as $h\to \infty$, a contradiction.
We have thus proved that $\mathcal{H}^1(A_h^1\cup A_h^2)\geq r/6$ for $h$ large enough. Now we claim that for $h$ even larger, we have $\mathcal{H}^1(A_h^1)\leq r/12$. 
By contradiction again, assume that $\mathcal{H}^1(A_h^1)\geq r/12$ for some subsequence. Then, we can find a good radius $\rho_h\in A_h^1$ such that $u_h\in SBV^2(\partial B_{\rho_h})$ and 
\begin{equation}\label{blablacpacrot}
\int_{\partial B_{\rho_h}}|\nabla v_h-\nabla v_\mu|^2+|v_h-v_\mu|^2+\frac{1}{2\eh^2} (1-|u_h|^2)^2\,d\mathcal{H}^1\leq \frac{12\gamma_h}{r}\,.
\end{equation}
By our choice of $\rho_h$, there is a single point $x_h\in \partial B_{\rho_h}$ such that $u_h\in W^{1,2}(\partial B_{\rho_h}\setminus\{ x_h\})$.
Rescaling variables if necessary, we may assume without too much loss of generality that the radius $\rho_h=\rho$ is independent of $h$.  By \eqref{blablacpacrot}, $v_h\to v_\mu$
uniformly on $\partial B_\rho$. As a consequence, $|u_h|=|v_h|\geq 1/2$ on $\partial B_\rho$, and ${\rm deg}(v_h,\partial B_\rho)=0$ for $h$ large enough. 
In particular, we can write $u_h=|u_h|e^{i\varphi_h}$ on $\partial B_{\rho_b}\setminus\{ x_h\}$ for some $\varphi_h\in W^{1,2}(\partial B_{\rho_b}\setminus\{ x_h\})$.
Let $\varphi_h^\pm$ be the traces of $\varphi_h$ at $ x_h$. Since $v_h=|u_h|e^{im\varphi_h}\in W^{1,2}(\partial B_\rho)$, we have $m(\varphi_h^+-\varphi_h^-)=2\pi {\rm deg}(v_h,\partial B_\rho)=0$. 
Hence $\varphi_h^+=\varphi_h^-$, which yields $\varphi_h\in W^{1,2}(\partial B_\rho)$. 
We obtain that  $u_h\in W^{1,2}(\partial B_\rho)$  contradicting our choice $\rho\in A^1_h$.  

We may now assume that $h$ is sufficiently large so that $\mathcal{H}^1(A_h^2)\geq r/12$. Arguing as in the proof of Lemma~\ref{LemfarS}, we select a good  radius $\rho_h\in A_h^2$
so that $u_h\in SBV^2(\partial B_{\rho_h})$ and \eqref{condbdrhohcpa} holds  together with \eqref{blablacpacrot}. Here again we shall assume for simplicity $\rho_h=\rho$ 
is independent of $h$ (otherwise we rescale variables as in the proof of Lemma~\ref{LemfarS}).  We write $\{ x_h, y_h\}:= J_{u_h}\cap \partial B_\rho$, and then $\mathcal{C}^h_1$ and $\mathcal{C}_2^h$ the two (open) arcs  in $\partial B_\rho$ joining $ x_h$ and $ y_h$. 
As above, we infer from \eqref{blablacpacrot} that $|u_h|=|v_h|\geq 1/2$ on $\partial B_\rho$, and ${\rm deg}(v_h,\partial B_\rho)=0$ for $h$ large enough. Since $u_h\in W^{1,2}(\mathcal{C}^h_j)$ for $j=1,2$, 
we deduce that there exist $\varphi_h^j\in W^{1,2}(\mathcal{C}^h_j)$ such that $u_h=|u_h|e^{i\varphi_h^j}$ on $\mathcal{C}^h_j$. Denote by $\varphi_h^{j,1}$ the trace of $\varphi_h^{j}$ at $ x_h$, and $\varphi_h^{j,2}$ the trace of $\varphi_h^{j}$ at $ y_h$. Since $v_h\in W^{1,2}(\partial B_\rho)$, and $v_h=|u_h|e^{im\varphi_h^j}$ on $\mathcal{C}^h_j$,
we obtain the relation  
$$m(\varphi^{2,1}_h-\varphi_h^{1,1})=2\pi k_1^h \quad\text{and}\quad m(\varphi^{2,2}_h-\varphi_h^{1,2})=2\pi k_2^h$$
for some $k_1^h,k_2^h\in\mathbb{Z}\setminus\{0\}$.  Define $k_h\in\{1,\ldots,m-1\}$ to be such that ${\bf a}^{k_h}=e^{2i\pi k_2^h/m}$, and consider the map 
$$g_h:=\begin{cases} 
u_h & \text{in $ \mathcal{C}^h_1$}\,,\\
{\bf a}^{-k_h} u_h & \text{in $ \mathcal{C}^h_2$}\,. 
\end{cases}
$$ 
By construction we have $\p(g_h)=v_h$ and $g_h\in W^{1,2}(\partial B_\rho\setminus \{ x_h\})$. However, since ${\rm deg}(v_h,\partial B_\rho)=0$, we can argue as above (when proving that $\mathcal{H}^1(A^1_h)\leq r/12$) to show that $g_h\in W^{1,2}(\partial B_\rho)$. In addition,  \eqref{condbdrhohcpa}  yields 
\begin{equation}\label{rhoghcplukoialafin}
 \rho^{1/2}\lt(\int_{\partial B_\rho} |\partial_\tau g_h|^2+\frac{(1-|g_h|^2)^2}{2\eps_h^2}\,d\mathcal{H}^1\rt)^{1/2}\leq \delta\,.
 \end{equation}
We also notice that $\|g_h\|_{L^\infty(B_\rho)}\leq 1$ since $|u_h|\leq 1$. 
\vskip5pt

\noindent{\it Step 2.} Define $H_h$ to be the half space containing $\mathcal{C}^h_1$ and such that $x_h, y_h\in \partial H_h$. We claim that $k_h=k\in\{1,\ldots,m-1\}$ is independent of $h$ for $h$ sufficiently large, that $H_h\to H$ for some half space $H$
such that $\Gamma\cap \partial B_\rho= \partial H\cap \partial B_\rho$, and that $g_h\to g_\star$ uniformly on $\partial B_\rho$ where $g_\star\in W^{1,2}(\partial B_\rho;\Ss^1)$ is given by 
\begin{equation}\label{defg*jump1654}
g_\star:=\begin{cases}
u & \text{in $ \partial B_\rho \cap H$}\,,\\
{\bf a}^{-k}u& \text{in $ \partial B_\rho \setminus H$}\,.
\end{cases}
\end{equation}
First observe that $\liminf_h| x_h- y_h|>0$. Indeed, if for some subsequence we have $| x_h- y_h|\to 0$, then either $\chi_{H_h}\to 0$ in $L^1(\partial B_\rho)$ or $\chi_{H^c_h}\to 0$ in $L^1(\partial B_\rho)$.
Assume that $\chi_{H^c_h}\to 0$ in $L^1(\partial B_\rho)$ (the other case being analogous). From Proposition \ref{smoothcvawaygamma}, we infer that $u_h\to u$ in $L^1(\partial B_\rho)$, so that $g_h\to u$ 
in $L^1(\partial B_\rho)$. In view of \eqref{rhoghcplukoialafin}, we deduce that $u$ belongs to $W^{1,2}(\partial B_\rho)$, a contradiction.
Next, by Proposition~\ref{smoothcvawaygamma} again, $u_h\to u$ in $C^0_{\rm loc}(\partial B_\rho\setminus\Gamma)$, which now implies that $\{ x_h, y_h\}\to \Gamma\cap \partial B_\rho$ as $h\to\infty$. Writing $\Gamma\cap B_\rho =:\{ x_\star, y_\star\}$,
we may assume that $ x_h\to x_\star$ and $ y_h\to y_\star$.  In the same way, we may assume that $\mathcal{C}^h_1\to \mathcal{C}_1$ where $\mathcal{C}_1$ is an arc of $\partial B_\rho$ joining $x_\star$ and $y_\star$. This clearly implies that $H_h\to H$
where $H$ is the half space containing $\mathcal{C}_1$ and such that $ x_\star, y_\star\in \partial H$. In view of Remark \ref{remsmoothLmumin}, there exists a unique $k\in \{1,\ldots,m-1\}$ such that the map  defined in \eqref{defg*jump1654} belongs to $W^{1,2}(\partial B_\rho)$. 
Combining this fact with \eqref{rhoghcplukoialafin} and the convergence of $u_h$ toward $u$ in $L^1(\partial B_\rho)$, we deduce that $g_h\to g_\star$ uniformly in $\partial B_\rho$, and that $k_h=k$ for $h$ large enough.  

Since $g_\star^m=u^m=v_\mu$ on $\partial B_\rho$, and ${\rm deg}(v_\mu,\partial B_\rho)=0$, we derive that ${\rm deg}(g_\star,\partial B_\rho)=0$. We can now apply Proposition~\ref{buildGLcompet} to produce minimizers 
$w_h$ of $E_{\eps_h}(\cdot,B_\rho)$ over $W^{1,2}_{g_h}(B_\rho)$. Then $w_h\to w_\star$ strongly in $W^{1,2}(B_1)$ where $w_\star$ is the unique solution of \eqref{pbharmmaplimtw}. 
Arguing as in the proof of Lemma~\ref{LemfarS} (Step 2), we obtain that $w_\star^m=v_\mu$ in $B_\rho$, which leads for $h$ large enough  to 
\begin{equation}\label{smallassumedge}
\int_{B_\rho}|\nabla w_h|^2\,dx\leq \delta\,.
\end{equation}
\vskip5pt

\noindent{\it Step 3.} Consider the competitor $\widehat u_h\in\mathcal{G}_g(\Omega)$ given by 
$$\widehat u_h:=\begin{cases}
u_h & \text{in $\Omega\setminus B_\rho$}\,,\\
(\chi_{H_h}+{\bf a}^k\chi_{H_h^c})w_h & \text{in $ B_\rho$}\,.
\end{cases}$$
By minimality we have $F^0_{\eps_h,g}(u_h)\leq F^0_{\eps_h,g}(\widehat u_h)$, and since $J_{u_h}\cap \partial B_{\rho}=\{ x_h, y_h\}$, we deduce that 
\begin{equation}\label{0748jedmat}
F^0_{\eps_h}(u_h,B_{\rho})\leq F^0_{\eps_h}(\widehat u_h,B_{\rho}) =E_{\eps_h}(w_h,B_{\rho}) + L_h\,,
\end{equation}
where $L_h:=| x_h-y_h|$. Since $\p(u_h)=\p(g_h)$ on $\partial B_\rho$, and in view of \eqref{rhoghcplukoialafin} and \eqref{smallassumedge} (and our choice of~$\delta$), we can apply Lemma \ref{LM} and Proposition \ref{propLM} to derive that 
\begin{equation}\label{0748jedmatlwd}
F^0_{\eps_h}(u_h,B_{\rho})\geq  E_{\eps_h}(w_h,B_{\rho}) + \frac{1}{8}\int_{B_\rho}|\nabla\phi_h|^2\,dx+\mathcal{H}^1(J_{\phi_h}\cap B_\rho)\,,
\end{equation}
where $\phi_h:=u_h/w_h$ satisfies $\phi_h=\chi_{H_h}+{\bf a}^k\chi_{H_h^c}$ on $\partial B_\rho$ (and thus $\p(\phi_h)=1$ on $\partial B_\rho$). Putting \eqref{0748jedmat} and \eqref{0748jedmatlwd} together leads to 
\begin{equation}\label{bordelnbcplu}
L_h\geq \frac{1}{8}\int_{B_\rho}|\nabla\phi_h|^2\,dx+\mathcal{H}^1(J_{\phi_h}\cap B_\rho)\,.
\end{equation}
Let us now prove that 
\begin{equation}\label{idmeasseg}
\mathcal{H}^1(J_{\phi_h}\cap B_\rho)=L_h\,.
\end{equation}
Up to a rotation, we assume that $ x_h=(a,t)$ and $ y_h=(b,t)$ with $b-a=L_h$. For $s\in(a,b)$, we write $V_s:=\{s\}\times\R$.  Now, assume by contradiction that 
$\mathcal{H}^1(J_{\phi_h}\cap B_\rho)=L_h-\gamma$ for some $\gamma>0$. Then we infer from the coarea formula   \cite[Theorem II.7.7]{Maggi}  that 
\begin{equation}\label{intermineq1202}
 L_h-\gamma=\mathcal{H}^1(J_{\phi_h}\cap B_\rho)\geq \int_a^b\mathcal{H}^0(J_{\phi_h}\cap B_\rho\cap V_s)\,ds\,.
 \end{equation}
Set $\widetilde A_h:=\big\{s\in(a,b): J_{\phi_h}\cap B_\rho\cap V_s=\emptyset\big\}$, and recall that $\phi_h\in W^{1,2}(B_\rho\cap V_s)$ for a.e. $s\in \widetilde A_h$. From \eqref{intermineq1202} we deduce that $\mathcal{H}^1(\widetilde A_h)\geq \gamma$.
Since, $\phi_h=\chi_{H_h}+{\bf a}^k\chi_{H_h^c}$ on $\partial B_\rho$, we have $\phi_h=1$ on $\mathcal{C}^h_1\cap V_s$ and $\phi_h={\bf a}^k$ on $\mathcal{C}^h_2\cap V_s$ for a.e. $s\in(a,b)$. Therefore, 
$$\int_{B_\rho\cap V_s} |\partial_\tau\phi_h|^2\,d\mathcal{H}^1\geq \frac{|1-{\bf a}^k|^2}{2\rho} \geq \frac{c_m}{10r}\quad \text{for a.e. $s\in(a,b)$}\,.$$
Integrating with respect to $s$ (and recalling that $r<c_m/80$) yields
\begin{multline*}
\frac{1}{8}\int_{B_\rho}|\nabla\phi_h|^2\,dx+\mathcal{H}^1(J_{\phi_h}\cap B_\rho)\geq \frac{1}{8}\int_{\widetilde A_h} \int_{B_\rho\cap V_s} |\partial_\tau\phi_h|^2\,d\mathcal{H}^1\,dx+L_h-\gamma\\
\geq L_h+\lt(\frac{c_m}{80r}-1\rt)\gamma > L_h\,,
\end{multline*}
which contradicts \eqref{bordelnbcplu}. 

By combining \eqref{bordelnbcplu} and \eqref{idmeasseg} we deduce that $|\nabla p(\phi_h)|\leq m|\nabla\phi_h|=0$ in $B_\rho$. Since $\p(\phi_h)=1$ on $\partial B_\rho$, we conclude that $\p(\phi_h)=1$ in $B_\rho$. In other words, $\phi_h$ takes values in $\mathbf{G}_m$. Hence, there is a Caccioppoli partition $\{E_j\}_{j=0}^{m-1}$ of $B_\rho$ such that 
$$\phi_h=\sum_{j=0}^{m-1}{\bf a}^j\chi_{E_j}\,.$$
Recalling \cite[Section 4.4]{AFP}, we have 
\begin{equation}\label{lengthberl1653}
L_h=\mathcal{H}^1(J_{\phi_h}\cap B_\rho)=\mathcal{H}^1(\partial E_0\cap B_\rho)+\frac{1}{2}\sum_{j,\ell=1,\,j\neq \ell}^{m-1}\mathcal{H}^1(\partial E_j\cap \partial E_\ell\cap B_\rho) \,.
\end{equation}
Using that $\chi_{E_0}=\chi_{H_h}$ on $\partial B_\rho$, we have that $\mathcal{H}^1(\partial E_0\cap B_\rho)\geq L_h$ with equality if and only if $E_0=H_h\cap B_\rho$. 
Therefore $E_0=H_h\cap B_\rho$ and the sum on the right-hand side of \eqref{lengthberl1653} vanishes. Since $\chi_{E_k}=\chi_{H_h^c}$ on $\partial B_\rho$, we conclude that
$E_k=H_h^c\cap B_\rho$ and $E_j=\emptyset$ for $j\not\in\{0,k\}$. In other words, $\phi_h=\chi_{H_h}+{\bf a}^k\chi_{H_h^c}$ in $B_\rho$, and thus $u_h=(\chi_{H_h}+{\bf a}^k\chi_{H_h^c})w_h$ in $B_\rho$.  

To conclude, we observe that $w_h=(\chi_{H_h}+{\bf a}^{-k}\chi_{H_h^c})u_h$. Since $u_h\to u$ in $L^1(\Omega)$ and $H_h\to H$, Proposition~\ref{buildGLcompet} tells us that 
$w_h\to (\chi_{H_\star}+{\bf a}^{-k}\chi_{H_\star^c})u$ in  $C^\ell_{\rm loc}(B_\rho)$ for every $\ell\in\mathbb{N}$.
\end{proof}

\begin{remark}
 In order to prove \eqref{idmeasseg} one could also use a calibration argument (see \cite{AlBouDal,Mora}). However since our proof is elementary, we have decided to keep it this way.
\end{remark}

\subsubsection{Smoothness and convergence near triple junctions}

We now focus on the behavior of $u_h$ near the points of $T$, i.e., triple junctions. It will be convenient to describe a triple junction in the following way.
First write for $j=0,1,2$, $Y^j_{\rm ref}:=\big\{z\in\C\setminus\{0\}: {\rm arg}(z)\in(2j\pi/3,2(j+1)\pi/3)\big\}$. We say that an ordered triplet of open sets $(Y^0,Y^1,Y^2)$
is a triple junction if there are $x_0\in \C$  and $\theta\in[0,2\pi)$ such that  $Y^j=x_0+e^{i\theta} Y^j_{\rm ref}$ for $j=0,1,2$. 
Then, we say that $x_0$ is the center of the triple junction $(Y^0,Y^1,Y^2)$. In the statement below, we understand the convergence of triple junctions in the sense of local Hausdorff convergence. 
   
\begin{proposition}\label{propmincloseTnew}
For $\sigma\in(0,\sigma_0)$, let $r\in(0,\min\{r_\sigma,c_m/128 \})$ and $x_0\in T$. For $h$ large enough, there exist a triple junction $(Y_h^0,Y_h^1,Y_h^2)$
and distinct $k_1,k_2\in\{1,\ldots,m-1\}$ such that $u_h=:(\chi_{Y_h^0}+{\bf a}^{k_1}\chi_{Y_h^1}+{\bf a}^{k_2}\chi_{Y_h^2}) w_h$ with $w_h\in W^{1,2}(B_r(x_0))$,
and $w_h$ minimizes $E_{\eps_h}(\cdot,B_r(x_0))$ under its own boundary conditions. In addition, $w_h\in C^\infty(B_r(x_0))$, $(Y_h^0,Y_h^1,Y_h^2)\to (Y^0,Y^1,Y^2)$ 
for some triple junction   satisfying $\cup_j\partial Y^j\cap B_r(x_0)=\Gamma \cap B_r(x_0)$, and $w_h\to (\chi_{Y_h^0}+{\bf a}^{-k_1}\chi_{Y_h^1}+{\bf a}^{-k_2}\chi_{Y_h^2})u$  in $C^\ell_{\rm loc}(B_r(x_0))$ for every $\ell\in \mathbb{N}$. 
\end{proposition}

\begin{proof}
{\it Step 1.} Without loss of generality, we may assume that $x_0=0$. From Remark \ref{remsmoothLmumin}, we infer that there exist a triple junction $(Y^0,Y^1,Y^2)$ centered in $0$ 
and distinct $k_1,k_2\in\{1,\ldots,m-1\}$ such that the map $(\chi_{Y^0}+{\bf a}^{-k_1}\chi_{Y^1}+{\bf a}^{-k_2}\chi_{Y^2})u$ is smooth in $B_{2r}$, and $\cup_j\partial Y^j\cap B_{2r}=\Gamma\cap B_{2r}$. Since the values of $k_1$ and $k_2$ play no role we will assume that $k_1=1$ and $k_2=2$ to keep notation simpler. 
We write $\{a\}:=\partial Y^2\cap \partial Y^0\cap\partial B_r$, $\{b\}:=\partial Y^0\cap \partial Y^1\cap\partial B_r$, and 
$\{c\}:=\partial Y^1\cap \partial Y^2\cap\partial B_r$.

Choosing a sufficiently small radius $0<\kappa<r/2$, we can apply  Proposition \ref{smoothconvontheedge} in the balls $B_{2\kappa}(a)$, $B_{2\kappa}(b)$, and $B_{2\kappa}(c)$, and infer that there exist half spaces $H_h^0$, $H_h^1$, and $H_h^2$ such that  $H_h^0\cap B_{2\kappa}(a)\to Y^0\cap B_{2\kappa}(a)$, $H_h^1\cap B_{2\kappa}(b)\to Y^1\cap B_{2\kappa}(b)$, $H_h^2\cap B_{2\kappa}(c)\to Y^2\cap B_{2\kappa}(c)$, and 
\begin{equation}\label{cvsidesY1}
(\chi_{H_h^0}+{\bf a}^{-2}\chi_{(H_h^0)^c})u_h\to  (\chi_{Y^0}+{\bf a}^{-2}\chi_{Y^2})u \quad\text{in $C^k_{\rm loc}(B_{2\kappa}(a))$}\,,
\end{equation}
\begin{equation}\label{cvsidesY2}
(\chi_{(H_h^1)^c}+{\bf a}^{-1}\chi_{H_h^1})u_h\to  (\chi_{Y^0}+{\bf a}^{-1}\chi_{Y^1})u \quad\text{in $C^k_{\rm loc}(B_{2\kappa}(b))$}\,,
\end{equation}
\begin{equation}\label{cvsidesY3}
({\bf a}^{-1}\chi_{(H_h^2)^c}+{\bf a}^{-2}\chi_{H_h^2})u_h\to  ({\bf a}^{-1}\chi_{Y^1}+{\bf a}^{-2}\chi_{Y^2})u \quad\text{in $C^k_{\rm loc}(B_{2\kappa}(c))$}\,.
\end{equation}
In view of Proposition  \ref{smoothcvawaygamma}, we deduce that for $h$ large enough, $J_{u_h}\cap(B_{r+\kappa}\setminus\overline B_{r-\kappa})$
is made of three (disjoint) segments, each of them  intersecting $\partial B_t$ almost orthogonally (in particular at a single point) for every $t\in(r-\kappa,r+\kappa)$.
As a consequence, for $h$ large enough the open set $(B_{r+\kappa}\setminus\overline B_{r-\kappa})\setminus J_{u_h}$ has three connected components $Z^0_h$, $Z_h^1$, and $Z_h^2$
satisfying  $Z_h^j\to (B_{r+\kappa}\setminus\overline B_{r-\kappa})\cap Y^j$. Combining \eqref{cvsidesY1}-\eqref{cvsidesY2}-\eqref{cvsidesY3} with Proposition \ref{smoothconvontheedge}, we derive that 
$$(\chi_{Z_h^0}+{\bf a}^{-1}\chi_{Z_h^1}+{\bf a}^{-2}\chi_{Z_h^2})u_h\to (\chi_{Y^0}+{\bf a}^{-1}\chi_{Y^1}+{\bf a}^{-2}\chi_{Y^2})u \quad\text{in $C^k_{\rm loc}(B_{r+\kappa}\setminus \overline B_{r-\kappa})$}\,.$$
\vskip5pt

\noindent{\it Step 2.} Arguing as the proof of Lemma \ref{LemfarS} (Step 1), we find a good radius $\rho_h\in(r,r+\kappa/2)$ such that \eqref{condbdrhohcpa} holds (for $h$ even larger). 
Rescaling variables if necessary, we may assume without too much loss of generality that $\rho_h=\rho$ is independent of $h$. To simplify, we will further assume that actually $\rho=r$. Setting
$$g_h:= (\chi_{Z_h^0}+{\bf a}^{-1}\chi_{Z_h^1}+{\bf a}^{-2}\chi_{Z_h^2})u_h \in C^\infty(\partial B_r)\,,$$
estimate \eqref{rhoghcplukoialafin} holds, $\|g_h\|_{L^\infty(B_r)}\leq 1$, and $g_h\to g_\star:=(\chi_{Y^0}+{\bf a}^{-1}\chi_{Y^1}+{\bf a}^{-2}\chi_{Y^2})u$ uniformly on $\partial B_r$.
Once again, since $g_\star^m=v_\mu$ we have ${\rm deg}(g_\star,\partial B_r)=0$. Then, we apply Proposition~\ref{buildGLcompet} to produce minimizers $w_h$ of $E_{\eps_h}(\cdot,B_r)$ over $W^{1,2}_{g_h}(B_r)$, and  $w_h\to w_\star$ strongly in $W^{1,2}(B_r)$
where $w_\star$ is the unique solution of \eqref{pbharmmaplimtw}.  
Again, as in the proof of Lemma~\ref{LemfarS} (Step 2), we obtain that $w_\star^m=v_\mu$ in $B_r$, which leads to \eqref{smallassumedge} 
for $h$ large enough. 
\vskip5pt

\noindent{\it Step 3.} By Step 1, we have $J_{u_h}\cap\partial B_r=\{ x_h, y_h, z_h\}$ for $h$ large enough, with $ x_h\to a$, $ y_h\to b$, and $ z_h\to c$. 
For $h$ large enough, we can then find a triple junction $(Y^0_h,Y^1_h,Y^2_h)$ (which might not be centered at the origin) such that 
$\{ x_h\}=\partial Y_h^2\cap \partial Y_h^0\cap\partial B_r$, $\{ y_h\}=\partial Y_h^0\cap \partial Y_h^1\cap\partial B_r$, and 
$\{ z_h\}:=\partial Y_h^1\cap \partial Y_h^2\cap\partial B_r$. Obviously, $Y^j_h\to Y^j$ as $h\to \infty$.
Notice also that $g_h=(\chi_{Y_h^0}+{\bf a}^{-1}\chi_{Y_h^1}+{\bf a}^{-2}\chi_{Y_h^2})u_h$ on $\partial B_r$. 

Next, we consider the competitor $\widehat u_h\in\mathcal{G}_g(\Omega)$ given by 
$$\widehat u_h:=\begin{cases} 
u_h & \text{in $ \Omega\setminus B_r$}\,,\\
(\chi_{Y_h^0}+{\bf a}\chi_{Y_h^1}+{\bf a}^{2}\chi_{Y_h^2})w_h & \text{in $ B_r$}\,.
\end{cases}$$
By minimality we have $F^0_{\eps_h,g}(u_h)\leq F^0_{\eps_h,g}(\widehat u_h)$, and since $J_{u_h}\cap \partial B_{r}=\{ x_h, y_h, z_h\}$, we deduce that 
\begin{equation}\label{0748jedmattripl}
F^0_{\eps_h}(u_h,B_{r})\leq F^0_{\eps_h}(\widehat u_h,B_{r}) =E_{\eps_h}(w_h,B_{r}) + \mathcal{H}^1({\bf Y}_h\cap B_r)\,,
\end{equation}
where we have set ${\bf Y}_h:=\cup_j\partial Y_h^j$. Once again $\p(u_h)=\p(g_h)$ on $\partial B_r$, and by \eqref{rhoghcplukoialafin} and \eqref{smallassumedge}, we can apply Lemma \ref{LM} and Proposition \ref{propLM} to derive that \eqref{0748jedmatlwd} holds,  
where $\phi_h:=u_h/w_h$ satisfies $\phi_h=\chi_{Y_h^0}+{\bf a}^{k_1}\chi_{Y_h^1}+{\bf a}^{k_2}\chi_{Y_h^2}$ on $\partial B_r$ (and $\p(\phi_h)=1$ on $\partial B_r$). Combining  \eqref{0748jedmatlwd}  with \eqref{0748jedmattripl} leads to 
\begin{equation}\label{bordelnbcplutripl}
\mathcal{H}^1({\bf Y}_h\cap B_r) 
\geq \frac{1}{8}\int_{B_r}|\nabla\phi_h|^2\,dx+\mathcal{H}^1(J_{\phi_h}\cap B_r)\,.
\end{equation}
Our choice of $r$ (small compare to $c_m$) allows us to use the calibration in \cite[Example 5.4]{Mora} (with $\alpha=16$) to deduce that for $h$ large enough the map 
$\chi_{Y_h^0}+{\bf a}\chi_{Y_h^1}+{\bf a}^{2}\chi_{Y_h^2}$ is a Dirichlet minimizer of the Mumford-Shah functional
\footnote{Even though the calibrations defined in~\cite{Mora} (see also \cite{AlBouDal}) are given for centered triple junctions, we can consider  restrictions to $B_r$ of calibrations
defined on a larger ball centered at the center of ${\bf Y_h}$.}  \cite[Definition 3.1)]{Mora}. 
As a consequence, 
\begin{equation}\label{lundmidcalib}
\frac{1}{16}\int_{B_r}|\nabla\phi_h|^2\,dx+\mathcal{H}^1(J_{\phi_h}\cap B_r)\geq \mathcal{H}^1({\bf Y}_h\cap B_r)\,.
\end{equation}
Putting together \eqref{bordelnbcplutripl} and \eqref{lundmidcalib} yields
$$\int_{B_r}|\nabla\phi_h|^2\,dx=0\quad \text{and} \quad \mathcal{H}^1(J_{\phi_h}\cap B_r)=\mathcal{H}^1({\bf Y}_h\cap B_r)\,. $$
Arguing as in the proof of Proposition \ref{smoothconvontheedge} (Step 3), we deduce that 
$$\phi_h=\sum_{k=0}^{m-1}{\bf a}^k\chi_{E_k} $$
for a  Caccioppoli partition $\{E_k\}_{k=0}^{m-1}$ of $B_r$ satisfying 
\begin{equation}\label{mard1134}
E_0\cap (B_r\setminus \overline B_{r-\kappa})=Z_0^h\cap B_r\,, \;E_{1}\cap (B_r\setminus \overline B_{r-\kappa})=Z_1^h\cap B_r\,,\; E_{2}\cap (B_r\setminus \overline B_{r-\kappa})=Z_2^h\cap B_r\,,
\end{equation}
and $E_k\subset B_{r-\kappa}$ for $k\not\in\{0, 1, 2\}$. 

Let us now consider an arbitrary Caccioppoli partition $\{F_k\}_{k=0}^{m-1}$ of $B_r$ such that each $F_k\triangle E_k$ is compactly contained in $B_r$, and define the competitor $\widetilde u_h\in\mathcal{G}_g(\Omega)$ by 
$$\widetilde u_h:=\begin{cases} 
u_h & \text{in $\Omega\setminus B_r$}\,,\\
\phi w_h & \text{in $B_r$}\,,
\end{cases}\quad\text{with}\quad \phi:=\sum_{k=0}^{m-1}{\bf a}^k\chi_{F_k} \,.$$
By minimality  $F^0_{\eps_h,g}(u_h)\leq F^0_{\eps_h,g}(\widetilde u_h)$, which leads as before to
$$ \mathcal{H}^1(J_{\phi_h}\cap B_r)\leq \mathcal{H}^1(J_{\phi}\cap B_r)\,.$$
As in the proof of Theorem \ref{thmLDmu} (Step 3), it implies that $\{E_k\}_{k=0}^{m-1}$ is a minimal partition of $B_r$, so that $J_{\phi_h}\cap B_r=\cup_k\partial E_k\cap B_r$ is locally a finite union of  segments 
(see  \cite[Theorem~5.2]{CLM}). Since we already know that $J_{\phi_h}$ is made of three segments in a neighborhood of $\partial B_r$, we conclude that 
$\cup_k\partial E_k\cap B_r$ is  made of finitely many segments with $\cup_k\partial E_k\cap \partial B_r=\{ x_h, y_h, z_h\}$\footnote{Here and in the rest of the proof, by an abuse of notation we identify $\partial E_k\cap  B_r$ and $\overline{\partial E_k\cap {B}_r}$. } . 
In view of \eqref{mard1134}, we have $\{ x_h, y_h\}\subset\partial E_0\cap B_r$, and the connected component of  $\partial E_0\cap B_r$ containing $ x_h$ is a polygonal curve joining $ x_h$ to $ y_h$.  
Similarly, $\partial E_{2}\cap B_r$ contains a polygonal curve connecting
$ x_h$ to $ z_h$.  Set $\Gamma_h$ to be the union of these two curves. Then $\Gamma_h$ is a  connected set containing $\{ x_h, y_h, z_h\}$, and contained in ${J_{\phi_h}\cap B_r}$. 
Since ${\bf Y}_h\cap \overline B_r$ is the unique solution of  the Steiner problem relative to the points $\{ x_h, y_h, z_h\}$, we have 
$$\mathcal{H}^1({\bf Y}_h\cap B_r)= \mathcal{H}^1(J_{\phi_h}\cap B_r)\geq  \mathcal{H}^1(\Gamma_h)\geq \mathcal{H}^1({\bf Y}_h\cap B_r)\,,$$
and it follows that $J_{\phi_h}\cap B_r=\Gamma_h\cap B_r={\bf Y}_h\cap B_r$. From \eqref{mard1134} we conclude that 
$$\phi_h=\chi_{Y_h^0}+{\bf a}\chi_{Y_h^1}+{\bf a}^{2}\chi_{Y_h^2} \,,$$
that is $u_h=(\chi_{Y_h^0}+{\bf a}\chi_{Y_h^1}+{\bf a}^2\chi_{Y_h^2})w_h$ in $B_r$. 

Since $u_h\to u$ in $L^1(\Omega)$ and $Y^j_h\to Y^j$, Proposition~\ref{buildGLcompet} implies that $w_h\to (\chi_{Y^0}+{\bf a}^{-1}\chi_{Y^1}+{\bf a}^{-2}\chi_{Y^2})u$ in $C^k_{\rm loc}(B_r)$ for every $k\in\mathbb{N}$, and the proof is complete. 
\end{proof}

\section*{Acknowledgments}
We thank R. Badal, M. Cicalese, L. De Luca, and M. Ponsiglione for telling us about their result \cite{BadCicLuPo} and giving us an early access to a preprint version. We also thank M. Dos Santos for pointing out the paper \cite{Radulescu}. The authors have been supported by the
Agence Nationale de la Recherche through the grants
ANR-12-BS01-0014-01 (Geometrya),  ANR-14-CE25-0009-01 (MAToS), and by the PGMO
research project COCA. BM was partially supported by the INRIA team RAPSODI and the Labex CEMPI (ANR-11-LABX-0007-01).

%=======================
% BIBLIOGRAPHY AND INDEX
%=======================

\bibliographystyle{alpha}
\bibliography{ripple.bib}

\begin{thebibliography}{ABDM03}

\bibitem[ABDM03]{AlBouDal}
G.~Alberti, G.~Bouchitt{\'e}, and G.~Dal~Maso.
\newblock The calibration method for the {M}umford-{S}hah functional and
  free-discontinuity problems.
\newblock {\em Calc. Var. Partial Differential Equations}, 16(3):299--333,
  2003.

\bibitem[AFP00]{AFP}
L.~Ambrosio, N.~Fusco, and D.~Pallara.
\newblock {\em Functions of bounded variation and free discontinuity problems}.
\newblock Oxford Mathematical Monographs. The Clarendon Press, Oxford
  University Press, New York, 2000.

\bibitem[AHL17]{AlperHL}
O.~Alper, R.~Hardt, and F.-H. Lin.
\newblock Defects of liquid crystals with variable degree of orientation.
\newblock {\em Calculus of Variations and Partial Differential Equations},
  56(5):128, 2017.

\bibitem[AP14]{AlicPon}
R.~Alicandro and M.~Ponsiglione.
\newblock Ginzburg-{L}andau functionals and renormalized energy: a revised
  {$\Gamma$}-convergence approach.
\newblock {\em J. Funct. Anal.}, 266(8):4890--4907, 2014.

\bibitem[AT90]{AmbTor90}
L.~Ambrosio and V.~M. Tortorelli.
\newblock Approximation of functionals depending on jumps by elliptic
  functionals via {$\Gamma$}-convergence.
\newblock {\em Comm. Pure Appl. Math.}, 43(8):999--1036, 1990.

\bibitem[AT92]{AmbTor}
L.~Ambrosio and V.~M. Tortorelli.
\newblock On the approximation of free discontinuity problems.
\newblock {\em Boll. Un. Mat. Ital. B (7)}, 6(1):105--123, 1992.

\bibitem[BBH93]{BBHarticle}
F.~Bethuel, H.~Brezis, and F.~H{\'e}lein.
\newblock Asymptotics for the minimization of a {G}inzburg-{L}andau functional.
\newblock {\em Calc. Var. Partial Differential Equations}, 1(2):123--148, 1993.

\bibitem[BBH94]{BBH}
F.~Bethuel, H.~Brezis, and F.~H{\'e}lein.
\newblock {\em Ginzburg-{L}andau vortices}.
\newblock Progress in Nonlinear Differential Equations and their Applications,
  13. Birkh\"auser Boston, Inc., Boston, MA, 1994.

\bibitem[BBM00]{BoBrMi}
J.~Bourgain, H.~Brezis, and P.~Mironescu.
\newblock Lifting in {S}obolev spaces.
\newblock {\em J. Anal. Math.}, 80:37--86, 2000.

\bibitem[BC84]{BrezCor}
H.~Brezis and J.-M. Coron.
\newblock Multiple solutions of {$H$}-systems and {R}ellich's conjecture.
\newblock {\em Comm. Pure Appl. Math.}, 37(2):149--187, 1984.

\bibitem[BCG14]{BelChGol}
G.~Bellettini, A.~Chambolle, and M.~Goldman.
\newblock The {$\Gamma$}-limit for singularly perturbed functionals of
  {P}erona-{M}alik type in arbitrary dimension.
\newblock {\em Math. Models Methods Appl. Sci.}, 24(6):1091--1113, 2014.

\bibitem[BCL86]{BreCorLieb}
H.~Brezis, J.-M. Coron, and E.~H. Lieb.
\newblock Harmonic maps with defects.
\newblock {\em Comm. Math. Phys.}, 107(4):649--705, 1986.

\bibitem[BCLP16]{BadCicLuPo}
R.~Badal, M.~Cicalese, L.~De Luca, and M.~Ponsiglione.
\newblock {$\Gamma$}-convergence analysis of a generalized ${XY}$ model:
  fractional vortices and string defects, 2016.

\bibitem[BCP96]{BraiCP}
A.~Braides and V.~Chiad{\`o}~Piat.
\newblock Integral representation results for functionals defined on {${\rm
  SBV}(\Omega;{\bf R}^m)$}.
\newblock {\em J. Math. Pures Appl. (9)}, 75(6):595--626, 1996.

\bibitem[BCS07]{BCS}
A.~Braides, A.~Chambolle, and M.~Solci.
\newblock A relaxation result for energies defined on pairs set-function and
  applications.
\newblock {\em ESAIM Control Optim. Calc. Var.}, 13(4):717--734 (electronic),
  2007.

\bibitem[Bed16]{Bedford}
S.~Bedford.
\newblock Function spaces for liquid crystals.
\newblock {\em Arch. Ration. Mech. Anal.}, 219(2):937--984, 2016.

\bibitem[BFL91]{belfeiglev}
M.~L. Belaya, M.~V. Feigel'Man, and V.~G. Levadny.
\newblock Theory of the ripple phase coexistance.
\newblock {\em Journal de Physique II}, 1(3):375--380, 1991.

\bibitem[BL14]{BucLuc}
D.~Bucur and S.~Luckhaus.
\newblock Monotonicity formula and regularity for general free discontinuity
  problems.
\newblock {\em Arch. Ration. Mech. Anal.}, 211(2):489--511, 2014.

\bibitem[BMP05]{BMP}
H.~Brezis, P.~Mironescu, and A.~C. Ponce.
\newblock {$W^{1,1}$}-maps with values into {$S^1$}.
\newblock In {\em Geometric analysis of {PDE} and several complex variables},
  volume 368 of {\em Contemp. Math.}, pages 69--100. Amer. Math. Soc.,
  Providence, RI, 2005.

\bibitem[Bra98]{braidesfree}
A.~Braides.
\newblock {\em Approximation of free-discontinuity problems}, volume 1694 of
  {\em Lecture Notes in Mathematics}.
\newblock Springer-Verlag, Berlin, 1998.

\bibitem[Bra02]{Braides}
A.~Braides.
\newblock {\em {$\Gamma$}-convergence for beginners}, volume~22 of {\em Oxford
  Lecture Series in Mathematics and its Applications}.
\newblock Oxford University Press, Oxford, 2002.

\bibitem[BZ11]{BallZar}
J.~M. Ball and A.~Zarnescu.
\newblock Orientability and energy minimization in liquid crystal models.
\newblock {\em Arch. Ration. Mech. Anal.}, 202(2):493--535, 2011.

\bibitem[CJ99]{ColJer}
J.~E. Colliander and R.~L. Jerrard.
\newblock Ginzburg-{L}andau vortices: weak stability and {S}chr\"odinger
  equation dynamics.
\newblock {\em J. Anal. Math.}, 77:129--205, 1999.

\bibitem[CLM15]{CLM}
M.~Cicalese, G.~P. Leonardi, and F.~Maggi.
\newblock Improved convergence theorems for bubble clusters. i. the planar
  case, 2015.

\bibitem[CT99]{CorToa}
G.~Cortesani and R.~Toader.
\newblock A density result in {SBV} with respect to non-isotropic energies.
\newblock {\em Nonlinear Anal.}, 38(5, Ser. B: Real World Appl.):585--604,
  1999.

\bibitem[Dav05]{David}
G.~David.
\newblock {\em Singular sets of minimizers for the {M}umford-{S}hah
  functional}, volume 233 of {\em Progress in Mathematics}.
\newblock Birkh\"auser Verlag, Basel, 2005.

\bibitem[Dem90]{Demengel}
F.~Demengel.
\newblock Une caract\'erisation des applications de {$W^{1,p}(B^N,S^1)$} qui
  peuvent \^etre approch\'ees par des fonctions r\'eguli\`eres.
\newblock {\em C. R. Acad. Sci. Paris S\'er. I Math.}, 310(7):553--557, 1990.

\bibitem[DHW87]{DHW}
D.~Z. Du, F.~K. Hwang, and J.~F. Weng.
\newblock Steiner minimal trees for regular polygons.
\newblock {\em Discrete {\&} Computational Geometry}, 2(1):65--84, 1987.

\bibitem[DI03]{IgnDav}
J.~D{\'a}vila and R.~Ignat.
\newblock Lifting of {BV} functions with values in {$S^1$}.
\newblock {\em C. R. Math. Acad. Sci. Paris}, 337(3):159--164, 2003.

\bibitem[DM93]{DalMaso}
G.~Dal~Maso.
\newblock {\em An introduction to {$\Gamma$}-convergence}.
\newblock Progress in Nonlinear Differential Equations and their Applications,
  8. Birkh\"auser Boston, Inc., Boston, MA, 1993.

\bibitem[DPFP17]{PratFusDep}
G.~De~Philippis, N.~Fusco, and A.~Pratelli.
\newblock On the approximation of {SBV} functions.
\newblock {\em Atti Accad. Naz. Lincei Rend. Lincei Mat. Appl.},
  28(2):369--413, 2017.

\bibitem[FM13]{FarMir}
A.~Farina and P.~Mironescu.
\newblock Uniqueness of vortexless {G}inzburg-{L}andau type minimizers in two
  dimensions.
\newblock {\em Calc. Var. Partial Differential Equations}, 46(3-4):523--554,
  2013.

\bibitem[Fus03]{fuscoreview}
N.~Fusco.
\newblock An overview of the {M}umford-{S}hah problem.
\newblock {\em Milan J. Math.}, 71:95--119, 2003.

\bibitem[GMS79]{GMS}
M.~Giaquinta, G.~Modica, and J.~Sou\v{c}ek.
\newblock Functionals with linear growth in the calculus of variations. {I},
  {II}.
\newblock {\em Comment. Math. Univ. Carolin.}, 20(1):143--156, 157--172, 1979.

\bibitem[GP68]{GilPol}
E.~N. Gilbert and H.~O. Pollak.
\newblock Steiner minimal trees.
\newblock {\em SIAM J. Appl. Math.}, 16:1--29, 1968.

\bibitem[GS92]{Gromov}
M.~Gromov and R.~Schoen.
\newblock Harmonic maps into singular spaces and {$p$}-adic superrigidity for
  lattices in groups of rank one.
\newblock {\em Inst. Hautes \'Etudes Sci. Publ. Math.}, (76):165--246, 1992.

\bibitem[HL93]{HardtLin}
R.~Hardt and F.-H. Lin.
\newblock Harmonic maps into round cones and singularities of nematic liquid
  crystals.
\newblock {\em Math. Z.}, 213(4):575--593, 1993.

\bibitem[IL17]{IL2017}
R.~{Ignat} and X.~{Lamy}.
\newblock {L}ifting of $\mathbb{RP}^{d-1}$-valued maps in $bv$. {A}pplications
  to uniaxial $q$-tensors.
\newblock {\em ArXiv e-prints}, June 2017.

\bibitem[ILR00]{Radulescu}
L.~Ignat, C.~Lefter, and V.~D. Radulescu.
\newblock Minimization of the renormalized energy in the unit ball of {$\bold
  R^2$}.
\newblock {\em Nieuw Arch. Wiskd. (5)}, 1(3):278--280, 2000.

\bibitem[JS02]{JerSon}
R.~L. Jerrard and H.~M. Soner.
\newblock The {J}acobian and the {G}inzburg-{L}andau energy.
\newblock {\em Calc. Var. Partial Differential Equations}, 14(2):151--191,
  2002.

\bibitem[Lem16]{lemenant}
A.~Lemenant.
\newblock A selective review on {M}umford-{S}hah minimizers.
\newblock {\em Boll. Unione Mat. Ital.}, 9(1):69--113, 2016.

\bibitem[Lin89]{Lindefects}
F.-H. Lin.
\newblock Nonlinear theory of defects in nematic liquid crystals; phase
  transition and flow phenomena.
\newblock {\em Comm. Pure Appl. Math.}, 42(6):789--814, 1989.

\bibitem[LM93]{LubMkintosh}
T.~C. Lubensky and F.~C. MacKintosh.
\newblock Theory of ``ripple'' phases of lipid bilayers.
\newblock {\em Phys. Rev. Lett.}, 71:1565--1568, Sep 1993.

\bibitem[LM99]{LasMi}
L.~Lassoued and P.~Mironescu.
\newblock Ginzburg-{L}andau type energy with discontinuous constraint.
\newblock {\em J. Anal. Math.}, 77:1--26, 1999.

\bibitem[LS07]{LenzSchmidt}
O.~Lenz and F.~Schmid.
\newblock Structure of symmetric and asymmetric ``ripple'' phases in lipid
  bilayers.
\newblock {\em Phys. Rev. Lett.}, 98:058104, Jan 2007.

\bibitem[LX99]{XinLin}
F.-H. Lin and J.~X. Xin.
\newblock On the incompressible fluid limit and the vortex motion law of the
  nonlinear {S}chr\"odinger equation.
\newblock {\em Comm. Math. Phys.}, 200(2):249--274, 1999.

\bibitem[Mag12]{Maggi}
F.~Maggi.
\newblock {\em Sets of finite perimeter and geometric variational problems},
  volume 135 of {\em Cambridge Studies in Advanced Mathematics}.
\newblock Cambridge University Press, Cambridge, 2012.
\newblock An introduction to geometric measure theory.

\bibitem[Mer06]{Merletlift}
B.~Merlet.
\newblock Two remarks on liftings of maps with values into {$S^1$}.
\newblock {\em C. R. Math. Acad. Sci. Paris}, 343(7):467--472, 2006.

\bibitem[Mor94]{MorMed}
F.~Morgan.
\newblock $({M}, \varepsilon, \delta)$-minimal curve regularity.
\newblock {\em Proceedings of the American Mathematical Society},
  120(3):677--686, 1994.

\bibitem[Mor02]{Mora}
M.~G. Mora.
\newblock The calibration method for free-discontinuity problems on
  vector-valued maps.
\newblock {\em J. Convex Anal.}, 9(1):1--29, 2002.

\bibitem[Pol78]{pollak}
H.~O. Pollak.
\newblock Some remarks on the {S}teiner problem.
\newblock {\em J. Combinatorial Theory Ser. A}, 24(3):278--295, 1978.

\bibitem[RS83]{ruppelsackman}
D.~Ruppel and E.~Sackmann.
\newblock On defects in different phases of two-dimensional lipid bilayers.
\newblock {\em J. Phys. France}, 44(9):1025--1034, 1983.

\bibitem[Sac95]{Sackmann}
E.~Sackmann.
\newblock Physical basis of self-organization and function of membranes:
  physics of vesicles.
\newblock {\em Handbook of Biological Physics}, 1:213--304, 1995.

\bibitem[SS04]{SS}
S.~Serfaty and E.~Sandier.
\newblock Vortices for {G}inzburg-{L}andau equations: with magnetic field
  versus without.
\newblock In {\em Noncompact problems at the intersection of geometry,
  analysis, and topology}, volume 350 of {\em Contemp. Math.}, pages 233--244.
  Amer. Math. Soc., Providence, RI, 2004.

\bibitem[Wen69]{Wente}
H.~C. Wente.
\newblock An existence theorem for surfaces of constant mean curvature.
\newblock {\em J. Math. Anal. Appl.}, 26:318--344, 1969.

\bibitem[YZ96]{Yezu}
D.~Ye and F.~Zhou.
\newblock Uniqueness of solutions of the {G}inzburg-{L}andau problem.
\newblock {\em Nonlinear Anal.}, 26(3):603--612, 1996.

\end{thebibliography}

\end{document}